\documentclass{amsart}
\usepackage[arrow,matrix,curve,cmtip,ps]{xy}
\usepackage{pb-diagram}
\usepackage{pb-xy}
\usepackage{amsmath,amsthm,amssymb}
\usepackage{graphicx}
\usepackage{amssymb}
\usepackage{nicefrac}
\usepackage{enumerate}
\usepackage{float}
\usepackage{tabularx}
\usepackage{hyperref}
\usepackage{tikz}
\usetikzlibrary{patterns}
\usetikzlibrary{decorations}
\usetikzlibrary{decorations.pathreplacing}
\usetikzlibrary{topaths}
\usetikzlibrary{shapes.geometric}
\usetikzlibrary{matrix}
\usetikzlibrary{arrows.meta}
\usepackage{tikz-cd}
\usepackage{graphicx, amsmath, amssymb, amsthm, amscd}
\usepackage{color}
\usepackage{epsf}
\usepackage{enumerate}
\usepackage{mathtools}
\usepackage{caption}
\usepackage{tikz}
\usetikzlibrary{decorations.pathreplacing}
\usetikzlibrary{decorations.markings}
\newcommand{\Mod}[1]{\ (\mathrm{mod}\ #1)}

\usepackage[utf8]{inputenc} % allow utf-8 input
\usepackage[T1]{fontenc}    % use 8-bit T1 fonts
\usepackage{hyperref}       % hyperlinks
\usepackage{url}            % simple URL typesetting
\usepackage{booktabs}       % professional-quality tables
\usepackage{amsfonts}       % blackboard math symbols
\usepackage{nicefrac}       % compact symbols for 1/2, etc.
\usepackage{microtype}      % microtypography
\usepackage{lipsum}

\newtheorem{question}{Question}
\newtheorem{theorem}{Theorem}[section]
\newtheorem{corollary}[theorem]{Corollary}
\newtheorem{lem}[theorem]{Lemma}
\newtheorem{claim}[theorem]{Claim}
\newtheorem{remark}[theorem]{Remark}

\newtheorem{definition}[theorem]{Definition}

\DeclareCaptionType{equ}[Diagram][]
 
\newcommand{\Ps}{\mathcal{P}}

 \newcommand{\N}{\mathbb{N}}

 \newcommand{\Z}{\mathbb{Z}}
 \newcommand{\R}{\mathbb{R}}

 \newcommand{\E}{\mathcal{E}}
 \newcommand{\F}{\mathcal{F}}
 \newcommand{\htop}{h_{\text{top}}}

 \newcommand{\z}{\varphi}
 \newcommand{\zi}{\varphi_\infty}
 \newcommand{\rs}{\mathfrak{s}}
 
 \newcommand{\eps}{\varepsilon}
 \renewcommand{\Bbb}[1]{\mathbb{#1}}
\newcommand{\spaceX}{\mathbb{D}_\infty}
\newcommand{\mainmap}{\hat{R}_\infty}
 \newcommand{\Mm}{\mathcal{C}}
 \newcommand{\black}{\color{black}}
 
\DeclareMathOperator{\cl}{cl}
\DeclareMathOperator{\Int}{int}
\DeclareMathOperator{\id}{id}
\DeclareMathOperator{\diam}{diam}
\DeclareMathOperator{\mesh}{mesh}
\DeclareMathOperator{\supp}{supp}
\DeclareMathOperator{\Per}{Per}

\input xy

\begin{document}

\title[Beyond $0$ and $\infty$: A solution to the Barge Entropy Conjecture]{Beyond $0$ and $\infty$: A solution to the Barge Entropy Conjecture}

\author{J. P. Boro\'nski}\thanks{J. B. was supported in part by the National Science Centre, Poland (NCN), grant no. 2019/34/E/ST1/00237.}

\address[J. P. Boro\'nski]{Faculty of Mathematics and Computer Science, Jagiellonian University in Krak\'ow, ul. Łojasiewicza 6, 30-348 Kraków, Poland -- and -- National Supercomputing Centre IT4Innovations, University of Ostrava, IRAFM, 30. dubna 22, 70103 Ostrava, Czech Republic} \email{jan.boronski@uj.edu.pl}

\author{J. \v Cin\v c}\thanks{J. \v C. was supported by the Austrian Science Fund (FWF) Schrödinger Fellowship stand-alone project J 4276-N35 and Slovenian Research Agency ARIS grant J1-4632.} \address[J. \v Cin\v c]{University of Maribor, Koro\v ska 160, 2000 Maribor, Slovenia -- and -- National Supercomputing Centre IT4Innovations, University of Ostrava, IRAFM, 30. dubna 22, 70103 Ostrava, Czech Republic} \email{jernej.cinc@um.si}

\author{P. Oprocha}\thanks{P. O. was supported by
National Science Centre, Poland (NCN), grant no. 2019/35/B/ST1/02239.}
\address[P. Oprocha]{AGH University of Krakow, Faculty of Applied Mathematics,	al. Mickiewicza 30,	30-059 Krak\'ow, Poland -- and -- National Supercomputing Centre IT4Innovations, University of Ostrava,	IRAFM,	30. dubna 22, 70103 Ostrava,	Czech Republic} \email{oprocha@agh.edu.pl}

\date{\today}

\maketitle

{\centering\footnotesize \emph{Dedicated to the memory of W\l{}odzimierz J. Charatonik (1957-2021).}\par}
\begin{abstract}
We prove the entropy conjecture of M. Barge from 1989: for every $r\in [0,\infty]$ there exists a pseudo-arc homeomorphism $h$, whose topological entropy is $r$. Until now all pseudo-arc homeomorphisms with known entropy have had entropy $0$ or $\infty$. 
\end{abstract}

%\keywords{Pseudo-arc, topological entropy, homeomorphism}
%\maketitle

\section{Introduction}
\subsection{Main result}
The present paper concerns a long-standing open problem on the dynamics and the homeomorphism group of a one-dimensional fractal-like planar object called the pseudo-arc. The pseudo-arc was originally described in 1922 by Knaster \cite{Knaster}, as the first example of a hereditarily indecomposable continuum. Recall that a {\em continuum} is a compact and connected nondegenerate metric space. {\em Indecomposability} of the pseudo-arc means that it is not the union of two distinct proper subcontinua, and {\em hereditarility} of this property means that every subcontinuum is also indecomposable. The last 100 years have seen a very intensive research on the pseudo-arc, culminating in 2016 in the spectacular achievement of classification of topologically homogeneous plane compacta by Hoehn and Oversteegen \cite{Hoehn}, in which the pseudo-arc played the central role: any such compactum is topologically a point, pseudo-arc, circle, circle of pseudo-arcs, or a Cartesian product of one of these three with either a finite set, or Cantor set. In an even more recent major breakthrough, in 2019, Hoehn and Oversteegen \cite{HoehnOversteegen19} showed that the pseudo-arc is the only, other than the arc, hereditarily equivalent\footnote{\black A continuum $K$ is said to be {\em hereditarily equivalnet}, if every proper subcontinuum $H\subset K$ is homeomorphic to $K$.} planar continuum, addressing a question of Mazurkiewicz from 1921 \cite{Mazurkiewicz}. In the present paper, we answer the following long-standing question in the affirmative.
\begin{question}{\sc (M. Barge, 1989 \cite{LewisProblems})}\label{q:1}
	Is every positive real number the entropy of some homeomorphism of the pseudo-arc?
\end{question}
\begin{theorem}\label{thm:main}
	For every $r\in [0,\infty]$ there exists a pseudo-arc homeomorphism $H_r$ such that $h_{top}(H_r)=r$. 
\end{theorem}
Our proof of Theorem \ref{thm:main} is obtained by a combination of a variant of Mary Rees' technique, developed recently by Beguin, Crovisier and Le Roux \cite{Cro}, with several inverse limit techniques developed for the purpose of Theorem \ref{thm:main}. The homeomorphisms $H_r$ are obtained as perturbations of a pseudo-arc homeomorphism $\mainmap$ that exhibits Cantor fan-like dynamics; i.e. $\mainmap$ possesses a unique fixed point and every other point lies in a Cantor set that is the closure of the orbit of that point with odometer acting on it. In fact we obtain even stronger result than Theorem~\ref{thm:main}, since we adapt the approach from \cite{Cro} which is very flexible. It allows to ``replace'' the Haar measure on the odometer by a joining of Haar measure with any other probability measure. Making such replacements recursively, we can introduce any countable sequence of invariant measures in the pseudo-arc, provided that all of them are extensions of an odometer. So in fact, we can get much richer structure of invariant measures than what is sufficient for solving Question~\ref{q:1}
(e.g. we can obtain an example of a pseudo-arc homeomorphism without a measure of maximal entropy).

In turn, we also answer three other open questions from around the same time. 
\begin{question}{\sc (M. Barge, 1989 \cite{LewisProblems})}\label{q:2}
	Does every homeomorphism of the pseudo-arc of positive entropy have periodic points of infinitely many different periods?
\end{question}
\vspace{-0.3cm}
{\bfseries Answer 2} {\itshape No, a unique fixed point may be the only periodic point.\footnote{This is optimal, since the pseudo-arc has the fixed point property \cite{Hamilton}.}}
\begin{question}{\sc (W. Lewis, 1989 \cite{LewisSurvey})}\label{q:3}
	What periodic structure does positive entropy of the pseudo-arc homeomorphism $f:P\to P$ imply?
\end{question}
\vspace{-0.3cm}
{\bfseries Answer 3} {\itshape None.}
\begin{question}{\sc (W. Lewis, 1989 \cite{LewisSurvey})}\label{q:4}
	If the homeomorphism $f: P \to P$ of the pseudo-arc $P$ has positive entropy, does $f$ have homoclinic orbits?
\end{question}
\vspace{-0.3cm}
{\bfseries Answer 4} {\itshape Not necessarily.} 

This in particular means that, in contrast to the interval dynamics, positive topological entropy on the pseudo-arc does not induce any standard horseshoe-like dynamical behavior. In fact, it turned out that dynamics of maps of the pseudo-arc with positive entropy can be much richer that expected in early intuitions, motivated by constructions of pseudo-arc homeomorphisms via inverse limits of arcs.

In the next section we provide a historic overview of the pseudo-arc and its appearance across various branches of mathematics. In Section \ref{rigflex} we discuss results from entropy theory, that provide background for Question 1. In Section \ref{leading} we review prior work on Question 1 and recall various partial results obtained until now. In Section \ref{ingredients} we describe the main ingredients of our proof of the main result, and give an outline of the present paper. 
\subsection{Pseudo-arc: historic overview, characterizations and its appearance across various branches of mathematics}
First let us recall some more historic highlights about the origin and current charactarizations of the pseudo-arc. Answering a question of Mazurkiewicz \cite{Mazurkiewicz} in the negative, whether the arc is the only continuum homeomorphic to all of its non-degenerate subcontinua, in 1948 Moise \cite{Moise1} constructed a ``new'' example of a space with the same self-similarity property, which he named a pseudo-arc. The same year Bing \cite{Bing48} constructed a ``new'' example of a homogenous indecomposable continuum. A few years later Bing \cite{Bing} published a paper where he proved that all three spaces constructed independently by Knaster, Moise and himself are homeomorphic. Bing \cite{BingPacific} showed that in the Baire category sense the pseudo-arc is a generic continuum. Namely, he showed that for any manifold  $M$ of dimension at least $2$, the set of subcontinua homeomorphic to the pseudo-arc is a dense $G_{\delta}$ subset of the hyperspace of all subcontinua of $M$. This phenomenon resembles the one of irrational numbers in $\mathbb{R}$: seemingly atypical objects in the space are in fact generic.

We say that a space is {\em arc-like} if it can be represented as an inverse limit of arcs, or equivalently if for any $\epsilon>0$ there exists a mapping from the space to the unit interval, with all fibers smaller than $\epsilon$.
In the view of the properties mentioned above, the pseudo-arc can be characterized as the unique:
\begin{itemize}
	\item homogeneous arc-like continuum (Bing \cite{Bing59}),
	\item planar homogeneous continuum different from the circle and from the circle of pseudo-arcs (Hoehn, Oversteegen \cite{Hoehn}),
	\item planar continuum homeomorphic to all of its subcontinua which is not an arc (Hoehn, Oversteegen \cite{HoehnOversteegen19}),
	\item hereditarily indecomposable arc-like continuum (Bing \cite{BingPacific}).
\end{itemize}
A closely related object is the pseudo-circle. The pseudo-circle was first described by Bing \cite{Bing} in 1951. All its proper subcontinua are pseudo-arcs, but in contrast to the pseudo-arc, the pseudo-circle is not homogenous \cite{Fearnley1970, Rogers}. It is characterized in \cite{Fearnley,RogersPhD} as the unique planar hereditarily indecomposable circle-like continuum which separates the plane into two components.

Besides being interesting from the topological perspective, the pseudo-arc and pseudo-circle appear also in other branches of mathematics. To review it chronologically, they first made their appearance in smooth dynamical systems through a construction of Handel \cite{Ha}, where he
obtained the pseudo-circle as a minimal set of a $C^{\infty}$ area-preserving diffeomorphism of the plane. This line of work was continued by Kennedy and Yorke \cite{KennedyYorke1}, who constructed a $C^{\infty}$ map on a $3$-manifold with an invariant set consisting of uncountably many pseudo-circle components. Their example is interesting also due to its stability under any small $C^1$ perturbation which suggests that it could arise naturally in physical simulations. Kennedy and Yorke pushed their result even further and provided in \cite{KennedyYorke2, KennedyYorke3} a construction of a diffeomorphism with the same properties as above on an arbitrary $7$-manifold. Recently the methods from \cite{Ha} were applied to construct various planar homeomorphisms whose stable sets consists of unions of translation pseudo-arcs \cite{RuizDelPortal}. In \cite{BO} the first and the last author constructed a homeomorphism on the $2$-torus with the pseudo-circle as a Birkhoff-like attractor. Furthermore, in this context a decomposition of the $2$-torus into pseudo-circles was given by B\'eguin, Crovisier and J\"ager \cite{BCJ}, which is invariant under a torus homeomorphism semi-conjugate to an irrational rotation. Also, Boro\'nski, Clark and Oprocha \cite{BCO} constructed a family of minimal sets of torus homeomorphisms, each of which does not admit a minimal homeomorphism on its Cartesian square. These spaces were obtained by blowing up an orbit on a Denjoy exceptional minimal set on the torus and inserting a null sequence of pseudo-arcs. \\
Furthermore, the pseudo-arc very recently appeared as a significant object in measure theory. Namely, it was proven by the second and third author in \cite{CO}, that inverse limit of a generic Lebesgue measure-preserving interval map is the pseudo-arc. The authors used this fact to construct a parameterized family of planar homeomorphisms with pseudo-arc attractors having rich boundary dynamics as well as displaying interesting measure-theoretic properties.

The pseudo-arc also surfaced in the field of complex dynamics. The origins of work in this direction come from a paper of Herman \cite{Herman}, who extended Handel's construction \cite{Ha} and provided a $C^{\infty}$-smooth diffeomorphism of the complex plane with an invariant open topological disk bounded by the pseudo-circle, on which the diffeomorphism is complex analytic and complex analytically conjugate to an irrational rotation, and on the complement the diffeomorphism is smoothly conjugate to such a rotation. Later, this result was further improved by Ch\'eritat \cite{Cheritat}, who showed that the pseudo-circle can appear as the boundary of a Siegel disk. Recently, Rempe-Gillen \cite{Rempe-Gillen} considered {\black the} pseudo-arc (and other arc-like continua) in the context of Eremenko's conjecture \cite{Eremenko}, proving that pseudo-arcs appear as one-point compactifications of connected components of Julia sets of transcendental entire functions.\\
Apart from dynamics the pseudo-arc appears also in isometric theory of Banach spaces, as a counterexample \cite{Rambla, Kawamura} to Wood's Conjecture \cite{Wood}, which asserted that there exists no non-degenerate almost transitive space. 
A closely related result was obtained by Irwin and Solecki in a seminal paper  \cite{IrwinSolecki}, where using tools from model theory they developed a dualization of the classical Fra\"{i}ss\'e limit construction, and obtained surjective universality and projective homogeneity of the pseudo-arc in the class of arc-like continua. This result gave a new characterization of the pseudo-arc and initiated a new research direction. One of the most notable results that followed is the result of Kwiatkowska \cite{Kwiatkowska}, who proved that the automorphisms group of the pseudo-arc has a residual conjugacy class. A difficult open problem that arose in this direction of research asks to describe the universal minimal flow of the homeomorphism group of the pseudo-arc, see \cite[Question 1.3]{Uspenskij} as well as \cite{BartosovaKwiatkowska}. This problem is only one of the many present in the literature that aim at improving the understanding of the homeomorphism group of the pseudo-arc. Other open questions include those asking about covering dimension of that group \cite{AvM}, or whether it contains the Erdös space\footnote{This question was raised by Logan Hoehn at Spring Topology and Dynamics Conference in 2019. The Erdös space consists of all vectors in the real Hilbert space $\ell^2$ that have only rational coordinates.}.\\
\subsection{Topological rigidity and flexibility}\label{rigflex}
Topological entropy is a measure of complexity of maps on topological spaces and was originally defined by Adler, Konheim and McAndrew \cite{AKMc} as a topological analogue of metric entropy from ergodic theory. Bowen \cite{Bowen} characterized it as the exponential growth rate of the number of orbits which can be distinguished to an increasing number of iterates using measurements of fixed accuracy. Another characterization which will prove useful to us is through the variational principal, as the supremum of the measure-theoretical entropies of regular invariant probability measures \cite{Dinaburg}. The question of entropy rigidity and flexibility within a given class of maps on an underlying topological space is one of the mainstream questions in the modern theory of Dynamical Systems. For example, Herman's Positive Entropy Conjecture, recently resolved by Berger and Turaev  \cite{Berger} is one of the instances of this line of research. Shub and Williams \cite{ShubWilliams} related topological entropy of maps to the underlying space and this work was continued by a seminal paper of Manning \cite{Manning}, where connections between the structure of a compact Riemannian manifold and the topological entropy of the geodesic flow on its unit tangent bundle were given. In such a great generality let us also mention the most recent important work \cite{DeSLVY} by De Simoi, Leguil, Vinhage, and Yang, who studied in detail entropy rigidity and flexibility for Anosov flows on a compact $3$-manifold preserving a smooth volume.\\
All the works mentioned above concentrated on more trackable systems, where foliations of the space exist at least to some extent. The topological structure of many relevant one-dimensional spaces, such as various hyperbolic attractors, including those derived from Anosov, is very different from manifolds, and that feature has proven to be a challenging aspect of questions concerning entropy of homeomorphisms of those spaces. For example, the pseudo-arc that we are dealing with in the present paper is as far from being a foliated space as possible, since all its arc-components are degenerate to a point. Perhaps a better
paradigm example of one-dimensional continua to start with in this aspect are \emph{solenoids}, which are inverse limits of the unit circle with finite-sheeted covering maps, and which were used as one of the first examples of hyperbolic attractors. 
Let $\Sigma_{\alpha}$ denote the solenoid obtained as the inverse limit of covering maps of degree $\alpha=(p_1,p_2,\ldots)$, where one can assume without loss of generality that the powers $p_i$ are prime numbers (Bing \cite{Bing1960} and McCord \cite{McCord}). Each homeomorphism $h$ of a solenoid $\Sigma_{\alpha}$ has the same topological entropy as a certain automorphism from its automorphism group associated uniquely with $h$. This result was proven by Kwapisz~\cite{Kwapisz}. There is also a way to obtain planar arc-like one-dimensional continua from solenoids by taking appropriate quotient maps.
Namely, {\em generalized Knaster continua} $K_{\alpha}$ can be defined using sequences of primes $\alpha=(p_1,p_2,\ldots)$ and are isomorphic to $\Sigma_{\alpha}/\sim$ where the equivalence relation identifies points from $\Sigma_{\alpha}$ with their inverses.
 The corresponding quotient mapping is exactly two-to-one except at one or two points. These generalized Knaster continua can also be represented as inverse limits $K_{\alpha} = \underleftarrow{\lim}(I , f_{p_i} )$. Kwapisz~\cite{Kwapisz} proved that the mapping class group of $K_{\alpha}$ is isomorphic to either $\otimes^k_{n=1}\mathbb{Z}$ or $\otimes^k_{n=1}\mathbb{Z}\otimes \mathbb{Z}$, depending on whether $2$ does or does not occur infinitely often in the sequence $\alpha$. Apart from generalized Knaster continua, another well-studied family of arc-like continua is the
family of tent inverse limit spaces, that arise as attractors of some planar homeomorphisms \cite{Bruin, BdCH, BdCH2}. It was shown that in this family typical parameters give intrinsically complicated continua that are locally far from being a product of Cantor set and open arc \cite{BBD}. However, it turns out that these typical parameters do not give rise to homeomorphic spaces. In fact much more is true. A conjecture due to Ingram, which attracted much attention in the last three decades, asserted that 
inverse limit spaces of tent maps
with different slopes $s\in(\sqrt{2},2]$ are non-homeomorphic. It was proven in \cite{BBS} by Barge, Bruin and \v Stimac.
As a somewhat surprising by-product of its solution (but in agreement with the results described above), it was shown in
\cite{BS1} that the mapping class group of these spaces is $\Z$.
This in turn was used to characterize possible values of the topological entropy of homeomorphisms on these spaces \cite{BS2}. According to it, the entropy is always a non-negative integer multiple of the logarithm of the slope of the tent map.
\subsection{Work leading to the solution of Barge's entropy conjecture}\label{leading}
In this subsection we will review the important historical steps towards the answer to Question~\ref{q:1}. 
A natural approach to construct a homeomorphism on the pseudo-arc is to use the inverse limit technique using a single bonding map; such approach  provided several new interesting results in the past which we review later in this section. If one can determine the topological entropy of the bonding map, the result of Bowen \cite{Bowen1} shows that the natural extension homeomorphism on the inverse limit will attain the same value.
Henderson \cite{Henderson} gave the first example of an interval map that gives pseudo-arc in the inverse limit. His map is very simple from the dynamical perspective; all points but one repelling fixed point are attracted to the attracting fixed point and there are no other recurrent points. Therefore, the topological entropy of this map and subsequently the topological entropy of the natural extension on the pseudo-arc is $0$. Lewis \cite{Lewis} provided a way to lift dynamics from the interval to pseudo-arc; namely he showed that every interval map is semi-conjugate to a pseudo-arc homeomorphism and underlined the richness of its homeomoprhism group, however giving no additional information on topological entropies of these extensions. 
Shortly after Question~\ref{q:1} was asked, Kennedy \cite{Kennedy1989}  proved that if $C$ is a Cantor set that intersects each composant of a pseudo-arc $P$ in at most one point, then each homeomorphism of $C$ extends to a homeomorphism of $P$ onto $P$. As a consequence she obtained first known homeomorphisms of the pseudo-arc with positive topological entropy. However, she did not establish the control of the dynamics for such homeomorphisms and subsequently could not assure that the topological entropy of her constructed homeomorphisms is finite. Also, Cook \cite{Cook} has shown, that no Cantor set $C$ from her construction can intersect every composant of $P$.
Kennedy \cite{KennedyStable} also provided the first example of a transitive homeomorphism of the pseudo-arc, as an extension of the full tent map from the interval. In the same year Minc and Transue \cite{MincTransue} provided a more general construction of transitive maps on the pseudo-arc as shifts from inverse limits of intervals.
First step to show that construction through inverse limits with single bonding maps will not suffice to solve the conjecture was done by Block, Keesling and Uspenskij \cite{BlockKeeslingUspenskij} who proved that the homeomorphisms on the pseudo-arc, that are conjugate to shifts on inverse limit of intervals, have topological entropy greater than $\mathrm{log}(2)/2$, if positive. Subsequently, Mouron \cite{Mouron} proved that natural extensions on pseudo-arc through inverse limits have topological entropy either $0$ or $\infty$. The last result meant that, in order to construct finite non-zero topological entropy homeomorphism on the pseudo-arc, one needs a novel approach. As a corollary he also obtained that homeomorphisms on the pseudo-arc that are semi-conjugate to interval maps have topological entropy either $0$ or $\infty$; this in particular has a consequence that the topological entropy of the transitive examples of Kennedy, extensions by Lewis and shift homeomorphism of Minc and Transue are $\infty$. Motivated by this results and lack of other examples, Mouron asked in \cite{Mouron} if a homeomorphism of a hereditarily indecomposable continuum must have entropy $0$ or $\infty$? Addressing this question, Boro\'nski and Oprocha showed in \cite{BO2} that if $G$ is a topological graph, then provided $\underleftarrow{\lim}(G,f)$ is hereditarily indecomposable and $h_{\mathrm{top}}(f)>0$  
the entropy must be infinite. For circle maps of non-zero degree it was proven that $h_{\mathrm{top}}(f)=\infty$. In particular it means, that in contrast to pseudo-arc homeomorphisms,
zero entropy homeomorphisms of the pseudo-circle are never conjugate to shift homeomorphisms on the inverse limit of circles (in particular, this includes the diffeomorphism of Handel).
Later Boro\'nski, Clark and Oprocha \cite{BCO1} provided for any $\beta\in [0,\infty]$ a construction of a hereditarily indecomposable continuum admitting a minimal homeomorphism with topological entropy $\beta$ and thus answering Mouron's question in the negative. The present work answers Question~\ref{q:1} in full generality; as a consequence we also obtain answers on Questions~\ref{q:2}-\ref{q:4}.

\subsection{Description of the main ingredients of the proof}\label{ingredients}

Now let us address the structure of the proof of our main theorem. Our starting point are the following two theorems.

\begin{theorem}\label{Thm:LewisPeriodic}(Lewis, \cite{LewisPeriodic})
For any $n\in\mathbb{N}$ there exists an arc-like continuum $\mathcal{P}$, that admits an $n$-periodic homeomorphism $g$ with a unique fixed point, and all other points of least period $n$.
\end{theorem}
\begin{theorem}(Lewis, \cite{LewisGroups})
For any sequence of prime numbers $\mathbb{P}$ there exists an arc-like continuum $\mathcal{P}$ that admits a $\mathbb{P}$-adic Cantor group action. 
\end{theorem}
In both \cite{LewisPeriodic} and \cite{LewisGroups} it is stated without proof that  modifications of the constructions presented therein can be made to ensure that $\mathcal{P}$ is a pseudo-arc, and that the homeomorphism $g$ in Theorem \ref{Thm:LewisPeriodic} extends to a period $n$ rotation of $\mathbb{D}^2$. Even though it is widely accepted that these results indeed admit such modifications, a formal proof has not appeared in the literature, and since the main result of the present paper resolves a long-standing problem, we have decided that a formal proof with the desired modifications  should be supplied. The following result is also a more detailed description of the dynamical properties of Lewis' maps. By $\omega(f,x)$ we denote $\omega$-limit set of $x$ with respect to a map $f$, and by $\Per(f)$ the set of periodic points of $f$. 
\begin{theorem}\label{odoOld}
Let $\mathbb{P}=(q_n:n\in\mathbb{N})$ be a sequence of primes and $\phi_{\mathbb{P}}:\Lambda\to \Lambda$ be the $\mathbb{P}$-adic odometer. There exists a pseudo-arc homeomorphism $\mainmap:\Ps_\infty\to \Ps_\infty$ and $0_{\infty}\in \Ps_\infty$ such that \begin{enumerate}
    \item $\Per(\mainmap)=\{0_{\infty}\}$, 
    \item $\mainmap|_{\omega(\mainmap,x)}$ is conjugate to $\phi_{\mathbb{P}}$ for every $x\in \Ps_\infty\setminus\{0_{\infty}\}$, and
    \item $x\in \omega(\mainmap,x)$ for every $x\in \Ps_\infty$,
 \end{enumerate}
where $\Ps_\infty=\varprojlim (\Ps_n,\varphi_{\infty,n})$ is an inverse limit of pseudo-arcs, with the bonding maps $\varphi_{\infty,n}\colon {\black \Ps_{n+1}\to \Ps_n}$ being {\black restrictions of the} branched {\black $2\text{-to-}1$ covers on $6$-dimensional cubes.}%$q_{n}\text{-to-}1$ covers of 2-disks $S_n$, with one branch point $0_n$. 
%{\black Furthermore, for every $\eps>0$ there exists an $\eps$-cover $\mathcal{W}$ of $\mathcal{P}_n$ by topological disks such that if $Z\subset \mathcal{P}_n$ is a continuum and if $d_H(Z,0_n)>\eps$ then there exists a chain subcover $\mathcal{W}'\subset \mathcal{W}$ of $Z$.}
\end{theorem}

{\black In fact, we will prove an extended version of Theorem~\ref{odoOld}, which is Theorem~\ref{odo}. We will present it in details later since it requires additional technical notation.}

The main difficulty in obtaining homeomorphisms on pseudo-arc with finite non-zero entropy is sufficient control of  their dynamics. Theorem~\ref{odoOld} serves as our starting point for constructing a more complicated homeomorphism with richer dynamics and to obtain the main result it is important that such starting homeomorphism carries no topological entropy. A very useful procedure that helps us in this direction is called the {\em Denjoy-Rees technique} and was described in \cite{Cro}.  There, the authors used the technique to get the following result: any compact manifold of dimension at least two admitting a minimal uniquely ergodic homeomorphism  also admits a uniquely ergodic homeomorphism with positive topological entropy. 
The Denjoy-Rees technique can be viewed as an ingenious generalization of the famous Denjoy example \cite{Denjoy} (a periodic point free homeomorphism of the unit circle which is not conjugate to a rotation) and subsequently more involved Rees' construction \cite{Rees} (homeomorphism of the two-torus which is minimal and has positive topological entropy) with the additional new technique which allows to control how rich the dynamics of these homeomorphisms {\black are}. For a more detailed outline of the Denjoy-Rees technique we refer the reader to the article \cite[Subsection 1.5.2. pp. 257--259]{Cro}.
We follow the Denjoy-Rees technique described above, however, several new difficulties arise in our setting. The first of them is that we are working on a one-dimensional space, unlike in \cite{Cro}. We solve this by providing an inverse limit version of the Denjoy-Ress technique; namely we construct our pseudo-arc as the inverse limit of pseudo-arcs and we are applying the technique on the coordinate spaces of the inverse limit. The other difficulty is that, after this enrichment, we want to get back {\black a} space homeomorphic to the original one-dimensional space, the pseudo-arc. Here we use several properties of the pseudo-arc, namely it is important that the space is homogeneous so that we can find homeomorphism between neighbourhoods of different sets and points.  
Also, crookedness of the pseudo-arc is used in a sense that the deformations we perform on the space throughout our construction cannot make it even more crooked. The argument would fail at this point in the case of many other arc-like continua (e.g. for the arc or the Knaster continuum). In \cite{Cro} the authors state a collection of conditions $\mathbf{A_1}$-$\mathbf{A_3}$, $\mathbf{B_1}$-$\mathbf{B_6}$ and $\mathbf{C_1}$-$\mathbf{C_8}$ which allow them to obtain their main result. They state the conditions for {\em rectangles} which are in their context closed unit balls in $\R^d$ where $d>1$ is the dimension of a given compact topological manifold. For us, the rectangles from their setting will be pre-images of natural projections from the inverse limit of closed cubes {\em tamely embedded}\footnote{A $d$-dimensional cube $B\subset \mathbb{R}^d$ is \emph{tamely embedded} in the interior of $X$ if $B$ is the image of the unit ball of $\mathbb{R}^d$ under a continuous one-to-one map from $\mathbb{R}^d$ into $X$.} in $\mathbb{D}^4$, where $\mathbb{D}^4$ denotes the four dimensional unit cube. Therefore, our rectangles will be of the form of closed cubes from  $\mathbb{D}^4$ times a Cantor set. Despite this change in the setting, basic tools from \cite{Cro} still hold true and we will be able to perform an adaptation of the Denjoy-Rees technique through inverse limits. Our basic starting point is to construct a Cantor set $K\subset \mathcal{P}_{\infty} \subset\mathbb{D}_{\infty}$ of positive Haar measure. Then the orbits of $K$ are ``blown-up'' similarly as in Rees' construction and the Denjoy-Rees technique  will in addition allow us to completely control the orbits of points from $K$ so that no measure escapes to an uncontrolled part of $\mathcal{P}_{\infty}\setminus K$. Therefore we will be able to control invariant measures on $\mathcal{P}_{\infty}$, subsequently also metric entropy as required and use the variational principle to determine explicit values of topological entropy as well.

The paper is organized as follows. In Sections~\ref{preliminaries} and ~\ref{crookedness}, we present some standard notation and state several results on crookedness that we will use in later sections.
In Section~\ref{sec:odo}, 
we will prove Theorem~\ref{odoOld}.
In Section~\ref{sec:DRtechnique}, we will introduce conditions for the Denjoy-Rees technique in our setting. Section~\ref{sec:mainproof} is the core part of the paper; in this section we will prove Lemma~\ref{lem:main}, i.e., show how to realize conditions given in Section~\ref{sec:DRtechnique}. Later in Subsection~\ref{Q_is_pseudo} we prove the crucial properties of constructed maps on the pseudo-arc.
In Appendix~\ref{sec:appendix}, written by George Kozlowski, a  proof of Theorem \ref{thm:Kozlowski} that we use in our proof of the main result is included, as we could not find any reference for it in the literature.

\section{Preliminaries}\label{preliminaries}

In this section we introduce some preliminary notation needed through the rest of the paper.
Let $\N:=\{0,1,2,\dots\}$.

Let $Z$ be a compact metric space and $f:Z\to Z$ a continuous map. 

For the main technical tool we use the
\emph{inverse limit spaces}. For a collection of continuous maps $f_i:Z_{i+1}\to Z_i$ where $Z_i$ are compact metric spaces for all $i\geq 1$ we define
\begin{equation}
\underleftarrow{\lim} (Z_i,f_i)
:=
\{\big(z_{1},z_2,\ldots \big) \in Z_1\times Z_2,\ldots\big|  
z_i\in Z_i, z_i=f_i(z_{i+1}), \text{ for any } i \geq 1\}.
\end{equation}
We equip $\underleftarrow{\lim} (Z_i,f_i)$ with the subspace 
metric induced from the 
\emph{product metric} in $Z_1\times Z_2\times\ldots$,
where $f_i$ are called the {\em bonding maps}. In general we start the enumeration for the inverse limit with $1$, however it will sometimes be convenient to start with $0$ (e.g. in Section~\ref{sec:odo}).
If $Z_i=Z$ and $f_i=f$ for all $i\geq 1$, the inverse limit space $\underleftarrow{\lim} (Z_i,f_i)$ 
also comes with a natural homeomorphism, 
called the \emph{natural extension} of $f$ (or the 
\emph{shift homeomorphism})
$\hat f:\underleftarrow{\lim} (Z,f)
\to \underleftarrow{\lim} (Z,f)$, 
defined as follows. 
For any $z= \big( z_{1},z_2,\ldots \big)\in \underleftarrow{\lim} (Z,f)$,
\begin{equation}
\hat{f}(z):= \big(f(z_1),z_{1},z_2,\ldots \big).
\end{equation}
By $\pi_{i}$ we shall denote
the \emph{$i$-th projection} 
from 
$\underleftarrow{\lim} (Z,f)$ to the $i$-th coordinate. 

Now let us give preliminaries on the topological and metric entropy.
In this paper topological entropy of $f$ will be denoted $h_{\mathrm{top}}$ and if $\mu$ is an $f$-invariant probability measure, then associated metric entropy is denoted $h_{\mu}(f)$. For definitions and basic facts on invariant measures and entropy we refer the reader to standard textbooks, e.g. \cite{Walters}. Below we briefly recall a few most important facts used in this paper.

If $\mathcal{M}_f(Z)$ denotes the collection of all $f$-invariant and Borel probability measures of $Z$, and $\mathcal{M}^e_f(Z)$ denotes the the collection of all the elements of $\mathcal{M}_f(Z)$ that are ergodic, then the  \emph{variational principal} connects the topological and metric entropies:
$$
h_{\mathrm{top}} (f)=\sup\{h_{\mu}(f): \mu\in \mathcal \mathcal{M}_{f}(Z)\}=\sup\{h_{\mu}(f): \mu\in \mathcal \mathcal{M}^e_{f}(Z)\}.
$$

Finally, let us introduce the notion that will be used in the very end of the paper.
A measurable dynamical system $(Z, \mathcal{B},S)$ is a bijective bi-measurable map $S$ on a set $Z$ with $\sigma$-algebra $\mathcal{B}$. 

An $S$-invariant set $Z_0\subset Z$ is {\em universally full}, if it has full measure for any $S$-invariant probability measure on $(Z,\mathcal{B})$.

Two measurable systems $(Z,\mathcal{B}_1, S)$ and $(Y,\mathcal{B}_2,T)$ are called {\em universally isomorphic} if there exist:
\begin{enumerate}
\item $S$-invariant universally full set $Z_0\subset Z$,
\item $T$-invariant universally full set $Y_0\subset Y$, and
\item bijective bi-measurable map $\theta:Z_0\to Y_0$ such that $\theta \circ S=T\circ \theta$.
\end{enumerate}

It is well known that universally isomorphic systems have the same topological entropy.

\section{Crookedness Revisited}\label{crookedness}

In what follows we will recall known results that will help us proving Theorem~\ref{odo}.

\begin{definition}
	A metric $d$ on a finite (topological) tree $T$ is called an \emph{arc-length metric}, if for each pair of points $x$ and $y$ in $T$, there is an isometry $\beta: [0,d(x,y)]\to T$ such that $\beta(0)=x$ and $\beta(d(x,y))=y$. For a subset $A\subset T$, $\mathrm{diam}(A)$ denotes the supremum of distances between two points from $A$ with respect to $d$. For $f,g:T\to T$, $d(f,g)$ will denote $\sup_{x\in T} d(f(x),g(x))$.
\end{definition}

\begin{definition}
 For $m\geq 3$, a \emph{chain} 
is a finite family of open subsets
$\mathcal{D}=\{D_{\ell}\}_{\ell \in \Bbb Z_m}$ of $\mathbb{R}^n$ 
such that
\begin{equation}
D_{\ell} \cap D_{ k} \ne \emptyset \text{ if and only if } 
| k-  \ell|\leq 1.
\end{equation}
Each element $D_{\ell}$ is called a link.
\end{definition}

\begin{definition}
Let $\mathcal{D}=\{D_{ \ell} :  \ell \in \Z_m\}$ and
$\mathcal{C}=\{C_{ \ell} :  \ell \in \Z_s\}$ 
be two chains. 
We say $\mathcal D$ is \emph{crooked inside} 
$\mathcal{C}$
if there exists a map $\pi \colon \Z_m\to \Z_s$
with the following properties. 
\begin{enumerate}[(i)]
	\item for any $ i \in \Bbb Z_m$,
	the closure of $D_{ i}$ is contained in $C_{\pi( i)}$.
	\item if 
	$ i< j$ 
	are two elements in $\Bbb Z_m$,
	such that $\pi( i) \leq \pi( j)$,
	and for any $ i< k<  j$,
	$\pi( i) \leq \pi( k) \leq \pi( j)$, 
	then there exist
	$ i< u\leq v< j$,
	such that, 
	\begin{align}
	& |\pi( i)-\pi( v)|\leq 1. \\
	& |\pi( j)-\pi( u)|\leq 1.
	\end{align}
	\end{enumerate}
\end{definition}
\begin{definition}
	Let $\kappa:T\to T$ be a continuous function, where $T$ is a finite tree and let $n>0$ be an integer.
	A path $\alpha:[0,1]\to T$ is \emph{$(\kappa,\eps)$-crooked}, if there exist $0\leq s< t\leq 1$ such that $d(\kappa\circ \alpha(0),\kappa\circ \alpha(t))\leq \eps$ and $d(\kappa\circ \alpha(s),\kappa\circ \alpha(1))\leq\eps$. The map $\kappa$ is called \emph{$\eps$-crooked} if every path $\alpha: [0,1]\to T$ is $(\kappa,\eps)$-crooked.
\end{definition}

Fix $\eps>0$. Let $f\colon T\to T$ be a continuous map on a finite tree $T$. We define 
$$L(\eps,f)=\mathrm{sup}\{\delta>0; \text{ if } d(x,y)<\delta \text{ then } d(f(x),f(y))<\eps\}.$$
The following result follows from Lemma~2 in \cite{Brown}.

\begin{lem}\label{lem:Brown}
	For each $i\geq 1$ {\black fix a finite tree $T^i$} of the diameter $d_i$. Suppose that for all $m\geq 1$, $f_m:T^{m+1}\to T^m$ is $\eps_m$-crooked with $\eps_m<\min_{i<m-1}L(2^{-m}d_i,f_i\circ \ldots \circ f_{m-1})$.
	Then $\varprojlim(T^i,f_i)$ is hereditarily indecomposable.
\end{lem}

We will also need the following definition from \cite{Tuncali}. In the context of topological graphs it is in fact equivalent to the standard definition on open sets.
\begin{definition}
We say a map $f \colon T \to T$ {\black on a finite tree $T$} is \emph{topologically exact (or locally eventually onto)} if  for each nondegenerate subcontinuum $A$ of $T$, there is a positive integer $N$ such that $f^N(A) = T$.
\end{definition}

{\black Let $\mathbb{D}^2:=\{z\in \mathbb{C}, |z|\leq 2 \}|\times \{0\}$ and for $k\geq 2$ define $\mathbb{D}^k:=\{z\in \mathbb{C}, |z|\leq 2 \}\times [-1,1]^{k-2}$. For $k\geq 3$ we let $\z^k:\mathbb{D}^k\to\mathbb{D}^k$ be the branched 2-fold covering given by $$\z^k(r\cos(\theta), r\sin(\theta),t_3,t_4,\ldots, t_k)=(r\cos(2\theta), r\sin(2\theta),t_3,t_4,\ldots, t_k).$$ 
 Let $\bar{0}$ be the complex zero, $\sqrt[2^n]{1}=\{w_j:j=0,\ldots,2^n-1\}$ be the complex $2^n$-th roots of $1$, and $\ell_l:=\{(r\cos(\frac{l2\pi}{2^n}), r\sin(\frac{l2\pi}{2^n}),0)|r\in[0,1/3]\}$ be the chords in $\mathbb{D}^2$ from $(\bar{0},0)$ to $(\frac{1}{3}w_l,0)$, for $l=0,\ldots,2^{n}-1$. 
We set 
$$T_n:=\bigcup_{i=0}^{2^{n}-1}\ell_i\subset \mathbb{D}^2\subset \mathbb{D}^k.$$  
 We call $\ell_i$ the {\em branches} of $T_n$.
Note that $T_n$ is invariant under rotation 
$R^k_n:\mathbb{D}^k\to \mathbb{D}^k$ defined by 
$$R^k_n(r\cos(\theta), r\sin(\theta),t_3,t_4,\ldots, t_k)=(r\cos(\theta+2^{-n+1}\pi), r\sin(\theta+2^{-n+1}\pi),t_3,t_4,\ldots, t_k).$$
When the dimension is clear from the context we will simply write $\z$ instead of $\z^k$ and $R_n$ instead of $R^k_n$.
Let $Q\subset \mathbb{D}^k$ be such that $R_n(Q)=Q$ and let $g:Q\to \mathbb{D}^k$. We say that $g$ is {\em invariant under rigid rotation $R_n$}, if $R_n(g(x))=g(R_n(x))$ for all $x\in Q$.}

%Let $\mathbb{D}^3=\mathbb{D}^2\times [-1,1]$ and $\mathbb{D}^4=\mathbb{D}^2\times [-1,1]^2$. We extend the map $\varphi$ to $\varphi: \mathbb{D}^3\to \mathbb{D}^3$ by  $$\z(r\cos(\theta), r\sin(\theta),t)=(r\cos(2\theta), r\sin(2\theta),t).$$ 

Let us recall Lemmas~3.2 and 3.3 from \cite{Tuncali}. Both furthermore parts follow by the constructions in these proofs (c.f. Remark. 1)) on page 203 and page 205 in \cite{Tuncali}). 
Note that the proofs of these two lemmas are constructive. 
{\black To see the moreover parts, first note that the proofs of Lemmas 3.2 and 3.3 from \cite{Tuncali} rely on the construction in Lemma 3.1 from \cite{Tuncali}. Observe that the map $h$ from Lemma~\ref{tunc:3.2} can be defined so that it is invariant under rigid rotation $R_n$. First, we can start with the partition of $T_n$ that is invariant under $R_n$. Then, in the notation of the proof of Lemma 3.1 from \cite{Tuncali}, if $R,R'$ are elements of the partition such that $R=R_n(R')$, then $R^{*}=R_n(R')^{*}=R_n((R')^{*})$. Now, let $R$ be an element of partition contained in $\ell_j$. When defining $\phi_R$ on p. 200 in \cite{Tuncali} we change the order $(v_{R,0},a_1,a_2,\ldots ,a_{2^n},a_{2^n -1},\ldots,a_1,v_{R,1})$ to the order depending on $j$, for example 
$$(v_{R,0},a_{1+j\Mod{2^n}},a_{2+j\Mod{2^n}},\ldots ,a_{2^n+j\Mod{2^n}},a_{2^n -1 +j\Mod{2^n}},\ldots,a_{1+j\Mod {2^n}},v_{R,1} ).$$ 
We also assume that $a_{j}\in \ell_{j-1}$ for all $1\leq j\leq 2^n$. The last assumption is not important for the invariance under $R_n$, but it will be important later in the construction.}

\begin{figure}[!ht]
\centering
\begin{tikzpicture}[scale=3]
\draw (0,1)--(0,-1);
\draw (-1,0)--(1,0);
\draw (-0.03,0.1)--(0.03,0.1);
\draw (-0.03,0.2)--(0.03,0.2);
\draw (-0.03,0.3)--(0.03,0.3);
\draw (-0.03,0.4)--(0.03,0.4);
\draw (-0.03,0.5)--(0.03,0.5);
\draw (-0.03,0.6)--(0.03,0.6);
\draw (-0.03,0.7)--(0.03,0.7);
\draw (-0.03,0.8)--(0.03,0.8);
\draw (-0.03,0.9)--(0.03,0.9);
\draw[orange,thick] (0.7,0)--(0.8,0);
\draw[orange,thick,-latex] (0.7,0)--(0.8,0);
\draw[orange,thick] (0.1,0)--(0.2,0);
\draw[orange,thick,-latex] (0.1,0)--(0.2,0);

\draw (-0.03,-0.1)--(0.03,-0.1);
\draw (-0.03,-0.2)--(0.03,-0.2);
\draw (-0.03,-0.3)--(0.03,-0.3);
\draw (-0.03,-0.4)--(0.03,-0.4);
\draw (-0.03,-0.5)--(0.03,-0.5);
\draw (-0.03,-0.6)--(0.03,-0.6);
\draw (-0.03,-0.7)--(0.03,-0.7);
\draw (-0.03,-0.8)--(0.03,-0.8);
\draw (-0.03,-0.9)--(0.03,-0.9);
\draw[red,thick] (-0.7,0)--(-0.8,0);
\draw[red,thick,-latex] (-0.7,0)--(-0.8,0);
\draw[red,thick] (-0.1,0)--(-0.2,0);
\draw[red,thick,-latex] (-0.1,0)--(-0.2,0);

\draw (0.1,-0.03)--(0.1,0.03);
\draw (0.2,-0.03)--(0.2,0.03);
\draw (0.3,-0.03)--(0.3,0.03);
\draw (0.4,-0.03)--(0.4,0.03);
\draw (0.5,-0.03)--(0.5,0.03);
\draw (0.6,-0.03)--(0.6,0.03);
\draw (0.7,-0.03)--(0.7,0.03);
\draw (0.8,-0.03)--(0.8,0.03);
\draw (0.9,-0.03)--(0.9,0.03);
\draw[teal,thick] (0,0.7)--(0,0.8);
\draw[teal,thick,-latex] (0,0.7)--(0,0.8);
\draw[teal,thick] (0,0.1)--(0,0.2);
\draw[teal,thick,-latex] (0,0.1)--(0,0.2);

\draw (-0.1,-0.03)--(-0.1,0.03);
\draw (-0.2,-0.03)--(-0.2,0.03);
\draw (-0.3,-0.03)--(-0.3,0.03);
\draw (-0.4,-0.03)--(-0.4,0.03);
\draw (-0.5,-0.03)--(-0.5,0.03);
\draw (-0.6,-0.03)--(-0.6,0.03);
\draw (-0.7,-0.03)--(-0.7,0.03);
\draw (-0.8,-0.03)--(-0.8,0.03);
\draw (-0.9,-0.03)--(-0.9,0.03);
\draw[blue,thick] (0,-0.7)--(0,-0.8);
\draw[blue,thick,-latex] (0,-0.7)--(0,-0.8);
\draw[blue,thick] (0,-0.1)--(0,-0.2);
\draw[blue,thick,-latex] (0,-0.1)--(0,-0.2);

\draw[red,thick] (-0.7,0.05)--(-0.8,0.05)--(-0.8,0.1)--(-0.6,0.1)--(-0.6,0.15)--(-0.9,0.15)--(-0.9,0.2)--(-0.7,0.2);

\draw[red,thick,latex-] (-0.7,0.15)--(-0.8,0.15);
\draw[blue,thick,latex-,rotate=90] (-0.7,0.15)--(-0.8,0.15);
\draw[orange,thick,latex-,rotate=180] (-0.7,0.15)--(-0.8,0.15);
\draw[teal,thick,latex-,rotate=270] (-0.7,0.15)--(-0.8,0.15);

\draw[orange,dotted,thick] (0.7,-0.05)--(0.8,-0.05)--(0.8,-0.1)--(0.6,-0.1)--(0.6,-0.15)--(0.9,-0.15)--(0.9,-0.2)--(0.7,-0.2);

\draw[teal, dashed, thick] (0.05,0.7)--(0.05, 0.8)--(0.1,0.8)--(0.1,0.6)--(0.15,0.6)--(0.15,0.9)--(0.2,0.9)--(0.2,0.7);

\draw[blue, dash dot, thick] (-0.05,-0.7)--(-0.05,-0.8)--(-0.1,-0.8)--(-0.1,-0.6)--(-0.15,-0.6)--(-0.15,-0.9)--(-0.2,-0.9)--(-0.2,-0.7);

%\draw[red,thick,yscale=-1] (-0.1,0.05)--(-0.05,0.05)--(-0.05,0.25)--(-0.1,0.25)--(-0.1,0.1)--(-0.3,0.1)--(-0.3,-0.1)--(-0.1,-0.1)--(-0.1,-0.3)--(0.1,-0.3)--(0.1,-0.1)--(0.3,-0.1)--(0.3,0.1)--(0.1,0.1)--(0.1,0.3)--(-0.15,0.3)--(-0.15,0.15)--(-0.35,0.15)--(-0.35,-0.15)--(-0.15,-0.15)--(-0.15,-0.35)--(0.15,-0.35)--(0.15,-0.15)--(0.35,-0.15)--(0.35,0.15)--(0.15,0.15)--(0.15,0.35)--(-0.2,0.35);
\draw[orange,dotted,-latex] (0.1,0.04)--(0.2,0.04);

\draw[teal,dashed,-latex,rotate=90] (0.1,0.05)--(0.2,0.05);

\draw[red,thick] (-0.1,-0.05)--(-0.25,-0.05)--(-0.25,-0.1)--(-0.05,-0.1)--(-0.05,-0.25)--(0.05,-0.25)--(0.05,-0.05)--(0.25,-0.05)--(0.25,0.05)--(0.05,0.05)--(0.05,0.3)--(0.1,0.3)--(0.1,0.1)--(0.3,0.1)--(0.3,-0.1)--(0.1,-0.1)--(0.1,-0.3)--(-0.1,-0.3)--(-0.1,-0.15)--(-0.3,-0.15)--(-0.3,0.05)--(-0.2,0.05);

\draw[red,thick,-latex] (-0.15,-0.05)--(-0.2,-0.05);

\draw[red,thick,latex-] (-0.3,-0.02)--(-0.3,-0.07);

\draw[blue,dash dot, rotate=90] (-0.1,-0.07)--(-0.27,-0.07)--(-0.27,-0.12)--(-0.07,-0.12)--(-0.07,-0.27)--(0.07,-0.27)--(0.07,-0.07)--(0.27,-0.07)--(0.27,0.07)--(0.07,0.07)--(0.07,0.32)--(0.12,0.32)--(0.12,0.12)--(0.32,0.12)--(0.32,-0.12)--(0.12,-0.12)--(0.12,-0.32)--(-0.12,-0.32)--(-0.12,-0.17)--(-0.32,-0.17)--(-0.32,0.07)--(-0.2,0.07);

\draw[blue,dash dot,-latex,rotate=90] 
(-0.13,-0.07)--(-0.19,-0.07);

\draw[blue,dash dot,latex-,rotate=90] (-0.32,-0.04)--(-0.32,-0.09);

\draw[orange,dotted,latex-,rotate=180] (-0.36,-0.02)--(-0.36,-0.06);

\draw[orange,dotted, thick] (0.36,0.15)--(0.36,-0.15)--(0.2,-0.15);

\draw[teal,dashed,latex-,rotate=270] (-0.36,-0.02)--(-0.36,-0.06);

\draw[teal,dashed, thick,rotate=90] (0.36,0.15)--(0.36,-0.15)--(0.2,-0.15);

\node at (0.42,-0.05) {\scriptsize $a_1$};
\node at (0.1,0.4) {\scriptsize $a_2$};
\node at (-0.4,0.05) {\scriptsize $a_3$};
\node at (-0.05,-0.38) {\scriptsize $a_4$};
\node at (1,-0.05) {\scriptsize $\ell_0$};
\node at (0.1,1) {\scriptsize $\ell_1$};
\node at (-1,0.05) {\scriptsize $\ell_2$};
\node at (-0.05,-1) {\scriptsize $\ell_3$};
\end{tikzpicture}
\caption{{\black Sketch of a graph of the adjusted map $h$ on $T_2$ from Lemma~\ref{tunc:3.3}.}}\label{fig1:sym}
\end{figure}

\begin{lem}\label{tunc:3.2}
Let $f \colon T \to T$ be a piecewise linear and topologically exact map {\black on a finite tree $T$}. Then, for each real number $\delta$ with $0 < \delta < 1$, there exists a piecewise linear map $\tilde{f}\colon T \to T$ and a positive real number $\xi < \delta$ such that
\begin{enumerate}
\item $d(\tilde{f},f)<\delta$.
\item For each subcontinuum $A$ in $T$ with $\diam(A) < \xi /5$,
$\diam(\tilde{f}(A)) \geq  2 \mathrm{diam}(A)$. 
\item The map $\tilde{f}$ is topologically exact.
\end{enumerate}
Furthermore, $\tilde{f}=f\circ h$ for some map $h\colon T\to T$.  {\black Moreover, if $T=T_n$ and $f$ is invariant under rigid rotation $R_n$, then $h$ and $\tilde f$ are invariant under rigid rotation $R_n$.}
\end{lem}

\begin{lem}\label{tunc:3.3}
Let $\tilde{f} \colon T \to T$ be a topologically exact and piecewise linear map {\black on a finite tree $T$} which satisfies the following condition: there exists a real number $0 < \beta < 1$ such that
\begin{itemize}
\item[$(+)_\beta$] for each subcontinuum $A$ with $\diam(A) \leq \beta$, $\diam(\tilde{f}(A)) \geq 2 \diam(A)$.
\end{itemize}
Then, for each real number $\delta$ with $0<\delta<1$, there exist a piecewise linear and topologically exact map $F \colon T\to T$, a positive real number $\xi < \delta$, and an integer $n > 0$ such that
\begin{enumerate}
	\item $d(F,\tilde{f})<\delta$.
	\item $F^n(A) = T$ for every subcontinuum $A$ of $T$ with $\diam(A) \geq  \xi /5$.
\item The map $F$ satisfies condition $(+)_{\xi/5}$.
\item Each map ${\black \tau} \colon [0, 1] \to T$ is $(F^n, \delta)$-crooked.	 
\end{enumerate}
Furthermore, $F=\tilde{f}\circ g$ for some map $g\colon T\to T$. {\black Moreover, if $T=T_n$ and $f$ is invariant under rigid rotation $R_n$, then $h$ and $\tilde f$ are invariant under rigid rotation $R_n$.}
\end{lem}

\begin{remark}\label{rem:const}
If we start with a piecewise linear and topologically exact map $f$ {\black on a finite tree $T$}, we can apply Lemma~\ref{tunc:3.3}. Indeed, we must first apply Lemma~\ref{tunc:3.2} and then Lemma~\ref{tunc:3.3}.
This way we obtain the map $F=f\circ h\circ g$, where $h$ comes from Lemma~\ref{tunc:3.2} and $g$ from Lemma~\ref{tunc:3.3} applied to $f\circ h$.\\
In the next section we will be concerned with maps on {\black $2^n$-ods $T_n$ which are invariant under $R_n$}.
%Therefore, if $T=T_n$ is an $2^n$-odd, by the construction, we may additionally require that $h,g$ commute with rotations of $T_n$, and thus if the same holds also for the map $f$ that we started with, it follows that $F$ commutes with  rotations of $T_n$ as well.
\end{remark}

\section{Proof of Theorem~\ref{odoOld}}\label{sec:odo}

{\black We will prove an extended version of Theorem~\ref{odoOld} which is required for our particular purposes.

\begin{theorem}\label{odo}
For every $n\geq 1$ let $S_n$ be homeomorphic to $\mathbb{D}^6$ and let $\Ps_n\subset S_n$ be a pseudo-arc.  For every $n\geq 1$ there exists a set $\mathbb{F}_n\subset S_n$ homeomorphic to $\mathbb{D}^4$ such that $\mathcal{P}_n\cap \mathbb{F}_n$ is a single point denoted by $\hat 0_n$ and 
%Moreover, 
for each $n\geq 1$ there are maps $\hat R_n:S_n\to S_n$ and $\varphi_{\infty, n}:S_{n+1}\to S_n$ such that
\begin{enumerate}
\item $\hat R_n$ is identity on $\mathbb{F}_n$ and $2$-periodic on $S_n\setminus \mathbb{F}_n$.
\item $\varphi_{\infty,n}$ is a $2$-to-$1$ branched covering with branch set $\mathbb{F}_n$, $\hat R_n\circ \varphi_{\infty,n}=\varphi_{\infty, n+1}\circ \hat R_{n+1}$ and $\varphi_{\infty, n}(\mathcal{P}_{n+1})=\mathcal{P}_n$.
\end{enumerate}

Let $\Ps_\infty=\varprojlim (\Ps_n,\varphi_{\infty,n})$ be an inverse limit of pseudo-arcs and let $\mainmap:\mathbb D_\infty\to \mathbb D_\infty$, where
$\mathbb{D}_{\infty}=\prod_{i=1}^\infty S_i$, be given as follows
\begin{equation}\label{mainmap}\mainmap=(\hat R_1,\hat R_2, \hat R_3,\ldots).
\end{equation}

% and a pseudo-arc homeomorphism $\mainmap:\Ps_\infty\to \Ps_\infty$ and $0_{\infty}\in \Ps_\infty$ such that 
Let $\mathbb{P}=(q_n:n\geq 1)$ be a sequence of primes and $\phi_{\mathbb{P}}:\Lambda\to \Lambda$ be the $\mathbb{P}$-adic odometer.
Then we have the following statements.

\begin{enumerate}
    \item $\Per(\mainmap)=\{0_{\infty}\}$, where $0_{\infty}=(\hat 0_1,\hat 0_2,\ldots)$, 
    \item $\mainmap|_{\omega(\mainmap,x)}$ is conjugate to $\phi_{\mathbb{P}}$ for every $x\in \Ps_\infty\setminus\{0_{\infty}\}$, and
    \item $x\in \omega(\mainmap,x)$ for every $x\in \Ps_\infty$,
 \end{enumerate}

% Let $\Ps_\infty=\varprojlim (\Ps_n,\varphi_{\infty,n})$ be an inverse limit of pseudo-arcs, with the bonding maps $\varphi_{\infty,n}\colon {\black \Ps_{n+1}\to \Ps_n}$ being {\black restrictions of the} branched $q_{n}\text{-to-}1$ covers of 2-disks $S_n$, with one branch point $0_n$. 

Furthermore, for every $\eps>0$ there exists an $\eps$-cover $\mathcal{W}$ of $\mathcal{P}_n$, $\bigcup_{W\in \mathcal{W}}\mathrm{int}(W)\supset \mathcal{P}_n$, by tamely embedded $6$-cubes
such that if $Z\subset \mathcal{P}_n$ is a continuum and if $d_H(Z,\mathbb{F}_n)>\eps$ then there exists a chain subcover $\mathcal{W}'\subset \mathcal{W}$ of $Z$. If we fix $p\in \mathcal{P}_n$ at start, then we may require that $p$ belongs only to one element of the cover $\mathcal{W}$.
\end{theorem}

\begin{remark}
To simplify the notation, in particular, when using the composition of the bonding maps, we will write $\z_{\infty}$ instead of $\z_{\infty,n}$. In such cases, the domain $S_{n+1}$ of $\z_{\infty}$ will be clear from the context.
\end{remark}
}
\begin{proof}[Proof of Theorem \ref{odo}]
For clarity of exposition, we shall carry out our construction for the $2$-adic odometer; i.e. $q_n=2$ for all $n\in \mathbb{N}$. The construction for other odometers is analogous. 

The following claim is implicitly contained in the proof of \cite[Theorem, p.336, $k=1$]{LewisGroups}{\black, see also the construction of maps $g_j$ in \cite[Theorem 2, p.82]{LewisPeriodic}, property $h^{i}_n\circ f_i=f_i\circ h^{i+1}_n$ and the first line on page 83 of \cite{LewisPeriodic}, where in our construction we start with $M> 2^{N}$.}

\begin{claim}\label{lem:arclike}
For any integer ${\black N} > 0$ and any $\gamma>0$, there exist {\black a sequence of} map{\black s} ${\black\tilde q}_{\gamma,n}\colon T_n\to [0,1]$,
{\black and} ${\black\tilde f}_{\gamma,n}:T_n\to T_n$, {\black $n=1,\ldots, N$} so that ${\black\tilde f}_{\gamma,n}$ is piecewise linear and topologically exact, invariant under rigid rotation of branches of $T_n$, and such that ${\black\tilde f}_{\gamma,n-1}\circ \z=\z\circ {\black\tilde f}_{\gamma,n}$  and if ${\black\tilde q}_{\gamma,n}(x)={\black\tilde q}_{\gamma,n}(y)$ then $d({\black\tilde f}_{\gamma,n}(x),{\black\tilde f}_{\gamma,n}(y))<\gamma$.
\end{claim}

Now we will prove the core claim of this proof, which uses also the claim above.

\begin{claim}\label{PSS}
There exists a family of maps 
 $\{g_{i,n}:T_n\to T_n\}_{i\in \N}$ and a sequence $(\gamma_i)_{i\in \N}$ of positive real numbers decreasing to $0$ such that if we set $f_{i,n}=\hat f_{\gamma_i,n}\circ g_{i,n}$ for all positive integers $i,n$
then $\mathcal{P}_{n}=\varprojlim(T_{n}, f_{i,n})$ is a pseudo-arc for each $n$. 
\end{claim}
\begin{proof}[Proof of Claim \ref{PSS}]
The construction will be ``diagonal''. 
Let $\{\eps_k\}_{k=0}^\infty$ be any sequence of real numbers decreasing to $0$. Later, while constructing maps $f_{k,n}$, we will speed up convergence of the sequence $\eps_k$ to satisfy assumptions of Lemma~\ref{lem:Brown}. 
Set $\gamma_0=\eps_0$. Let
$\hat{f}_{0,1}={\black\tilde f}_{\gamma_0,1}$ and $q_{0,1}={\black\tilde q}_{\gamma_0,1}$ be provided by {\black Claim}~\ref{lem:arclike}
for $\gamma=\gamma_0$. Since $\hat{f}_{0,1}$ is topologically exact
we can find an integer $s_0>0$ such that $\hat f_{0,1}^{s_0}(\ell_0)=\hat f_{0,1}^{s_0}(\ell_1)=T_1$. 
Next, applying Remark~\ref{rem:const} to $\hat f_{0,1}^{s_0}$ we can find a map $h_{0,1}$ and a positive integer $j_0$ such that if we denote $F_{0,1}=\hat f_{0,1}^{s_0}\circ h_{0,1}$, then each map ${\black \tau} \colon [0, 1] \to T_n$ is $(F^{j_0}_{0,1}, \eps_0)$-crooked.
Denote $f_{0,1}=F^{j_0}_{0,1}$ and observe that we have $f_{0,1}=\hat{f}_{0,1}\circ g_{0,1}$, where $g_{0,1}=\hat{f}_{0,1}^{s_0-1}\circ h_{0,1}\circ F^{j_0-1}_{0,1}$.

Next, assume that the maps $f_{i,n},q_{i,n},h_{i,n},g_{i,n}$ have already been defined for $n=1,\ldots,k$ and $0\leq i \leq k-1$. We adjust $\eps_{k+1}$, if necessary, to be so small that 
$$
\eps_{k+1}<\min_{n\leq k;i<k}L(2^{-k-1}\diam (T_n), f_{i,n}\circ f_{i+1,n}\circ \ldots \circ f_{k-1,n}).
$$
For each $0\leq j \leq  k$ we choose $f_{j,k+1}$, a lift of $f_{j,k}$ through the branched cover $\z$, and similarly we define maps $g_{j,k+1}, h_{j,k+1}$. {\black To choose the lift $f_{j,k+1}$ we use the fact that $f_{j,k}$ is constructed independently on branches $\ell_i$, invariant under $R_k$, and also the fact that for the consecutive pieces of linearity of $f_{j,k}$ in $\ell_i$ their consecutive images via $f_{j,k}$ move between branches 
$\ell_l$ by one index %$i$ 
forward or backward $\Mod{2^n}$,
%visit only neighboring branches, %in the ordering given by the indices on branches, i.e. 
%moving back and forth on branches one by one, 
see Fig.~\ref{fig1:sym}. When defining $f_{j,k+1}$ we also move between branches 
$\ell_l$ by one index %$i$ 
forward or backward $\Mod{2^n}$,
following the above pattern determined by $f_{j,k}$. This allows us to find proper ordering of branches when lifting.}

Let $\gamma_{k+1}<\eps_{k+1}$ be such that, for each $n=1,\ldots, k+1$, if $x,y\in T_n$ are such that $d(x,y)<\gamma_{k+1}$,
then for each $0\leq j< k$
we have
\begin{equation}\label{eq:estpart1}
d(f_{j,n}\circ \ldots \circ f_{k-1,n}(x),f_{j,n}\circ \ldots \circ f_{k-1,n}(y))<\eps_{k+1}
\end{equation}
and {\black for $0< j< k$ we have}
\begin{eqnarray}
    \label{eq:estpart2}
&{\black d(g_{j,n}\circ f_{j,n}\circ \ldots \circ f_{k-1,n}(x),g_{j,n}\circ f_{j,n}\circ \ldots \circ f_{k-1,n}(y))
<\eps_{k+1},}\\
&{\black d(g_{j,n}(x),g_{j,n}(y))
<\eps_{k+1}}\nonumber
\end{eqnarray}

Apply {\black Claim}~\ref{lem:arclike} to $\gamma_{k+1}$ to get maps
$\hat{f}_{k,n}={\black\tilde f}_{\gamma_{k+1},n}$ and $q_{k,n}={\black\tilde q}_{\gamma_{k+1},n}$ for $n=1,\ldots, k+1$. By topological exactness, we select a positive integer $s_{k}$
such that $\hat f_{k,n}^{s_k}(\ell_i)=T_n$ for $i=0,\ldots,n-1$.
Repeating the argument from above, by Remark~\ref{rem:const} we can find a map $h_{k,k+1}$ and a positive integer $j_{k+1}$ such that for the map $F_{k,k+1}=\hat f_{k,k+1}^{s_k}\circ h_{k,k+1}$ each map ${\black \tau} \colon [0, 1] \to T_{1}$ is $(F^{j_{k+1}}_{k,k+1}, \eps_{k+1})$-crooked. 
For each $j=0,1,\ldots,k{\black-1}$ we let $h_{k,k-j}$ to be a projection of $h_{k,k+1-j}$ through the branched cover $\z$, which is well defined since $h_{k,k+1}$ is invariant under rotation of branches of $T_{k+1}$ (see Remark~\ref{rem:const}).
Denote $F_{k,n}=\hat f_{k,n}^{s_k}\circ h_{k,n}$ for $n\leq k+1$ and
observe that we can represent $f_{k,n}=\hat{f}_{k,n}\circ g_{k,n}$, where $g_{k,n}= \hat f_{k,n}^{s_k-1}\circ h_{k,n}\circ F^{j_{k+1}-1}_{k,n}$.
Since $\z$ does not increase distance, and for every $n\leq k$ each ${\black \tau}  \colon [0,1]\to T_n$ can be represented as ${\black \tau} =\z \circ{\black \tau} '$ for some ${\black \tau} '\colon [0,1]\to T_{n+1}$, we see that each such map ${\black \tau} $ is $(F^{j_{k+1}}_{k,n}, \eps_{k+1})$-crooked.

By the construction, Lemma~\ref{lem:Brown} ensures that each $\mathcal{P}_n:=\varprojlim(T_n,f_{i,n})$ is hereditarily indecomposable. We also have an $2^n$-periodic homeomorphism $H\colon \mathcal{P}_n \to \mathcal{P}_n$ induced by $2^n$-periodic rotations of $T_n$ on each coordinate, where $\mathcal{P}_n=\varprojlim(T_n,f_{i,n})$.

It remains to prove that each $\mathcal{P}_{n}$ is arc-like. 
In order to prove it we note that $\mathcal{P}_n=\varprojlim(T_n,r_{j,n})$
where $r_{2j,n}=\hat f_{j,n}$ and $r_{2j+1,n}=g_{j,n}$ for $j=1,2,\ldots$.
But then, if $\pi_{j,n}$ denotes the $j$-th coordinate projection from $\mathcal{P}_n$ to $T_n$, we have a well defined projection $p_{j,n}=q_{j,n}\circ \pi_{2j,n}\colon \mathcal{P}_n\to [0,1]$
with the property that if $p_{k,n}(\hat x)=p_{k,n}(\hat y)$ for some $k$ and $\hat{x},\hat{y}\in \mathcal{P}_n$, then 
$$
d(r_{2k,n}(\pi_{2k,n}(\hat x)),r_{2k,n}(\pi_{2k,n}(\hat y)))=d(\hat f_{k,n}(\pi_{2k,n}(\hat x)),\hat f_{k,n}(\pi_{2k,n}(\hat y)))<\gamma_k
$$
which by \eqref{eq:estpart1} and \eqref{eq:estpart2} implies
that for $j=0,\ldots, 2k-1$ we have
$$
d(\pi_{j,n}(\hat{x}),\pi_{j,n}(\hat{y}))<\eps_k.
$$
But then
$$
d(\hat{x},\hat{y})\leq \eps_k\sum_{j=0}^{2k-1}2^{-j}+\sum_{j=2k}^{\infty}2^{-j}\leq 2\eps_k+2^{-2k+1}. 
$$
Therefore, for every $\eps>0$ there exists a natural number $k$ so that $p_{k,n}$ is an $\eps$-map.
This shows that for all $n\geq 1$ the space $\mathcal{P}_n$ is arc-like and thus $\mathcal{P}_n$ is the pseudo-arc, which proves the claim.
\end{proof}
For what follows we refer the reader to the Diagram~\ref{megadiagram}.
 {\black Note that} $f_{i,n}\circ \z=\z\circ f_{i,n+1}$ for all $i\in \mathbb{N}$ and $n\geq 1$.
{\black To see this, fix any $i\in \N$. In the construction above, the map $f_{i,n+1}$ was obtained by Claim~\ref{lem:arclike} and that is a consequence of $f_{i,n}\circ \z=\z\circ f_{i,n+1}$ for $n\leq i+1$. But for $n\geq i+1$, the map $f_{i,n+1}$ was obtained as a lift of $\varphi$.}
Since $T_1$ is an arc, by \cite{Barge2}  the maps $(f_{i,1}:i\in\mathbb{N})$ extend to near-homeomorphisms $(\bar f_{i,1}:i\in\mathbb{N})$ of the topological disk $\mathbb{D}^2$ (see also \cite{Youngs}). {\black We could possibly extend each $f_{i,1}$ to a map $\bar f_{i,1}$ of $\mathbb{D}^2$ which coincides on $T_n$ with $f_{i,1}$, however showing that such a map is a near-homeomorphism may be challenging, and in fact many maps on finite graphs do not extend to near-homeomorphsims \cite{Kasse} of the plane. Nonetheless, for our purpose it suffices to show how to extend our maps to near-homeomorphisms of $\mathbb{D}^3$.}

{\black
\begin{claim}\label{clm:extension}
The maps $(f_{i,n}:i\in\mathbb{N},n\geq 1)$ extend to near-homeomorphisms $(\bar f_{i,n}:i\in\mathbb{N},n\geq 1)$ of $\mathbb{D}^3$ and each $\bar f_{i,n}$ commutes with $R_n$.
\end{claim}
\begin{proof}
Fix $i\in \N$ and $n\geq 1$. By the construction $f_{i,n}\circ \z=\z\circ f_{i,n+1}$. First, we will define maps $\bar f_{k,k+1}$ and then use lifts and projections $\varphi$ to define maps $\bar f_{i,n}$, the same as it was done when we defined maps $f_{i,n}$. These lifts and projections will be well defined due to two features in the construction. First, the homeomorphism $\tilde f$ defined below will act in the same way on all radial lines; that is, $\tilde f((se^{i\phi},t))=(y e^{i\phi},t')$ for every $\phi$. Second, the smash map $\zeta$ is completely determined through its action on the $\zeta$-invariant set $\{(se^{i\phi},t): t\in [-1,1], s\in [0,1], \phi \in [0,2\pi/n]\}$, see Fig.~\ref{fig2:smash}.
 This allows us to provide exact formulas for the lifted/projected maps; since we are not using these formulas explicitly, we leave the details to the reader.

Fix $k\geq 0$ and for simplicity of notation we put $f=f_{k,k+1}$. 
We define the smash map $\zeta:\mathbb{D}^3\to \mathbb{D}^3$.  The construction will be in two steps: first we define the homeomorphism $\tilde f: \mathbb{D}^3\to \mathbb{D}^3$ and then compose it with $\zeta$ to obtain $\bar f_{k,k+1}$.
We will describe the construction when $n> 1$. For $n=1$ the construction is analogous but more technical to describe, since, in what follows, we would need to use half-circles instead of triangles. We leave the details of the particular case $n=1$ to the reader. Alternatively, the standard technique of \cite{Barge2} can be used, drawing a graph of $f_{i,1}$ in $\mathbb{D}^2$ symmetrically in $\bar 0$ to preserve rotation by $R_1$.

We consider triangles with vertices $(\frac{2}{3}w_j,0)$, $(\frac{2}{3}w_{j+1},0)$, and $(\bar 0,0)$. We project each point of the triangle along the lines defined by normal vectors opposite to the edge between $(\frac{2}{3}w_j,0)$ and $(\frac{2}{3}w_{j+1},0)$. So far, the map $\zeta$ has been defined for the polygon with vertices $(\frac{2}{3}w_1,0), (\frac{2}{3}w_{2},0), \ldots, (\frac{2}{3}w_{2^n},0)$. For points $(x,t)$, with $x$ in the polygon and $|t|\leq \frac{2}{3}$, we define $\zeta((x,t))=\zeta((x,0))$. The map $\zeta$ defined up to now is a near-homeomorphism that commutes with $R_n$, so we can extend it to a near-homeomorphism on $\mathbb{D}^3$, which commutes with $R_n$ and is an identity on the boundary.

We fix $\delta>0$ such that $f|_{\ell_0}\cap B((\bar 0,0),\delta)$ is linear. If we put $b=\ell_0\cap \partial B((\bar 0,0),\delta)$, then $f(b)\in \ell_0\setminus B((\bar 0,0),\delta)$. 
If $x\in B((\bar 0,0),\delta)$ or $f(x)\notin B((\bar 0,0),\delta)$ we define $g(x)=(y_1,|x_1|)$, where $f(x)=(y_1,0)$ and $x=(x_1,0)$. In other words, we lift the graph of $f$, see Figure~\ref{fig2:smash}.

If $f(x)\in B((\bar 0,0),\delta)$ or $x\notin B((\bar 0,0),\delta)$, then we find $y_1,y_2\in \ell_0$ such that $x\in (y_1,y_2)$ and $f(y_1),f(y_2)\in \partial B((\bar 0,0),\delta)$ and $f((y_1,y_2))\cap \partial B((\bar 0,0),\delta)=\emptyset$.\\
We have two cases; let us first assume that $f((y_1,y_2))\cap \ell_j\neq \emptyset$ and $f((y_1,y_2))\cap \ell_{j+1\Mod{2^n}}\neq \emptyset$ for some $0\leq j< 2^n$.
Then there exists a unique $z_1\in \mathbb{C}$, $|z_1|=\delta$, such that $z=(z_1,0)$ is in the triangle with vertices $(\frac{2}{3}w_j,0)$, $(\frac{2}{3}w_{j+ 1},0)$ and $(\bar 0,0)$, and such that $\zeta(z)=f(x)$.
We define $\tilde f(x)=(z_1,|x_1|)$. The case where $f((y_1,y_2))\cap \ell_j\neq \emptyset$ and $f((y_1,y_2))\cap \ell_{j-1\Mod{2^n}}\neq \emptyset$ for some $0\leq j< 2^n$ is symmetric to the above.\\
Now, let us assume that $f((y_1,y_2))\in \ell_j$ for some $0\leq j< 2^n$. Then we again consider the triangle with vertices $(\frac{2}{3}w_j,0)$, $(\frac{2}{3}w_{j+ 1},0)$ and $(\bar 0,0)$ and repeat the previous construction.\\
If $x\in \ell_j$ for $1\leq j< 2^n$, then we put $\tilde f(x)=R_n^{j}(\tilde f(R_n^{-j}(x)))$.
Map $\tilde f: T_n\to \mathbb{D}^3$ is a homeomorphism on its image, it commutes with $R_n$ and $\zeta\circ \tilde f=f$.

In the interval between pairs of points:
$(\frac{1}{3}w_1,\frac{2}{3})$ and $(\bar 0, \frac{2}{3})$; $(\bar 0, \frac{2}{3})$ and $(\bar 0,-\frac{2}{3})$; $(\frac{1}{3}w_1,-\frac{2}{3})$ and $(\bar 0, -\frac{2}{3})$
the map is identity. We have $\tilde f((\frac{1}{3}w_1,0))=(\frac{1}{3}w_1,\frac{1}{3})$. The points in the interval $(\frac{1}{3}w_1,t)$, $t\in [0,\frac{2}{3}]$ and in the interval $(\frac{1}{3}w_1,t)$, $t\in [-\frac{2}{3},0]$ are mapped linearly to the interval spanned by the values at the endpoints of these intervals.

It is not hard to find a continuous map $\eta\colon \mathbb{D}^3\to \mathbb{D}^3$ such that: 
\begin{itemize}
\item[(i)] for every $r\in [0,1]$, $t\in [-1,1]$ the map $\eta$ 
restricted to the circle $\{(x,t) : |x|=r\}\subset \mathbb{D}^3$ is a rotation; 
\item[(ii)] for $|t|\geq 2/3$ or $r\geq 1/3$ the map $\eta$ is identity on the associated circle; 
\item[(iii)] for any $z\in \ell_0$ the point $(y,t):=\eta (\tilde f (z))$ satisfies $(y,0)\in \ell_0$. 
\end{itemize}
For what follows, we refer the reader to Figure~\ref{fig:liftedf}.
Informally, the map $\eta$ ``unwinds'' the image of the chord $\ell_0$ by $\tilde f$ again to the square $Q$ spanned by points 
$(\frac{1}{3}w_1,2/3)$, $(\frac{1}{3}w_1,-2/3)$, $(\bar 0,-2/3)$ and $(\bar 0, 2/3)$. Now, if we consider the $\theta$-curve $\Gamma$ (i.e. a union of two Jordan curves) defined by $\partial Q \cup \ell_0$, then $(\eta \circ \tilde f)(\Gamma)\subset Q$ is also a $\theta$-curve defined in fact by $\partial Q \cup (\eta \circ \tilde f)(\ell_0)$. Then we can extend $\eta \circ \tilde f$ to a homeomorphism $\beta \colon Q\to Q$. But since $\beta$ is the identity in the interval $(\bar 0, t)$, $|t|\leq 2/3$ we can extend it further to a homeomorphism of the square $Q'=\{((s,0),t) :s,t\leq 1\}$ which is identity on the boundary of $Q'$. Observe that if we define $\tilde f$ on $Q'$ by putting $\tilde f = \eta^{-1}\circ \beta$ then it is well defined, that it, this new formula coincides on points previously defined. Then we can extend this map to a map on $\mathbb{D}^3$ by defining $\tilde f((se^{i\phi},t))=(y e^{i\phi},t')$ for each $\phi\in [0,2\pi)$, $s\in [0,1]$ and $-1\leq t\leq 1$, where $(y,t'):=\tilde f( (s,0),t))$. It is clear that the map $\tilde f$ defined as above is commuting with the rigid rotation $R_n$. We put $\bar f_{k,k+1}=\zeta\circ \tilde f$. Since $\zeta$ is commuting with $R_n$, also $\bar f_{k,k+1}$ is commuting with $R_n$ and it is also clear that $\bar f_{k,k+1}$ is a near-homeomorphism.

\end{proof}
}

\begin{figure}[!ht]
\centering
\begin{tikzpicture}[scale=2]
\draw (-1.5,0)--(1.5,0);
\draw (0,-1.5)--(0,1.5);
%\draw  (0,0) circle[radius=10mm];
\draw (1.2,0)--(0,-1.2)--(-1.2,0)--(0,1.2)--(1.2,0);
\begin{scope}[decoration={
    markings,
    mark=at position 0.5 with {\arrow{<}}}
    ] 
\draw[postaction={decorate}] (0.6,0.6)--(0.3,0.3); 
\draw[postaction={decorate}] (0.5,0.7)--(0.1,0.3); 
\draw[postaction={decorate}] (0.4,0.8)--(0.1,0.5); 
\draw[postaction={decorate}] (0.3,0.9)--(0.1,0.7); 
\draw[postaction={decorate}] (0.2,1)--(0.1,0.9); 
\draw[postaction={decorate}] (0.7,0.5)--(0.3,0.1); 
\draw[postaction={decorate}] (0.8,0.4)--(0.5,0.1); 
\draw[postaction={decorate}] (0.9,0.3)--(0.7,0.1); 
\draw[postaction={decorate}] (1,0.2)--(0.9,0.1);
\end{scope}
\draw[red,>->](1.3,0.1)--(0.3,0.1)--(0.3,0.2)--(0.5,0.2)--(0.5,0.3)--(0.3,0.3)--(0.3,0.5)--(0.2,0.5)--(0.2,0.3)--(0.1,0.3)--(0.1,1.3);

\draw[red,<-<](0.2,1.3)--(0.2,1.1)--(0.6,0.7)--(0.55,0.85)--(0.85,0.55)--(0.7,0.6)--(1.1,0.2)--(1.3,0.2);

\begin{scope}[decoration={
    markings,
    mark=at position 0.5 with {\arrow{<}}}
    ] 
\draw[postaction={decorate}] (-0.6,-0.6)--(-0.3,-0.3); 
\draw[postaction={decorate}] (-0.5,-0.7)--(-0.1,-0.3); 
\draw[postaction={decorate}] (-0.4,-0.8)--(-0.1,-0.5); 
\draw[postaction={decorate}] (-0.3,-0.9)--(-0.1,-0.7); 
\draw[postaction={decorate}] (-0.2,-1)--(-0.1,-0.9); 
\draw[postaction={decorate}] (-0.7,-0.5)--(-0.3,-0.1); 
\draw[postaction={decorate}] (-0.8,-0.4)--(-0.5,-0.1); 
\draw[postaction={decorate}] (-0.9,-0.3)--(-0.7,-0.1); 
\draw[postaction={decorate}] (-1,-0.2)--(-0.9,-0.1);
\end{scope}
\draw[blue,>->](-1.3,-0.1)--(-0.3,-0.1)--(-0.3,-0.2)--(-0.5,-0.2)--(-0.5,-0.3)--(-0.3,-0.3)--(-0.3,-0.5)--(-0.2,-0.5)--(-0.2,-0.3)--(-0.1,-0.3)--(-0.1,-1.3);

\draw[blue,<-<](-0.2,-1.3)--(-0.2,-1.1)--(-0.6,-0.7)--(-0.55,-0.85)--(-0.85,-0.55)--(-0.7,-0.6)--(-1.1,-0.2)--(-1.3,-0.2);

\begin{scope}[decoration={
    markings,
    mark=at position 0.5 with {\arrow{<}}}
    ] 
\draw[postaction={decorate}] (-0.6,0.6)--(-0.3,0.3); 
\draw[postaction={decorate}] (-0.5,0.7)--(-0.1,0.3); 
\draw[postaction={decorate}] (-0.4,0.8)--(-0.1,0.5); 
\draw[postaction={decorate}] (-0.3,0.9)--(-0.1,0.7); 
\draw[postaction={decorate}] (-0.2,1)--(-0.1,0.9); 
\draw[postaction={decorate}] (-0.7,0.5)--(-0.3,0.1); 
\draw[postaction={decorate}] (-0.8,0.4)--(-0.5,0.1); 
\draw[postaction={decorate}] (-0.9,0.3)--(-0.7,0.1); 
\draw[postaction={decorate}] (-1,0.2)--(-0.9,0.1);
\end{scope}
\draw[teal,<-<](-1.3,0.1)--(-0.3,0.1)--(-0.3,0.2)--(-0.5,0.2)--(-0.5,0.3)--(-0.3,0.3)--(-0.3,0.5)--(-0.2,0.5)--(-0.2,0.3)--(-0.1,0.3)--(-0.1,1.3);

\draw[teal,>->](-0.2,1.3)--(-0.2,1.1)--(-0.6,0.7)--(-0.55,0.85)--(-0.85,0.55)--(-0.7,0.6)--(-1.1,0.2)--(-1.3,0.2);

\begin{scope}[decoration={
    markings,
    mark=at position 0.5 with {\arrow{<}}}
    ] 
\draw[postaction={decorate}] (0.6,-0.6)--(0.3,-0.3); 
\draw[postaction={decorate}] (0.5,-0.7)--(0.1,-0.3); 
\draw[postaction={decorate}] (0.4,-0.8)--(0.1,-0.5); 
\draw[postaction={decorate}] (0.3,-0.9)--(0.1,-0.7); 
\draw[postaction={decorate}] (0.2,-1)--(0.1,-0.9); 
\draw[postaction={decorate}] (0.7,-0.5)--(0.3,-0.1); 
\draw[postaction={decorate}] (0.8,-0.4)--(0.5,-0.1); 
\draw[postaction={decorate}] (0.9,-0.3)--(0.7,-0.1); 
\draw[postaction={decorate}] (1,-0.2)--(0.9,-0.1);
\end{scope}
\draw[orange,<-<](1.3,-0.1)--(0.3,-0.1)--(0.3,-0.2)--(0.5,-0.2)--(0.5,-0.3)--(0.3,-0.3)--(0.3,-0.5)--(0.2,-0.5)--(0.2,-0.3)--(0.1,-0.3)--(0.1,-1.3);

\draw[orange,<-<](0.2,-1.3)--(0.2,-1.1)--(0.6,-0.7)--(0.55,-0.85)--(0.85,-0.55)--(0.7,-0.6)--(1.1,-0.2)--(1.3,-0.2);

\end{tikzpicture}
\hspace{0.2cm}
\begin{tikzpicture}[scale=2]
\draw(-1.5,-1.5)--(-1.5,1.5)--(1.5,1.5)--(1.5,-1.5)--(-1.5,-1.5);
\draw (-1,0)--(1,0);
\draw (0,-1)--(0,1);
\draw (1.2,0)--(0,-1.2)--(-1.2,0)--(0,1.2)--(1.2,0);
\draw[-latex] (0.55,0.55)--(0.1,0.1); 
\draw[-latex] (0.45,0.65)--(0.1,0.3); 
\draw[-latex] (0.35,0.75)--(0.1,0.5); 
\draw[-latex] (0.25,0.85)--(0.1,0.7); 
\draw[-latex] (0.15,0.95)--(0.1,0.9); 
\draw[-latex] (0.65,0.45)--(0.3,0.1); 
\draw[-latex] (0.75,0.35)--(0.5,0.1); 
\draw[-latex] (0.85,0.25)--(0.7,0.1); 
\draw[-latex] (0.95,0.15)--(0.9,0.1);

\draw[-latex] (-0.55,-0.55)--(-0.1,-0.1); 
\draw[-latex] (-0.45,-0.65)--(-0.1,-0.3); 
\draw[-latex] (-0.35,-0.75)--(-0.1,-0.5); 
\draw[-latex] (-0.25,-0.85)--(-0.1,-0.7); 
\draw[-latex] (-0.15,-0.95)--(-0.1,-0.9); 
\draw[-latex] (-0.65,-0.45)--(-0.3,-0.1); 
\draw[-latex] (-0.75,-0.35)--(-0.5,-0.1); 
\draw[-latex] (-0.85,-0.25)--(-0.7,-0.1); 
\draw[-latex] (-0.95,-0.15)--(-0.9,-0.1);

\draw[-latex] (-0.55,0.55)--(-0.1,0.1); 
\draw[-latex] (-0.45,0.65)--(-0.1,0.3); 
\draw[-latex] (-0.35,0.75)--(-0.1,0.5); 
\draw[-latex] (-0.25,0.85)--(-0.1,0.7); 
\draw[-latex] (-0.15,0.95)--(-0.1,0.9); 
\draw[-latex] (-0.65,0.45)--(-0.3,0.1); 
\draw[-latex] (-0.75,0.35)--(-0.5,0.1); 
\draw[-latex] (-0.85,0.25)--(-0.7,0.1); 
\draw[-latex] (-0.95,0.15)--(-0.9,0.1);

\draw[-latex] (0.55,-0.55)--(0.1,-0.1); 
\draw[-latex] (0.45,-0.65)--(0.1,-0.3); 
\draw[-latex] (0.35,-0.75)--(0.1,-0.5); 
\draw[-latex] (0.25,-0.85)--(0.1,-0.7); 
\draw[-latex] (0.15,-0.95)--(0.1,-0.9); 
\draw[-latex] (0.65,-0.45)--(0.3,-0.1); 
\draw[-latex] (0.75,-0.35)--(0.5,-0.1); 
\draw[-latex] (0.85,-0.25)--(0.7,-0.1); 
\draw[-latex] (0.95,-0.15)--(0.9,-0.1);
\end{tikzpicture}
\caption{{\black A sketch of how the graph of $f$ on $T^2$ is transformed to the graph of $\bar f|_{T_2}$ (left) and the smash map $\zeta$ (right). }}\label{fig2:smash}
\end{figure}

\begin{figure}[!ht]
\centering
\begin{tikzpicture}[scale=2]
\draw (-1,1)--(0,1)--(0,-1)--(-1,-1)--(-1,1);
\draw (-2/3,2/3)--(0,2/3)--(0,-2/3)--(-2/3,-2/3)--(-2/3,2/3);
\draw[thick] (-2/3,0)--(2/3,0);
\node at (0.4,0.15) {\bf \scriptsize $T_2$};
\draw[thick] (-1/2,-1/2)--(1/2,1/2);
\node at (0.04,0.12) {\scriptsize $\bar 0$};
\node at (-0.9,0.9) {\scriptsize $Q'$};
\node at (-0.08,-0.6) {\scriptsize $Q$};
\draw[thick,teal] (-2/3,2/3)--(0,2/3)--(0,-2/3)--(-2/3,-2/3);
\draw[violet,thick] (0,0)--(-1/9,1/36)--(-1/18,2/36)--(-3/18,3/36)--(-1/18,4/36)--(-6/18,5/36)--(-3/18,6/36)--(-9/18,7/36)--(-2/18,8/36)--(-2/3,1/3);

\draw[red,thick] (-2/3,1/3)--(-2/3,2/3);
\draw[<-] (-0.2,0.3)--(0.3,0.5);
\node at (0.3,0.55) {\scriptsize $\beta(\ell_0)$};
\draw[violet, thick,dashed] (-2/3,0)--(0,0);
\draw[red, thick,dashed] (-2/3,0)--(-2/3,1/3);
\draw[dotted] (1,1)--(0,1)--(0,-1)--(1,-1)--(1,1);
\draw[dotted] (3/4,7/4)--(3/4,-1/4)--(-3/4,-7/4)--(-3/4,1/4)--(3/4,7/4);

\end{tikzpicture}

\caption{{\black Sketch of construction of near-homeomorphism $\bar f_{k,k+1}$ from the proof of Claim~\ref{clm:extension}}.}\label{fig:liftedf}
\end{figure}

{\black We extend the maps $\bar f_{i,n}, \varphi$ and $R_n$ of $\mathbb{D}^3$ to $\mathbb{D}^{5}$  by adding the identity map on the additional coordinate. We can view $T_n\subset \mathbb{D}^{5}$ by adding the additional coordinate $\{0\}$ to each point in $T_n$.
Consequently, each $\mathcal{P}_{n}$ is embedded in {\black $\mathbb{D}^5$} by Brown's theorem \cite{Brown1} and if we put $
\mathbb{F}_n=\{\bar 0\}\times [-1,1]^3$, then $T_n\cap \mathbb{F}_n=\{(\bar 0,0,0,0)\}$. Note that $\bar f:\mathbb{F}_n\to \mathbb{F}_n$ is a near-homeomorphism on $\mathbb{F}_n$. That is so, since $\tilde f$ from the proof of Claim~\ref{clm:extension} is identity on $\{\bar 0\}\times [-1,1]$. Therefore, by Brown's theorem \cite{Brown1} $\mathbb{F}_n$ is homeomorphic to an arc $\underleftarrow{\lim}(\mathbb{F}_n,f_{i,n})$. 
Let $\mathbb{F}_n=\underleftarrow{\lim}(\mathbb{F}_n,f_{i,n})$.
Note that $\underleftarrow{\lim}(T_n,f_{i,n})\cap \mathbb{F}_n=\{(\bar 0,0,0,0)\}^{\infty}$. Thus, $\mathbb{F}_n$ is embedded in $\mathbb{D}^5$ in such a way that $\mathbb{F}_n \cap \mathcal{P}_n$ is one point. Furthermore, $\z_{\infty,n}$ is one-to-one on $\mathbb{F}_{n+1}$ to $\mathbb{F}_{n}$ and $2-1$ everywhere else; clearly each $\mathbb{F}_{n}$ is nowhere dense in $\mathbb{D}^5$. }

Now let the homeomorphisms $\hat{R}_n:S_n\to S_n$, where $S_n=\varprojlim(\mathbb{D}^{\black 5},{\black \bar f_{i,n}}){\black\times [-1,1]}\supset \mathcal{P}_n$ for each $n\geq 1$, {\black and to have the inclusion well defined we extend the definition of $\mathcal{P}_n=\varprojlim(T_n,f_{i,n})\times \{0\}$}. {\black Let} $\mainmap:\mathbb D_\infty\to \mathbb D_\infty$, where
$\mathbb{D}_{\infty}=\prod_{i=1}^\infty S_i$, be given as follows 
$$
\hat{R}_n=(R_n,R_n,R_n,\ldots),
$$
\begin{equation}\mainmap=(\hat R_1,\hat R_2, \hat R_3,\ldots).
\end{equation}
For every $n\geq 1$ we define $\z_{\infty{\black,n}}:S_{n+1}\to S_n$ by $\z_{\infty{\black,n}}=(\z,\z,\z,\ldots)$ and observe that by definition $\z_{\infty{\black,n}}(\mathcal{P}_{n+1})=\mathcal{P}_n$. 
{\black 
Since $\z$ is a branched covering it follows that $\z_{\infty,n}$ is also a branched covering. 
For each $n\geq 1$ we also need to adjust the definition of $\mathbb{F}_n$. We put $\mathbb{F}_n=\{\bar 0\}\times [-1,1]^4$.} 
Clearly, each point $x\in S_n\setminus {\black \mathbb{F}_n}$ is $2^n$-periodic for $\hat{R}_n$ {\black and any other point is fixed by $\hat{R}_n$. Since $\varphi_{\infty,n}$ is {\black one-to-one} on $\mathbb{F}_{n+1}$, if we put $0_\infty=\prod^{\infty}_{n=1} \mathbb{F}_n$, then $0_{\infty}\subset \mathbb{D}_{\infty}$. Let  $\Ps_\infty=\varprojlim (\Ps_n,\varphi_{\infty,n})$ and note that $\Ps_\infty\cap 0_{\infty}$ is a single point.}

{\black \begin{claim}\label{lem:embeddP}
 For every $\eps>0$ there exists an $\eps$-cover $\mathcal{W}$ of $\mathcal{P}_n$ in $\mathbb{D}^{6}$ by tamely embedded $6$-cubes
such that if $Z\subset \mathcal{P}_n$ is a continuum and if $d_H(Z,0_n)>\eps$ then there exists a chain subcover $\mathcal{W}'\subset \mathcal{W}$ of $Z$. If we fix $p\in \mathcal{P}_n$ at the beginning, then we may require that $p$ belongs only to one element of the cover $\mathcal{W}$.
\end{claim}

\begin{proof}[Proof of Claim~\ref{lem:embeddP}]
It is enough to show that the lemma holds for $\mathcal{P}_1$. % and then use the fact that $\z_{\infty,n}$ is a branched cover for every $n$.
 For $\mathcal{P}_n$ the construction is almost the same, with the only difference that sets $D_k$ bellow will follow $T_n$-like pattern with only one element over branch point of $T_n$.

We start in $\mathbb{D}^3$. Fix $\eps>0$. Let us cover the arc $T_1=\{(x,y):x\in[-1,1],y=0,z=0\}$ with a chain of cells $\{D_0,\ldots D_{j}\}$ in $\mathbb{D}^3$ so that $T_1\in \textrm{int} (D_0\cup\ldots \cup D_{j})$. For example we can set 
$$D_0=\Bigg[-1-\frac{1}{j-1},-1+\frac{1}{j-1}\Bigg]\times [2^{-j},-2^{-j}]\times [2^{-j},-2^{-j}],$$ 
$$D_1=\Bigg[-1+\frac{1}{j-1},-1+\frac{3}{j-1}\Bigg]\times [2^{-j},-2^{-j}]\times [2^{-j},-2^{-j}],$$
$$\vdots$$
$$\hspace{-0.35cm}D_{j-1}=\Bigg[-1+\frac{2j-5}{j-1},-1+\frac{2j-3}{j-1}\Bigg]\times [2^{-j},-2^{-j}]\times [2^{-j},-2^{-j}],$$
$$D_{j}=\Bigg[-1+\frac{2j-3}{j-1},-1+\frac{2j-1}{j-1}\Bigg]\times [2^{-j},-2^{-j}]\times [2^{-j},-2^{-j}].$$
It is clear that we can select sets $D_k$ in such a way that if $p\in \mathcal{P}_1$ is fixed, then $p_0$ belongs to a unique $D_i$.
Now fix $k\in\{0,...,j\}$ and $i\in\mathbb{N}$ and consider $D_k^{\leftarrow}=\pi_{i,1}^{-1}(D_k)$. Observe that $\mathcal{P}_1$ as preliminary defined in Claim~\ref{PSS} is contained in $\textrm{int} (D_0\cup\ldots \cup D_{j})$.
Since the bonding maps $\bar f_{1,1},\bar f_{2,1},...,\bar f_{i,1}$ are near-homeomorphisms of $\mathbb{D}^3$, they are cell-like\footnote{\black For the definition of cell-like and cellular, the reader is referred to the paper of Edwards \cite{Edwards80}.}  \cite[Proposition p.116]{Edwards80}, see also \cite[Theorem 1.2.(c)]{Lacher}. Hence, since $D_k$ is cell-like, the pre-image $\bar f_{i,1}^{-1}\circ...\circ \bar f_{1,1}^{-1}(D_k)$ is cell-like as well \cite[Theorem 1.4]{Lacher}. Since $D_k^{\leftarrow}$ is the intersection of a countable nested family of cell-like sets, it is also cell-like \cite[Remark (3) p.113]{Edwards80} in $\mathbb{D}^3$, and hence it can be regarded as a cell-like subset of $\mathbb{D}^3$. By \cite[Theorem 8]{McMillan} and \cite[Corollary 2]{Lacher1} space $D_k^{\leftarrow}\times \{0\}^2$ is cellular in $\mathbb{D}^5=\mathbb{D}^3\times [-1,1]^2$.

Since $\bar f_{i,1}$ is a near-homeomorphism, choosing $i$ and $j$ large enough we can ensure that $D_k^{\leftarrow}$ has diameter less than $\eps/2$ for each $k\in\{0,\ldots,j\}$. By cellularity of $D_k^{\leftarrow}$, for any $\delta<\eps/2$ there exist topological 5-cubes $C_k$ around $D_k^{\leftarrow}$ in $\mathbb{D}^5$, each point of which is not further away than $\delta$ from $D_k^{\leftarrow}$. If $\delta$ is chosen small enough, then $C_k\cap C_{k'}\neq\emptyset$ if and only if $D_k^{\leftarrow}\cap D_{k'}^{\leftarrow}\neq\emptyset$ for any $k\neq k'$. Because every topological cube in $\mathbb{R}^n\times\{0\}$ is flat in $\mathbb{R}^{n+1}$ %{\bf \color{blue} what is the right reference here (we could only find statements about spheres on this page)? Furthermore, we need a map between cubes.} 
(see e.g. p.66 in \cite{DevVen}) we have a homeomorphism $\Theta:\mathbb{D}^6\to \mathbb{D}^6$ such that $\Theta(C_k\times \{0\})=\mathbb{D}^5\times \{0\}$. Put $\hat C_k=\Theta^{-1}(\mathbb{D}^5\times [-\delta,\delta])$ and note that if $\delta$ is sufficiently small then $ C_k\cap C_{k'}\neq\emptyset$ if and only if $\hat C_k\cap \hat C_{k'}\neq\emptyset$ for any $k\neq k'$. If $\delta$ is sufficiently small, then only one of these sets intersects $\mathbb{F}_n$. It is also clear that $\mathcal{P}_1$ is contained in $ \textrm{int}(\hat C_0)\cup\ldots \cup \textrm{int}(\hat C_{j})$. This completes the proof of this claim.
% \begin{claim}\label{lem:embeddP}
% For every $\eps>0$ there exists an $\eps$-cover $\mathcal{W}$ of $\mathcal{P}_n$ in $\mathbb{D}^5$ by tame topological 5-balls 
% such that if $Z\subset \mathcal{P}_n$ is a continuum and if $d_H(Z,0_n)>\eps$ then there exists a chain subcover $\mathcal{W}'\subset \mathcal{W}$ of $Z$.
% \end{claim}
%\begin{proof}
% The $5$-balls $C_k\times[\epsilon/10,\epsilon/10]$ are tame in $\mathbb{D}^5$, with $C_k$ as in the proof of the previous claim.
\end{proof}}
\begin{claim}\label{claim3}
 $\hat R_\infty|_{\mathcal{P}_\infty}$ has a unique fixed point $0_\infty{\black \cap \mathcal{P}_{\infty}}$, and for every $z\in \mathcal{P}_\infty\setminus \{0_\infty\}$ we have that $\hat R_\infty|_{\omega(\hat R_\infty,z)}$ 
is a $2$-adic odometer.
\end{claim}
\begin{proof}[Proof of Claim~\ref{claim3}]
{\black By definition each $z\in 0_\infty$ is a fixed point of $\hat R_\infty$ and $0_\infty{\black \cap \mathcal{P}_{\infty}}$ is a single point. If $z\in \mathcal{P}_\infty$ and $z\not\in 0_\infty$} then $z_n$ is in a $2^n$-periodic orbit $O_n$ of $\hat R_n$,
so $z$ is clearly a point in $2$-adic odometer.
\end{proof}

\begin{claim}\label{claim4}
$\mathcal{P}_\infty=\varprojlim(\mathcal{P}_n,\z_\infty|_{\mathcal{P}_n}\}$ is a pseudo-arc.
\end{claim}
\begin{proof}[Proof of Claim~\ref{claim4}]
Since each $\mathcal{P}_n$ is a pseudo-arc, the inverse limit of them $\mathcal{P}_\infty$ is also a pseudo-arc, as it must be both arc-like and hereditarily indecomposable, see \cite{Reed}.
\end{proof}

The proof of Theorem~\ref{odo} is now complete.
\end{proof}

% \begin{remark}\label{rem:forbidden}
% $\mainmap$ extends to $\mathbb{D}^4_\infty:=\varprojlim (S_n^4,\varphi_{\Join})$ by {\black extending} from $S_n$ to $S_n^4=S_n\times \mathbb{D}^2$ and letting $\z_{\Join}:=\varphi_\infty\times \id_{\mathbb{D}^2}$. Recall, that by Browns' theorem \cite{Brown1} $S_n^4$ can be identified with $4$-dimensional cube $\mathbb{D}^2
% \times \mathbb{D}^2$. To simplify notation we shall still use $\mathbb{D}_\infty$, $S_n$ and $\z_\infty$ for $\mathbb{D}^4_\infty$, $S^4_n$ and $\z_{\Join}$ respectively. This fact will be used later in Section~\ref{sec:mainproof}. Also, for every $n\geq 1$ we shall let $\mathbb{F}_n:=\{0_n\}\times \mathbb{D}^2$ to be the set of points where $\varphi_{\Join}$ is not a local homeomorphism.
% \end{remark}

\begin{equ}[!ht]
\begin{equation*}
\xymatrix{
&\ar@{--}[rr]^{}\ar@{--}[dl]_{}\ar@{--}[dd]|!{[d];[d]}\hole &&  
\ar@{--}[rr]^{}\ar@{--}[dl]_{} \ar@{--}[dd]|!{[d];[d]}\hole &&
\ar@{--}[rr]^{}\ar@{--}[dl]_{} \ar@{--}[dd]|!{[d];[d]}\hole && \ar@{--}[rr]\ar@{--}[dd]\ar@{--}[dl] && 
\Ps_\infty \ar@{<-}[dl]_{\mainmap|_{\Ps_\infty}}\ar@{--}[dd]\\
\ar@{--}[rr]^{}\ar@{--}[dd]^{}&&\ar@{--}[rr]^{}\ar@{--}[dd]&& \ar@{--}[rr]^{}
\ar@{--}[dd]^(.35){}|!{[d];[d]}\hole &&
\ar@{--}[rr]\ar@{--}[dd] && \Ps_\infty \ar@{--}[dd]\\
&\ar@{->}[dd]^(.35){\z|_{T_4}}\ar@{--}[dl]\ar@{--}[rr]\hole&&\ar@{->}[dd]^(.35){\z|_{T_4}}\ar@{--}[dl]\ar@{--}[rr]\hole &&\ar@{->}[dd]^(.35){\z|_{T_4}}\ar@{--}[dl]\ar@{--}[rr]\hole && \ar@{}[rr]\ar@{--}[rr] \ar@{--}[dd]\ar@{--}[dl]&& \ar@{->}[dd]^{\z_{\infty}|_{\Ps_4}}\ar@{--}[dl]\\
\ar@{->}[dd]^(.35){\z|_{T_4}}\ar@{--}[rr]&&\ar@{->}[dd]^{}\ar@{--}[rr]&&
\ar@{->}[dd]^(.35){}\ar@{--}[rr]\hole &&
\ar@{--}[rr] \ar@{--}[dd]&&  \ar@{->}[dd]^(0.35){\z_{\infty|_{\Ps_4}}}\\
&T_3\ar@{<-}[rr]^{f_{0,3}}\ar@{<-}[dl]_{R_3}\ar@{->}[dd]|!{[d];[d]}\hole&&T_3 
\ar@{<-}[rr]^{f_{1,3}}\ar@{<-}[dl]_{R_3} \ar@{->}[dd]|!{[d];[d]}\hole &&T_3 
\ar@{<-}[rr]^{f_{2,3}}\ar@{<-}[dl]_{R_3} \ar@{->}[dd]|!{[d];[d]}\hole && \ar@{--}[rr]\ar@{--}[dd] \ar@{--}[dl]&& \Ps_3\ar@{<-}[dl]_{\hat R_3|_{\Ps_3}}\ar@{->}[dd]^(.35){\z_{\infty}|_{\Ps_3}}\\
T_3\ar@{<-}[rr]^{f_{0,3}}\ar@{->}[dd]^{\z|_{T_3}}&&T_3\ar@{<-}[rr]^{f_{1,3}}\ar@{->}[dd]&&T_3 \ar@{<-}[rr]^{f_{2,3}}
\ar@{->}[dd]^(.35){\z|_{T_3}}|!{[d];[d]}\hole &&
\ar@{--}[rr]\ar@{--}[dd] && \Ps_3 \ar@{->}[dd]^(.35){\z_{\infty}|_{\Ps_3}}\\
&T_2\ar@{<-}[rr]|!{[r];[r]}\hole\ar@{<-}[dl]_{R_2} 
\ar@{->}[dd]|!{[d];[d]}\hole&&T_2 \ar@{<-}[rr]|!{[r];[r]}\hole
\ar@{<-}[dl]_{R_2}\ar@{->}[dd]|!{[d];[d]}\hole
&&T_2\ar@{<-}[rr]^{f_{2,2}}\ar@{<-}[dl]_{R_2}\ar@{->}[dd]^{}|!{[d];[d]}\hole && \ar@{--}[rr]\ar@{--}[dd]\ar@{--}[dl] && \Ps_2\ar@{<-}[dl]_{\hat R_2|_{\Ps_2}}\ar@{->}[dd]^(.35){\z_{\infty}|_{\Ps_2}}\\
T_2\ar@{<-}^(.65){f_{0,2}}[rr]\ar@{->}[dd]^{\z|_{T_2}}&&T_2\ar@{->}[dd]
\ar@{<-}^(.65){f_{1,2}}[rr]&&T_2\ar@{<-}[rr]^{f_{2,2}}\ar@{->}[dd]^(.35){\z|_{T_2}}\hole && \ar@{--}[rr]\ar@{--}[dd] && \Ps_2\ar@{->}[dd]^(.35){\z_{\infty}|_{\Ps_2}}\\
&T_1\ar@{<-}[rr]_(.35){f_{0,1}}|!{[r];[r]}\hole\ar@{<-}[dl]^{R_1}&&T_1 
\ar@{<-} [rr]_(.35){f_{1,1}}|!{[r];[r]}\hole\ar@{<-}[dl]^{R_1}&&T_1\ar@{<-}[rr]^{f_{2,1}}
\ar@{<-}[dl]^{R_1}\hole && \ar@{--}[rr]\ar@{--}[dl] && \Ps_1 \ar@{<-}[dl]_{\hat R_1|_{\Ps_1}}\\
T_1\ar@{<-}[rr]^{f_{0,1}}&&T_1\ar@{<-}[rr]^{f_{1,1}}&&T_1 \ar@{<-}[rr]^{f_{2,1}}\hole &&  \ar@{--}[rr] && \Ps_1}
\end{equation*}
\caption{Diagram explaining the construction from the proof of Theorem~\ref{odo}.}\label{megadiagram}
\end{equ}

\section{The Denjoy-Rees technique}\label{sec:DRtechnique}
In this section we will give conditions we want to have fulfilled in our construction of pseudo-arc homeomorphismsms with arbitrary positive topological entropy. Namely, we will implement the conditions given in \cite{Cro} in the inverse limit technique. All the conditions that we state in this section will be realized through an inductive construction later in the following section.  We encourage the reader to familiarize with basic ideas from \cite{Cro} before going into this and the subsequent section. To be in the correspondence with \cite{Cro} and avoid potential confusion we will keep the notation for analogous conditions as in \cite{Cro}.

For a family $\E$ of closed sets of $\mathbb{D}_{\infty}$ define
\begin{equation}
\mesh (\E)=\max \{  \diam (X) \big | X\in \E\}.
\end{equation}
For the family $\E$ 
denote by 
$\rs(\E)$ 
the union of all the elements of 
$\E$, called the \textit{realization} of $\mathcal E$.
\begin{definition}
For an integer $q\geq 1$, 
a finite family $\E$
is \textit{$q$-iterable} if for any 
 $X,Y\in \E$ and  
integers $-q\leq k,s\leq q$, 
either  $\mainmap^k(X)=\mainmap^s(Y)$
or $\mainmap^k(X)\cap \mainmap^s(Y)=\emptyset$. 
\end{definition}
For any $q$-iterable family 
$\E^0$ and any $0\leq n\leq q$,
we denote
\begin{equation}
\E^n=\bigcup_{|k|\leq n} \mainmap^k(\E^0),
\end{equation}
where 
$\mainmap(\E^0)=\{\mainmap(X) \big | X\in \E^0\}$. 
For any $0\leq n < q$,
define an {\em oriented graph} $\mathcal G(\E^n)$, 
where the vertices
 are elements of $\E^n$,
 and there is an edge from $X$ to $Y$
if and only if $\mainmap(X)=Y$, {\black see Figure 3 in \cite{Cro}}. 

Furthermore, 
we say that 
$\E^n$ \textit{has no cycle} if the graph $\mathcal G(\E^n)$ has no cycle.

\begin{definition}
Let $\E,\F$ be two finite families of closed subsets of $\mathbb{D}_\infty$.
We say \emph{$\F$ refines $\E$} if the following conditions hold.
\begin{enumerate}
\item every element of $\E$ contains at least one element of $\F$.
\item for any 
$X\in \E$, $Y\in \F$,
 either $X\cap Y=\emptyset$ or $Y\subset \Int (X)$.
\end{enumerate}
\end{definition}

\begin{definition}
Let $\E,\F$ be 
two families as above.
For an integer $q\geq 0$,
we say 
 \emph{$\F$ is compatible with $\E$ for $q$ iterates},
 if the following conditions hold.
 \begin{enumerate}
 \item $\E$ is $q$-iterable.
 \item  $\F$ is $(q+1)$-iterable.
 \item  $\rs (\F) \subset \rs (\E)$, and $\F^{q+1}$ refines $\E^q$. 
 \item  For every $k$ with $|k|\leq q$, 
  \begin{equation}
  \rs(\F^{q+1}) \cap \mainmap^k( \rs(\E) ) = \mainmap^k( \rs(\F) ).
  \end{equation}
\end{enumerate}
\end{definition}

Let an increasing sequence of non-negative integers $\{k_n\}_{n\geq 0}$ be given (we will specify it after Lemma~\ref{lem:choiceofKinA}) and
let $\pi_n\colon \mathbb{D}_\infty \to S_n$ denote the projection onto $n$-th coordinate of $\mathbb{D}_\infty$.
In what follows $\{\E^0_{(n)}\}_{n\geq 0}$ denotes a sequence with the following properties (see {\black Theorem~\ref{odo}} to recall the definition of $\mathbb{F}_{n}$):
\begin{enumerate}
    \item Each $X\in \E^0_{(n)}$ is defined by $\pi_{k_n}^{-1}(D)$
    where $D$ is a closed cube tamely embedded in $S_{k_n}\setminus\mathbb{F}_{k_n}$.
    \item Each $\E^0_{(n)}$ is finite.
\end{enumerate}
Because we are using $2$-to-$1$ branch coverings, note that topologically we can view each $\pi_{k_n}^{-1}(D)$ as $D\times C$
where $C$ is the Cantor set.

The following list of hypotheses 
is an extension of analogous ones in \cite[p.262]{Cro}
stated for families of closed discs.
\begin{enumerate}
\item [$\mathbf{A_1}$.] The following assertions hold:
       \begin{enumerate}
	\item For every $n \geq 0$ the family $\E_{(n)}^0$ is 
	$(n+1)$-iterable and the graph $\mathcal G(\E_{(n)}^n)$ has no cycle.
	\item For every $n > 0$ the family $\E^{n+1}_{(n)}$ refines $\E_{(m)}^{m+1}$ for any $0\leq m <n$.
	\item For every $n> 0$ the family $\E_{(n)}^0$ is compatible with $\E_{(n-1)}^0$ for $n$ iterates.
       \end{enumerate}
\item [$\mathbf{A_2}$.]  For every $n\geq 0$ and every element $X\in \E_{(n)}^0$ there are at least two distinct elements of 
$\E_{(n+1)}^0$ contained in $X$.

\item [$\mathbf{A_3}$.]  The following holds:
\begin{equation}
\lim_{n\to \infty}\mesh (\E_{(n)}^n) = 0.
\end{equation}
\end{enumerate}

Let $(\Lambda,\mainmap)$ be the  $2$-adic odometer provided by Theorem~\ref{odo}. From now on we shall identify $\Lambda$ and $\prod_{i=0}^\infty\{0,...,2^{i}-1\}$.
The following fact is true for any minimal system by proper application of Kakutani-Rohlin partitions (see Proposition 2.10 in \cite{Cro}). Since we deal with odometers whose structure is known, the situation is simpler, because we know exactly the structure of our map. The choice of the point $p\in\Lambda$ below is crucial to our construction.
\begin{lem}\label{lem:choiceofKinA}
	Let $(\Lambda,\mainmap)$ be the  $2$-adic odometer provided by Theorem~\ref{odo}, where $\Lambda$ is assumed to be endowed with the lexicographic order and let $\mu$ be the Haar measure on $\Lambda$. For any sequence
	$\{\hat{k}_n\}_{n\geq 0}$ of positive integers, there exists a subsequence $\{{k}_n\}_{n\geq 0}$, a sequence $\{{s}_n\}_{n\geq 0}$ of positive integers, and a Cantor set $K\subset \Lambda$ with $\mu(K)>0$ such that:
	\begin{enumerate}
		\item\label{lem:choiceofKinA:1} $K\cap \mainmap (K)=\emptyset$,
		\item\label{lem:choiceofKinA:2} if $W_0,\ldots, W_{2^{k_n}-1}$ is a periodic decomposition of $\Lambda$ {\black using $\mainmap$} given by {\black the} cylinder sets, with 
$p=\min K \in W_0$, where minimum is taken
  with respect to the lexicographic order,
		then 
		$K\subset \bigcup_{i=0}^{s_n} \mainmap^i(W_0)$,
		\item\label{lem:choiceofKinA:3} $k_{n+1}>2^{k_n}+3n$ and  $s_n+2n+2<2^{k_n}$ for each $n$, and 
		\item\label{lem:choiceofKinA:4}  the set $\{a\in \mathbb{N} : \text{there is } y\in K, y_{k_{n+1}}=a, y_{k_n}=x_{k_n}\}$ has at least two elements for each $n$ and each $x\in K$.
	\end{enumerate}
\end{lem}\label{lem:setK}
\begin{proof}
	Since $(\Lambda,\mainmap)$ is 2-adic, we may forget for a moment about the condition $K\cap \mainmap (K)=\emptyset$, because we can always provide the construction for one of the odometers in $(\Lambda,\mainmap^2)$.

	Since $\mu$ gives the same mass to each cylinder defined by element of $\Z_{2^n}$, removing
	$k$ consecutive (in the sense of addition $+1$ in $n$-th coordinate) cylinders subtracts the mass of $k/2^n$. 
	Let $\Lambda_0=\{x\in \Lambda : x_{k_0}\neq 0\}$ 
	and for $n\geq 1$
	\begin{equation}
	    \Lambda_n=\{x\in \Lambda\setminus \Lambda_{n-1}:2^{k_n-1}-4(n+1)2^{k_{n-1}} \leq x_{k_n}\leq 2^{k_n-1}+4(n+1)2^{k_{n-1}}\}\label{def:AnK}
	\end{equation} and $K=\Lambda\setminus \bigcup_{n\geq 0} \Lambda_n$. By the structure of $\Lambda_n$ we obtain a useful {\black accumulation of holes in the sense that if $x\in \Lambda$ and $x_n=0$ for some $n>0$ then $x$ will not return to $K$ before $4(n+1)2^{k_{n-1}}$ iterations of $\mainmap$ neither backward nor forward}, see Fig.~\ref{fig:Sketch}.
	
	\begin{figure}[!ht]
	\centering
	\begin{tikzpicture}[scale=1.8]
	\draw[thick, domain=0:360] plot ({0.125*cos(\x)}, {0.875+0.125*sin(\x)});
	\draw[domain=0:360] plot ({0.025*cos(\x)}, {0.875+0.025*sin(\x)});
	\draw[thick, domain=0:360] plot ({1+0.125*cos(\x)}, {0.875+0.125*sin(\x)});
	\draw[domain=0:360] plot ({1+0.025*cos(\x)}, {0.875+0.025*sin(\x)});
	\draw[domain=0:360] plot ({0.25+0.025*cos(\x)}, {0.875+0.025*sin(\x)});
	\draw[domain=0:360] plot ({0.5+0.025*cos(\x)}, {0.875+0.025*sin(\x)});
	\draw[domain=0:360] plot ({0.75+0.025*cos(\x)}, {0.875+0.025*sin(\x)});
	
	\draw[thick, domain=0:360] plot ({0.125*cos(\x)}, {-1+0.875+0.125*sin(\x)});
	\draw[domain=0:360] plot ({0.025*cos(\x)}, {-1+0.875+0.025*sin(\x)});
	\draw[thick, domain=0:360] plot ({1+0.125*cos(\x)}, {-1+0.875+0.125*sin(\x)});
	\draw[domain=0:360] plot ({1+0.025*cos(\x)}, {-1+0.875+0.025*sin(\x)});
	\draw[domain=0:360] plot ({0.25+0.025*cos(\x)}, {-1+0.875+0.025*sin(\x)});
	\draw[domain=0:360] plot ({0.5+0.025*cos(\x)}, {-1+0.875+0.025*sin(\x)});
	\draw[domain=0:360] plot ({0.75+0.025*cos(\x)}, {-1+0.875+0.025*sin(\x)});
	
	\draw[thick, domain=0:360] plot ({1.5+0.125*cos(\x)}, {-1+0.875+0.125*sin(\x)});
	\draw[domain=0:360] plot ({1.5+0.025*cos(\x)}, {-1+0.875+0.025*sin(\x)});
	\draw[thick, domain=0:360] plot ({2.5+0.125*cos(\x)}, {-1+0.875+0.125*sin(\x)});
	\draw[domain=0:360] plot ({2.25+0.025*cos(\x)}, {-1+0.875+0.025*sin(\x)});
	\draw[domain=0:360] plot ({1.75+0.025*cos(\x)}, {-1+0.875+0.025*sin(\x)});
	\draw[domain=0:360] plot ({2+0.025*cos(\x)}, {-1+0.875+0.025*sin(\x)});
	\draw[domain=0:360] plot ({2.5+0.025*cos(\x)}, {-1+0.875+0.025*sin(\x)});
	
	\draw[thick, domain=0:360] plot ({3+0.125*cos(\x)}, {-1+0.875+0.125*sin(\x)});
	\draw[domain=0:360] plot ({3+0.025*cos(\x)}, {-1+0.875+0.025*sin(\x)});
	\draw[thick, domain=0:360] plot ({4+0.125*cos(\x)}, {-1+0.875+0.125*sin(\x)});
	\draw[domain=0:360] plot ({3.75+0.025*cos(\x)}, {-1+0.875+0.025*sin(\x)});
	\draw[domain=0:360] plot ({3.25+0.025*cos(\x)}, {-1+0.875+0.025*sin(\x)});
	\draw[domain=0:360] plot ({3.5+0.025*cos(\x)}, {-1+0.875+0.025*sin(\x)});
	\draw[domain=0:360] plot ({4+0.025*cos(\x)}, {-1+0.875+0.025*sin(\x)});
	
	\draw (-0.2,0.1)--(1.2,0.1);
	\draw (-0.2,-0.35)--(1.2,-0.35);
	\draw (-0.2,-0.35)--(-0.2,0.1);
	\draw (1.2,-0.35)--(1.2,0.1);
	
	\draw (1.3,0.1)--(2.7,0.1);
	\draw (1.3,-0.35)--(2.7,-0.35);
	\draw (1.3,-0.35)--(1.3,0.1);
	\draw (2.7,-0.35)--(2.7,0.1);
	
	\draw (-0.3,-0.35)-- (-0.3,1);
	\draw (1.25,-0.35)-- (1.25,1);
	\draw (2.75,-0.35)-- (2.75,1);
	\draw (4.25,-0.35)-- (4.25,1);
	\node at (0.5,1.2) {$\Lambda_1$};
	\draw[->] (0.4,1.15)--(0.2,1);
	\draw[->] (0.6,1.15)--(0.8,1);
	
	\node at (1.25,-0.65) {$\Lambda_2$};
	\draw[->] (1.15,-0.6)--(0.8,-0.4);
	\draw[->] (1.35,-0.6)--(1.7,-0.4);
	
	\node at (-0.5,0.85) {$K_1$};
	\node at (-0.5,-0.15) {$K_2$};
	
	\draw[decorate,decoration={brace,amplitude=10pt},xshift=-4pt,yshift=0pt](-0.5,-1.25) -- (-0.5,1.1); 
	\node at (-1,-0.1) {\large $K$};
	\node[circle,fill, inner sep=1] at (0.5,-0.7){};
	\node[circle,fill, inner sep=1] at (0.5,-0.8){};
	\node[circle,fill, inner sep=1] at (0.5,-0.9){};
	
	\node[circle,fill, inner sep=1] at (2,-0.7){};
	\node[circle,fill, inner sep=1] at (2,-0.8){};
	\node[circle,fill, inner sep=1] at (2,-0.9){};
	
	\node[circle,fill, inner sep=1] at (3.5,-0.7){};
	\node[circle,fill, inner sep=1] at (3.5,-0.8){};
	\node[circle,fill, inner sep=1] at (3.5,-0.9){};
	
	\draw[dashed] (0.1,0.75)-- (0.9,0);
	\draw[dashed] (1.2,0.75)-- (2.4,0);
	\draw[dashed] (1.3,0.85)-- (3.9,0);
	\end{tikzpicture}
	
	\caption{The idea of the construction of sets $\Lambda_n$ in Lemma~\ref{lem:setK}.}
	\label{fig:Sketch}
\end{figure}

	Then, for sufficiently fast increasing sequence $k_n$ we obtain that
	$$
	\mu(K)\geq 1 -\sum_{n\geq 1} \frac{%4
    {\black8}(n+1)2^{k_{n-1}}}{2^{k_n}}>0.
	$$
\end{proof}

Let us now fix, once and for all, the sequences $\{k_n\}_{n\geq 0}$, $\{s_n\}_{n\geq0}$, Cantor set $K$, and $p$ the minimum of $K$ with respect to the lexicographic order provided by Lemma~\ref{lem:choiceofKinA} for the sequence $\{2^n\}_{n\geq0}$. Now we proceed to the part where we  blow up the orbit of $K$.

By the definition we have that $\pi_1(K)=\{p_1\}$ is a singleton. Fix a Cantor set $C\subset \mathcal{P}_1$ {\black such that $p_1\in C$}. 
Note that $\z_\infty^{-n}(C)={\black \bigcup^{2^n}_{i=1}}C^{i,n} \subset \mathcal{P}_n$, such that $\z_\infty^n(C^{i,n})=C$ and $\z_\infty^n|_{C^{i,n}}$ is a homeomorphism for each $i$. Let $C^{n}_K:=\bigcup_{i=1}^{2^n}\{C^{i,n}: C^{i,n}\cap \pi_n(K)\neq\emptyset\}$. 
Therefore the set 
${\black \hat K=}\bigcap_{n\in\mathbb{N}} \pi_n^{-1}(C^n_K)$  
is a Cantor set
in $\spaceX$. 
{\black Indeed, to see that, first fix any point $(x,y)\in K\times C$. Note that there is a one-to-one correspondence between $x\in \Lambda$
and a sequence $i_n$ such that $x\in \bigcap_{n\in\mathbb{N}} \pi_n^{-1}(C^{i_n,n})$ and for each $m$ if points $x,p\in \Lambda$ are sufficiently close then they define the same indices $i_1,\ldots, i_m$. Note also that each set $\bigcap_{n\in\mathbb{N}} \pi_n^{-1}(C^{i_n,n})$ is homeomorphic to $C$,
with homeomorphism obtained as a projection onto the first coordinate.
Therefore, for each $x$ in $\Lambda$ we have a well defined map $\eta_x: C\to \bigcap_{n\in\mathbb{N}} \pi_n^{-1}(C^{i_n,n})\subset \mathcal{P}_\infty$ which is a homeomorphism onto its image. By the above discussion, for every $\eps>0$
there is $\delta>0$ such that if $d(x,p)<\delta$ then $d(\eta_x(y),\eta_p(q))<\eps$ for any $y,q\in C$. Namely, for small $\delta$ we have $\pi_m(\eta_x(y),\eta_p(q))\in C^{i_m,m}$ for some $i_m$. This shows that $\hat{K}\subset \mathcal{P}_\infty$ is homeomorphic to $K\times C$ by the map $\eta:K\times C$ defined by $(x,y)\mapsto \eta_x(y)\in \hat K$.}
%{\color{blue} Note that each point $x\in %K
%{\blackC^{i,n}}$
%defines a unique point in the finite set $\pi_n(K)$.} This allows us to homeomorphically identify the above intersection with $K\times C$.
If we intersect $K\times C$ with $\pi_n^{-1}(C^{i,n})$, then {\black the} obtained set can be identified
with $(\pi_n^{-1}(\{\hat{R}^i_n(p_n)\})\cap K)\times C^{i,n}
%$ which is homeomorphic with %$
{\black=\eta((\pi_n^{-1}(\{\hat{R}^i_n(p_n)\})\cap K)\times C)}$. In other words, we have natural representation of splitting of $K\times C$ over cylinder sets in $K$.
Note that we can pull back $C$ to $\mathcal{P}_\infty$ by defining
$\hat{C}=\bigcap_{n\in\mathbb{N}} \pi_n^{-1}(C^{2^n,n})$. 
This way we have a natural 1-1 correspondence between $C$ and $\hat{C}$, and by the definition also $p\in \hat{C}$.
If we fix any pseudo-arc
$P\subset \mathcal{P}_\infty\setminus\{0_{\infty}\}$ {\black such that $p\in P$,}
then its projection $\pi_1(P)$ is nondegenerate, and therefore we may also require that $C\subset \pi_1(P)$ which gives $\hat{C}\subset P$. Using identification of $C$ with $\hat{C}$, we can view $C$ as a subset of $\mathcal{P}_\infty$.
The pseudo-arc $P$ will be specified later, before Lemma~\ref{lem:main}.

Here and later we denote by $K_X:=K\cap X$ for some $X\in \E_{(n)}^0$. Proof of the following fact is the same as Lemma 2.5 in \cite{Cro}.
\begin{lem}
Assume that $\mathbf{A_1}$ holds and let $X\to \ldots \to X'=\mainmap^q(X)$ for some $q\geq 0$ be a path in the graph
$\mathcal G(\E_{(n)}^n)$ with $X,X'\in \E_{(n)}^0$. Then $\mainmap^q(K_X)=K_{X'}$.
\end{lem}

Now we will introduce basic ingredients of the construction, which are inspired by \cite{Cro}. Instead of working on manifolds as in \cite{Cro} we will work on inverse limits of cubes. The main difference in the approach is depicted in Figure~\ref{fig:1}.

First we set 
$\psi_0=\id$ and $G_0=\mainmap$. 
Next, we want to choose a sequences of 
homeomorphisms $\{H_n:\mathbb{D}_{\infty}\to \mathbb{D}_{\infty}\}_{n\geq 1}$, $\{h_n:\mathbb{D}_{\infty}\to \mathbb{D}_{\infty}\}_{n\geq 1}$, 
such that the following axioms $\mathbf{B_{1,2,3,5,6,7}}, \mathbf{C_{1,2,5,6,7,8}}$ hold for every $n\geq 1$
and furthermore:
\begin{enumerate}
    \item each $h_n$ (resp. $H_n$) is an extension  of a homeomorphism $\hat{h}_n
    \colon S_{k_n}\to S_{k_n}$ (resp. $\hat{H}_n$) through branched covers $\varphi_\infty$ to the whole inverse limit $\spaceX$.
\end{enumerate}
Whenever we have such sequences of homeomorphisms $h_n, H_n$, for any $n\geq 1$,
we define the homeomorphisms 
$\psi_n, g_n, G_n: \mathbb{D}_{\infty}\to \mathbb{D}_{\infty}$ as follows. 
\begin{align}
\psi_n &=H_n \circ h_n \circ \ldots \circ H_1\circ h_1.\label{eq:9}\\
g_n &=(h_n\circ\psi_{n-1})^{-1}\circ \mainmap \circ h_n\circ\psi_{n-1}.\label{eq:10}\\
G_n &=\psi_n^{-1}\circ \mainmap \circ \psi_n.\label{eq:11}
\end{align}

The following conditions are based on those given in \cite[pp.270, 275]{Cro}, with the exception that we skip condition $\mathbf{B_{4}}$, but add condition $\mathbf{B_{7_{a,b,c}}}$ instead. Namely, we do not need that fibers are nowhere dense in $\spaceX$, but rather we need to control carefully the
structure of $\psi^{-1}(P)$ (which will be nowhere dense in $\spaceX$ by the definition). Note also that our condition $\mathbf{B_{3}}$ is a combination of conditions $\mathbf{B_{3}}$ and $\mathbf{C_{3}}$ from \cite[pp.270, 283]{Cro}.
\begin{figure}[!ht]
		\includegraphics[width=12cm]{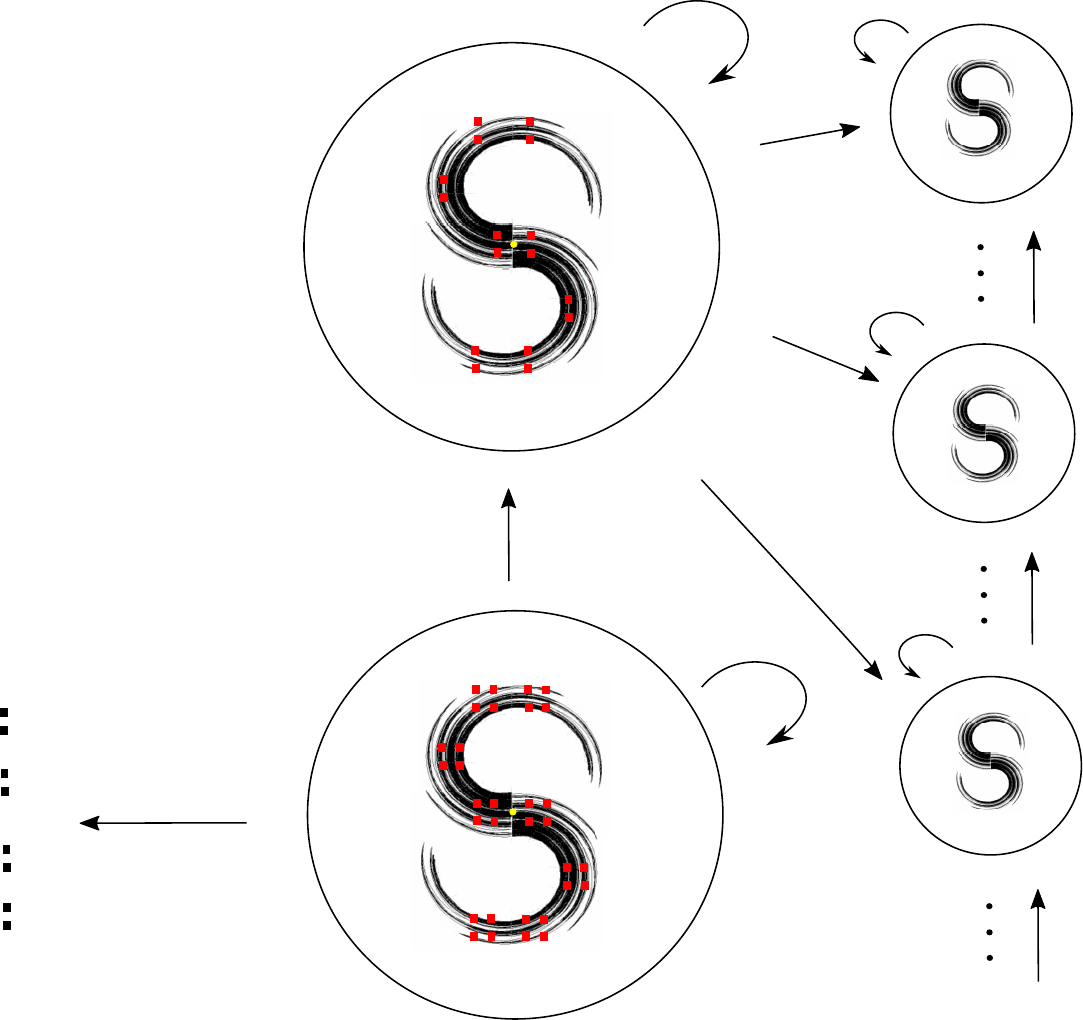}
		\put(-345,110){{$C$}}
		\put(-315,70){{\small{$C\leftarrow K\times C$}}}
		\put(-355,70){$\sigma$}
		\put(-200,150){{$\Psi$}}
		\put(-100,310){{$\hat{R}_{\infty}$}}
		\put(-240,305){{$\mathbb{D}_{\infty}$}}
		\put(-240,120){{$\mathbb{D}_{\infty}$}}
		\put(-220,80){{$\mathbb{P}$}}
		\put(-227,260){{$\mathcal{P}_{\infty}$}}
		\put(-192,80){{$K\times C$}}
		\put(-182,260){{$K$}}
		\put(-90,265){{$\pi_{k_0}$}}
		\put(-90,195){{$\pi_{k_1}$}}
		\put(-90,150){{$\pi_{k_2}$}}
		\put(-82,95){{$G$}}
		\put(-10,230){{$\varphi_{\infty}^{k_1-k_0}$}}
		\put(-10,130){{$\varphi_{\infty}^{k_2-k_1}$}}
		\put(-8,20){{$\varphi_{\infty}^{k_3-k_2}$}}
		\put(-2,305){{$S_{k_0}$}}
		\put(0,200){{$S_{k_1}$}}
		\put(0,95){{$S_{k_2}$}}
		\put(-20,285){\tiny{$\mathcal{P}_{k_0}$}}
		\put(-18,185){\tiny{$\mathcal{P}_{k_1}$}}
		\put(-18,80){\tiny{$\mathcal{P}_{k_2}$}}
		\put(-75,323){\tiny{$\hat R_{k_0}$}}
		\put(-70,230){\tiny{$\hat R_{k_1}$}}
		\put(-65,128){\tiny{$\hat R_{k_2}$}}
		\vspace{0.5cm}
		\caption{The Denjoy-Rees-like enrichment on the pseudo-arc.}\label{fig:1}
	\end{figure}
We also assume that we are given a minimal aperiodic dynamical system
$(C,\sigma)$. In practice, we will restrict our attention only to minimal subshifts with positive entropy, so it is natural to write $\sigma$ when denoting acting homeomorphism.

\begin{itemize}
    \item[$\mathbf{B_{1}}$.] 
For $n>0$ the closure
$\overline{\{x : h_n(x)\neq x\}}$ is contained
in the set $\rs (\E_{(n-1)}^{n-1})$.
    \item[$\mathbf{B_{2}}$.] 
For $n>0$ maps $h_n$ and $\mainmap$ commute 
along edges of the graph $\mathcal G(\E_{(n-1)}^{n-1})$.
    \item[$\mathbf{B_3}$.] The following holds:

\begin{equation}
\lim_{n\to \infty}
 \mesh \big(\psi_{n-1}^{-1}
 ( \E_{(n)}^{n+1} \backslash \E_{(n)}^{n-1}) \big)=0.
\end{equation}
\item[$\mathbf{B_{5}}$.] For $n\geq 0$ and every $X\in \E_{(n)}^0$ we have $K_X\times C\subset \psi_{n}^{-1}(\Int X)$.
\item[$\mathbf{B_{6}}$.] For $n>0$ let $\gamma: X\to \ldots \to X'=\mainmap^q(X)$ be a path in the graph
$\mathcal G(\E_{(n)}^n)$, such that $X,X'\in \E_{(n)}^0$ and intermediate vertices of $\gamma$ are not in  $\E_{(n)}^0$. Then
\begin{enumerate}[(a)]
\item if the path is fully contained in $\E_{(n)}^{n-1}$ then
$g_n^q(x,c)=G_{n-1}^q(x,c)=(\mainmap^q(x),\sigma^q(c))$  for every  $(x,c)\in K_X\times C$.
\item if the path contains an element from $\E_{(n)}^{n}\setminus \E_{(n)}^{n-1}$ then
$g_n^q(x,c)=(\mainmap^q(x),\sigma^{q}(c))$  for every  $(x,c)\in K_X\times C$.
\end{enumerate}
\item[$\mathbf{B_{7}}$.] For $n\geq 0$ there is a sequence of chains $\Mm_n\subset \mathbb{D}_{\infty}$ such that $\rs(\Mm_{n+1})\subset \rs(\Mm_n)$ and  $\bigcap_{n=0}^\infty \rs(\Mm_n)=\Ps_\infty$, and  
\begin{enumerate}[(a)]
\item 
$\mesh \big( \psi_{n}^{-1} (\Mm_{n}) \big)  < 2^{-n}.\label{mesh_is_small}$

\item If $n>0$ then for every $k<n$ we have
$
\psi_n^{-1}( \overline { \rs ( \Mm_{n}) })
\subset  \psi_k^{-1}( \rs ( \Mm_{k})) $
and
$\psi_{n}^{-1} (\Mm_{n})$ is crooked inside $\psi_{n-1}^{-1} (\Mm_{n-1})$.
\item $\rs (\E^{n}_{(n)})\subset \rs(\mathcal{C}_{n})$.
\end{enumerate}
\end{itemize}

Under the stated conditions we have the following two lemmas.

\begin{lem}\label{lem:hnconverge}
Assume that hypotheses $\mathbf{A_{1,2,3}}$ and $\mathbf{B_{1,2,3}}$ hold. Then:
\begin{enumerate}
    \item Sequence of homeomorphisms $\psi_n$ converges uniformly to a continuous surjection $\psi\colon \spaceX\to \spaceX$.
    \item Sequences of homeomorphisms $g_n$ and $g_n^{-1}$ (resp. $G_n$ and $G_n^{-1}$) converge uniformly to homeomorphism $g$ and $g^{-1}$, respectively (resp. $G$ and $G^{-1}$).
    \item The homeomorphism $G$ is a topological extension of $\mainmap$; i.e. $\mainmap\circ \psi=\psi\circ G$.
\end{enumerate}
\end{lem}
\begin{proof}
Proof is analogous to the proof of Proposition 3.1 in \cite{Cro} so we do not repeat it here.
\end{proof}

\begin{lem}\label{lem:fibres}
Assume that hypotheses $\mathbf{A_{1,2,3}}$ and $\mathbf{B_{1,2,3,7}}$ hold. Then:
\begin{enumerate}
\item Let $x\in \mainmap^k(K)$ for some integer $k\in \Z$ and let $X_n\in \E_{(n)}^k$, $n\geq |k|$
be a nested sequence of sets such that $x\in X_n$ for each $n$. Then
$$
\psi^{-1}(x)=\bigcap_{n\geq |k|}\psi^{-1}_n(X_n)
$$
and $\psi^{-1}(x)$ is a connected set.
\item For every $y\in \spaceX\setminus \bigcup_{i\in \Z} \mainmap^i(K)$ the set $\psi^{-1}(y)$ is a singleton.
\end{enumerate}
In particular, $\psi$ is a monotone map.
\end{lem}
\begin{proof}
Proof follows the same lines as the proof of Proposition~3.4 in \cite{Cro}.
Only the fact that $\psi^{-1}(x)$ is connected needs an additional argument, since sets $X_n$ are not discs but matchboxes $X_n=\pi_{k_n}^{-1}(D_n)\approx D_n\times C_n$ (where $\approx$ denotes that the spaces are homeomorphic) for some
cube $D_n\subset \mathbb{D}_4$ and Cantor set $C_n$. However by definition $h_n$ and $H_n$ (therefore also $\psi_n$) are lifts of homeomorphisms $\hat{h}_n,\hat{H}_n$ acting  on $S_{k_n}$, so 
$$
\psi_n^{-1}(X_n)\approx \hat{h}_n^{-1}\circ \hat{H}_n^{-1}(D_n)\times C_n.
$$
But then for each $s$ and $n>s$ we have that 
$$
\pi_s(\psi_n^{-1}(X_n))=\varphi_\infty^{k_n-s}(\hat h_n^{-1}\circ \hat H_n^{-1}(D_n))
$$
is a connected set, and the sequence $\pi_s(\hat h_n^{-1}\circ \hat H_n^{-1}(D_n)\times C_n)$ is nested, since sequence $\psi_n^{-1}(X_n)$ is nested. But then 
$\pi_s(\psi^{-1}(x))$ is connected for every $s$, thus $\psi^{-1}(x)$ is connected.
\end{proof}

The following conditions are also based on those given in \cite[pp.283-285]{Cro}, with the exception that we skip conditions $\mathbf{C_{3}}$ and $\mathbf{C_{4}}$. Similarly as above we do not need that fibers are nowhere dense in $\spaceX$.

\begin{itemize}
    \item[$\mathbf{C_{1}}$.] 
For $n>0$ the closure
$\overline{\{x : H_n(x)\neq x\}}$ is contained
in the set $\rs (\E_{(n)}^{n-1})$.
    \item[$\mathbf{C_{2}}$.] 
For $n>0$ maps $H_n$ and $\mainmap$ commute 
along edges of the graph $\mathcal G(\E_{(n)}^{n-1})$.

\item[$\mathbf{C_{5}}$.] For $n\geq 0$ and every $x\in K$ the map $H_n$ preserves $h_n\circ \psi_n(\{x\}\times C)$.
\item[$\mathbf{C_{6}}$.] For $n\geq 0$ if $X\to \ldots \to X'=\mainmap^q(X)$ is a path in the graph
$\mathcal G(\E_{(n)}^n)$ with $X,X'\in \E_{(n)}^0$ then
$$
G_n^q(x,c)=(\mainmap^q(x),\sigma^q(c)) \quad \text{ for every } (x,c)\in K_X\times C.
$$
\end{itemize}

Before we go deeper in tightening our construction let us briefly describe the main idea of what follows (see also \cite[Section~9]{Cro}). In order to control the dynamics of constructed homeomorphisms we will construct “waste bins” and perform “waste collection”. We construct the “waste bins” in order to control possible invariant measures on the pseudo-arc; without that control we{\black,} up to now{\black,} only have homeomorphisms with rich dynamics but we do not know how rich this dynamics exactly is. For all the dynamics that we do not need we will perform the “waste collection” and “recycle” by sending them back to the “waste bins”. Therefore, the following {\black additional} conditions will be given to ensure that any invariant measure for $\mainmap$ will give measure $0$ to the set $\Phi^{-1}(K)\setminus K\times C$. 
{\black We denote by $\tilde R_{\infty}$ the first return map of $\mainmap$ in $\Phi^{-1}(K)$. We will control, in terms of conditions $\mathbf{C_7}$ and $\mathbf{C_8}$, consecutive returns of points back to $\Phi^{-1}(K)$ which will force that omega limit set under $\tilde R_{\infty}$ for points returning infinitely many times to $\Phi^{-1}(K)$ is included in $K\times C$.}

%{\black control consecutive returns of points to $\Phi^{-1}(K)$ %and our in terms of conditions $\mathbf{C_7},\mathbf{C_8}$ which} will force {\black that} the omega limit set of any point $x\in\Phi^{-1}(K)$ will be included in $K\times C$ with respect {\color{blue} to such consecutive returns.} %to this first return map.

Next we define sequences of closed cubes  $(P^{(n)}_i )_{i\in\N}\subset \mathbb{D}_{\infty}$ defined by $\pi^{-1}_{k_n}(D)$, where $D\subset S_{k_n}\setminus \mathbb{F}_{k_n}$ that will play the role of “waste bins”. The idea is very similar to the one in \cite[p.290]{Cro}, however, we work on levels $S_{k_n}$. 
First, we require that $(P^{(n)}_i)_{i\in\N}\subset \mathbb{D}_{\infty}$ is disjoint from $K\times C$. Assume that the following conditions are satisfied:
\begin{enumerate}[(a)]
\item  $P^{(n)}_i\cap P^{(n)}_{i'}=\emptyset$ for $i\neq i'$, 
\item For every $n$ and every $i$ the set $P^{(n+1)}_i$ is a cylinder defined by cubes in $S_{k_{n+1}}$ and there is $j$ such that $P_i^{(n+1)}\subset P_j^{(n)}$,
\item If $i<j$ and $P_i^{(n+1)}\subset P_{i'}^{(n)}$ and $P_j^{(n+1)}\subset P_{j'}^{(n)}$ then $i'\leq j'$, 
\item If we denote $P^{(n)}=\limsup_{i\to \infty} P_i^{(n)}$ then 
\begin{equation}\label{eq:(d)}
\lim_{n\to\infty}d_H(P^{(n)}, K\times C)=0.
\end{equation}
\end{enumerate}

We introduce a decreasing sequence $(O_k)_{k\in\N}$ of neighborhoods of the Cantor set
$K\times C \subset \mathbb{D}_{\infty}$ such that
$$\bigcap_{k\in\N} O_k = K \times C.$$
The neighbourhoods $(O_k)_{k\in\N}$ will be used to perform the “waste collection”. Roughly speaking, at the k-th step of the construction, we will make sure that the orbit of any point that is not in $O_k$ falls in some waste bin after some time.

\begin{itemize}
\item[$\mathbf{C_7}.$]
For $n> 0$ let $X\to\ldots\to X'=\mainmap^q(X)$ for some $q\geq 0$ be a path in the graph $\mathcal{G}(\E_{(n)}^n)$ with $X,X'\in \E_{(n)}^0$, which is not a
path in the graph $\mathcal{G}(\E_{(n)}^{n-1})$. Then 
$$G^q_n(\psi_{n+1}^{-1}(\rs(\E_{(n+1)}^0)\cap X) \setminus O_n)\subset
\bigcup_{i\in\N}P_i^{(n)}.$$
\item[$\mathbf{C_8}.$]
For $n> 0$ let $X\to\ldots\to X'=\mainmap^q(X)$ for some $q\geq 0$ be a path in the graph $\mathcal{G}(\E_{(n)}^n)$ with $X,X'\in \E_{(n)}^0$. Then for every
$i\in\N$,
$$G^q_n(\psi_{n+1}^{-1}(\rs(\E_{(n+1)}^0)\cap X )\cap P_i^{(n)})\subset
\bigcup_{i'>i}P_{i'}^{(n)}.$$
\end{itemize}

\section{Proof of Theorem~\ref{thm:main}}\label{sec:mainproof}

In this section, we will provide a construction of maps satisfying the conditions from Section~\ref{sec:DRtechnique}. 
Before we can start, we shall need the following two lemmas. The first of them is \cite[Lemma 9.6]{Cro} and the second one is \cite[Corollary A.3]{Cro}.
\begin{lem}\label{lem:diff}
Let $W_1$ and $U_1$ be finite unions of pairwise disjoint closed topological balls tamely embedded in $\R^d$. Assume that every connected component of $U_1$ meets $\R^d \setminus W_1$. Let $Z$ be a compact set inside the interior of $U_1$.
Then there exists a set $W_1'$ which is again a finite union of pairwise disjoint closed topological balls tamely embedded in $\R^d$, which contains $W_1\setminus U_1$ and which does not intersect $Z$.
\end{lem}

\begin{lem}\label{lem:Pjs}
Let $X,X'$ be two copies of the unit cube $[0,1]^d$, and $\alpha$ a homeomorphism between $X$ and $X'$. Let $\Sigma,\Sigma'$ be two totally disconnected tamely embedded compact sets in $\Int(X)$ and $\Int(X')$ respectively, and $\beta$ be a homeomorphism between $\Sigma$ and $\Sigma'$.
Let also $(B_j)_{j\geq0}$ and $(B'_i)_{i\geq0}$ be two sequences of pairwise disjoint topological closed balls, respectively in $\Int(X),\Int(X')$ such that
\begin{enumerate}
    \item the $B_j$'s are disjoint from $\Sigma$, the $B'_i$'s are disjoint from $\Sigma'$; 
    \item $\limsup_{j\to\infty}B_j \subset \Sigma$;
\item  $\limsup_{i\to\infty} B'_i=\Sigma'$;
\item  each $B_j$ is tamely embedded in $X$.
\end{enumerate}
Let $\varphi: \N \to \N$ be any function. Then there exists a homeomorphism $\gamma$ between $X$ and $X'$, which coincides with $\beta$ on $\Sigma$ and with $\alpha$ on the boundary of $X$, and for every $j\geq 0$ there exists an $i > \varphi(j)$ such that $\gamma(B_j) \subset B_i'$.
\end{lem}

We shall also need the following theorem that seems similar to Lemma~\ref{lem:Pjs} at first sight. However, the main difference is that we want to control the exact way in which the cubes are transformed, not only their proper embeddings. The reason for this requirement is that
we have to control images of some pseudo-arcs inside of these cubes, making sure that they are sent into 
selected ones, which we use as a kind of markers. This makes the construction much more difficult, and requiring the use of more advanced techniques. The technical, but folklore, proof is contained in the Appendix.

\begin{theorem}\label{thm:Kozlowski}
Let $X$,$Y$ be two $d$-dimensional cubes and $\alpha$ an orientation preserving homeomorphism from $X$ onto $Y$.  
For  $i=1,\dots, n$ let 
$B_i$ be pairwise disjoint tamely embedded $d$-dimensional cubes in $\Int(X)$, let $C_i$ be pairwise disjoint tamely embedded $d$-dimensional cubes in $\Int(Y)$, and let $\beta_i$ be orientation preserving homeomorphisms from $B_i$ onto $C_i$. 
Then there is a homeomorphism $\gamma$ from $X$ onto $Y$ that coincides with each $\beta_i$ on $B_i$ and with $\alpha$ on the boundary of $X$.
\end{theorem}

Let us explain the difference in the terminology used in Theorem \ref{thm:Kozlowski} and Theorem~\ref{thm:A1}. Although the formulations are not identical, they are equivalent. The slight change in terminology between the main text of the paper and the appendix should not cause many problems, but will allow easier navigation through the literature supporting statements in the appendix.

The definition of a flat cell in Euclidean space may be found in \cite{DevVen}, although the definition had been known for many years{\black: a} $k$-cell or a $(k-1)$-sphere $X$ in $\mathbb{R}^n$ is {\em flat} if there is a homeomorphism $h$ of $\mathbb{R}^n$ such that $h(X)$ is the unit sphere of its type. Note that the definition of a tamely embedded cube is the same as the definition of a flat cell. To stay in consistency with \cite{Cro} we decided to keep the terminology from there. 
The standard definition of locally flat goes back to the summary paper \cite{BrownGluck-BAMS}.
Let $M^n$ be a connected topological manifold of dimension $n$ and $D^n$ an $n$-dimensional cube. An $n-1$-dimensional topological sphere $\Sigma^{n-1}\subset \mathbb{S}^n$ is said to be \emph{locally flat} if for each point $x\in \Sigma^{n-1}$ there is a neighbourhood $U_x$ of $x$ in $\mathbb{S}^{n}$ and a homeomorphism $f$ of $U_x$ into $\mathbb{S}^{n}$ so that $f(U_x)\subset \Sigma^{n-1}$. An embedding $f: D^n\to M^n$ is said to be \emph{locally flat} if $f(\partial D^n)$ is locally flat. For an $\Sigma^{n-1}\subset \mathbb{R}^n$ it follows from \cite{Brown-LFE} that locally flat $\implies$ collared, and from \cite{Brown-GST} that collared $\implies$ flat. The flattening homeomorphism carries the source $n$-cell into the standard $n$-cell. Note also that flat implies locally flat so in fact all three notions are equivalent (see \cite{Brown-LFE}).

We will also need the following result {\black which is a combination of} {\black\cite[Theorem~1]{Bryant} and \cite[Remark~2]{BryantSumners}}. Recall that a homeomorphism $h$ is an {\em $\eps$-push} of $(\mathbb{R}^{\black6},X)$ if some isotopy of $h$ to the identity reduces to the identity outside the $\eps$-neighborhood of $X$, and moves each point of $\mathbb{R}^{\black6}$ along a path of diameter less than $\eps$. 
\begin{theorem}
Suppose that $P$ is a $1$-dimensional compactum in a {\black5}-dimensional hyperplane {\black$Q$} in $\mathbb{R}^{\black6}$, that $\eps>0$, and that $f:P\to {\black Q}$ is an embedding such that $d(x,f(x))<\eps$ for each $x\in P$. Then there exists an $\eps$-push $h$ of $(\mathbb{R}^{\black6},P)$ such that $h|_{P}=f$.
\end{theorem}
Since the result is topological, by appropriate conjugation we can reformulate it as follows.
\begin{theorem}\label{thm:extension-single-cube}
Suppose that $P$ is a $1$-dimensional compactum in a {\black5}-dimensional hyperplane {\black$Q$} in $\mathbb{R}^{\black6}$, that $U$ is a neighborhood of $P$ in $\mathbb{R}^{\black6}$, and that $f\colon P\to {\black Q}$ is an embedding such that $f(P)\subseteq U$. Then there exists a homeomorphism $h:\mathbb{R}^{\black6}\to\mathbb{R}^{\black6}$ such that $h|_{P}=f$ and $h(x)=x$ for all $x\notin U$.
\end{theorem}

{\black Since $f$ from Theorem~\ref{thm:extension-single-cube} is identity outside of $U$, we can naturally view it as a homeomorphism on $\mathbb{D}^6\supset U$.}

From now on we will deal with the Diagram~\ref{smalldiagram}, derived from the proof of Theorem \ref{odo}, which represents the ``solenoidal cube'' $\mathbb{D}_\infty$ as the inverse limit $$\mathbb{D}_\infty=\varprojlim(S_n,\zi),$$ with the map $\mainmap$ arising as the limit map in (\ref{mainmap}) in that proof. The pseudo-arc $\Ps_\infty$ is represented here as the inverse limit of a subsystem $$\Ps_\infty=\varprojlim(\Ps_n,\zi|_{\Ps_n}),$$ where $\Ps_n\subset S_n$, for each $n=1,2,3, ...$.

\begin{equ}[ht!]
\begin{equation*}
\resizebox{0.99\hsize}{!}{
\xymatrix{&\Ps_1\subset S_1\ar@{<-}[rr]_{\zi}\ar@{<-}[dl]_{\hat R_1}&&\Ps_2\subset S_2 
\ar@{<-} [rr]_(.35){\zi}\ar@{<-}[dl]_{\hat R_2}&&\Ps_3\subset S_3
\ar@{<-}[dl]_{\hat R_3} \ar@{<--}[rr] && \Ps_\infty\subset\mathbb{D}_\infty \ar@{<-}[dl]_{\hat R_\infty}\\
\Ps_1\subset S_1\ar@{<-}[rr]^{\zi}&&\Ps_2\subset S_2\ar@{<-}[rr]^{\zi}&&\Ps_3\subset S_3  \ar@{<--}[rr]&& \Ps_\infty\subset \mathbb{D}_\infty}}
\end{equation*}
\caption{General setting from Lemma~\ref{lem:main}.}\label{smalldiagram}
\end{equ}

Let $p\in \mathcal{P}_\infty$ be as in Lemma \ref{lem:choiceofKinA} and let $P\subset \mathcal{P}_\infty\setminus\{0_{\infty}\}$ be a {\black pseudo-arc} such that $p\in P${\black; such pseudo-arc $P$ indeed exists, since $\mathcal{P}_\infty$ is hereditarily equivalent}. Then there is a natural homeomorphism \begin{equation}
\rho:\Lambda\times P\to \cl(\bigcup_{j\in\mathbb{N}}\mainmap^j(P))
\end{equation} such that for any closed set $\Lambda'\subset \Lambda$ we have 
\begin{equation}
\rho(\Lambda'\times \{p\})=\Lambda'.
\end{equation}
Let $P_n:=\pi_n(P)$.

\begin{lem}\label{lem:main}
Let the sets $K\subset \Lambda\subset\Ps_\infty$, $p\in P$, and $\{k_{n}\}_{n\geq 0}$ be chosen as mentioned after the proof of Lemma~\ref{lem:choiceofKinA}. Let also $C\subset P$ be a Cantor set, such that $p\in C$. Then there exist sequences $\{\E_{(n)}^0\}_{n=1}^\infty $, $\{h_n\}_{n=1}^\infty$ and $\{H_n\}_{n=1}^\infty$ satisfying conditions $\mathbf{A_{1,2,3}}$, $\mathbf{B_{1,2,3,5,6,7}}$ and $\mathbf{C_{1,2,3,5,6,7,8}}$. 
\end{lem}

\begin{proof}
We will prove the lemma by induction.
First we construct $\E_{(0)}^0$, $h_0$ and $H_0$.

Let $H_0=h_0=\psi_0=\id$ and $\mathcal{C}_0$ be a chain cover of $\Ps_\infty$ with $\mesh (\mathcal{C}_0)<1$.

We set $Z_{k_0}=P_{k_0}$ and $\E_{(0)}^0=\{\pi_{k_0}^{-1}(D)\}$ for a sufficiently small cube $D\subset S_{k_0}$ containing $P_{k_0}$ in its interior, so that the sets $\hat R_{k_0}^j(D)$ are pairwise disjoint for $j=0,\ldots,2^{k_0}-1$ and $\bigcup_{j=0}^{2^{k_0}-1}\mainmap^j(\pi_{k_0}^{-1}(D))\subset \rs (\mathcal{C}_0)$. 

Recall that $\cl(\bigcup_{n\in\mathbb{N}}\mainmap^n(P))=\pi_1^{-1}(P_1\cup\hat{R}_1(P_1))$.
For every (in fact unique) $Y\in \E_{(0)}^{0}$ we have  $p\in Y$ and by the definition of $\Lambda_0$ in Lemma~\ref{lem:choiceofKinA} we have $K_Y\subset K\subset \pi_{k_0}^{-1}(p_{k_0})$ and therefore $K_Y\times C\subset \pi_{k_0}^{-1}(P_{k_0})$.
If we put $i_Y=0$ then we have
 
\begin{equation}
\psi_0(K_Y\times C)=K_Y\times C\subset \pi_{k_0}^{-1} (\hat{R}^{i_Y}_{k_0}(Z_{k_0}))\subset \text{int} Y.\label{con:zko-start}
\end{equation}
By definition of $k_0$ the graph $\mathcal G(\E_{(0)}^0)$ has no cycle. In fact it has no edge by the definition of $\Lambda_{0}$ in Lemma~\ref{lem:choiceofKinA}.
Observe that conditions $\mathbf{A_{1,2,3}}$, $\mathbf{B_{1,2}}$ and  $\mathbf{C_{1,2}}$ are trivially satisfied for $n=0$ (with exception of $\mathbf{A_{1}}\text{(a)}$, but as we mentioned, the graph $\mathcal G(\E_{(0)}^0)$ has no cycle).
Condition $\mathbf{B_{3}}$ is important for large $n$, so we can ignore it at this point, and $\mathbf{B_{5,7}}, \mathbf{C_{5}}$
are satisfied directly by the definition. Since there is no edge in the graph $\mathcal G(\E_{(0)}^0)$, 
$\mathbf{B_{6}}$ and $\mathbf{C_{6,7,8}}$
are also trivially satisfied. Finally, we may view the identity maps $H_0,h_0$ as extensions of identity on the space $S_{k_0}$
onto $\spaceX$ via maps $\z_\infty$.

Fix $n\geq 0$ and assume that we have sets $\E_{(i)}^0$ satisfying $\mathbf{A_{1,2,3}}$, $\mathbf{B_{1,2,3,5,6,7}}$ and $\mathbf{C_{1,2,3,5,6,7,8}}$ for $i=0,\ldots, n$ and that there exists a pseudo-arc $Z_{k_n}\subset \Ps_{k_n}$ that contains $p_{k_n}$, such that for every $X\in \E_{(n)}^0$ if $K_X\neq\emptyset$ then there is an $i_X$ such that
\begin{equation}
\psi_n(K_X\times C)\subset \pi_{k_n}^{-1} (\hat{R}^{i_X}_{k_n}(Z_{k_n}))\subset X\label{con:paZ}
\end{equation}
and $\hat{R}^j_{k_n}(Z_{k_n})$ are pairwise disjoint for $j=0,...,2^{k_n}-1$. \\
We additionally assume that for each $j<n+1$ there is a closed set $V^+\subset S_{k_j}$ such that $p_{k_j}\in V^+$ and for each $X^*\in \E_{(j)}^0$ there exists a $0\leq t<2^{k_j}$ such that $X^*=\pi_{k_j}^{-1}(\hat R_{k_j}^t(V^+))$.
We also assume that $\psi_n:\spaceX\to \spaceX$ is defined by a homeomorphism of $S_{k_n}$, which is then extended to $\spaceX$ by recursive applications of maps $\z_\infty$.

{\black For any $i\in\mathbb{N}$ recall that $\mathbb{F}_i$ is the set defined in Theorem~\ref{odo} consisting of all the points of $S_{i+1}$} on which $\z_\infty$ is not a local homeomorphism. 

Let $m=k_{n+1}$ and set $Z_m\subset S_m$ to be the connected component of $\pi_m(\pi_{k_n}^{-1}(Z_{k_n}))=\z_\infty^{k_n-m}(Z_{k_n})$ that contains $p_m$. Note that $Z_{k_n}=\z_\infty^{m-k_n}(Z_m)$ and $Z_m$ are pseudo-arcs, which holds by the fact that any pseudo-arc is an acyclic\footnote{A set is said to be acyclic if it is the intersection of a descending family of tamely embeded topological balls.} continuum and $\mathbb{F}_{k_n}\cap Z_{k_n}=\emptyset$, 
hence the defining intersection of topological balls lifts to homeomorphic copies in the pre-image of $\z_\infty$. 

Fix $X\in \E_{(n)}^0$ such that $p\in X$. Note that by definition $\pi_{k_n}(X)\cap\mathbb{F}_{k_n}=\emptyset$, and since $\hat R_{k_n}$ is also a homeomorphism, $\hat R_{k_n}^{j}(\pi_{k_n}(X))\cap \mathbb{F}_{k_n}=\emptyset$ for all $j$. Recall that $\pi_{k_n}(X)$ is a cube containing $Z_{k_n}$ in its interior and intersecting $\hat R_{k_n}^{j}(Z_{k_n})$ if and only if $j=0(\text{mod }2^{k_n})$, with $\hat R_{k_n}^{2^{k_n}}(Z_{k_n})=Z_{k_n}$. Let $X_m$ be a connected component of $\pi_m(X)$ and set $Y:=\pi^{-1}_m(X_m)$. By the construction of $Z_{m}$, the fact that $\psi_n$ is defined coordinate-wise, and  \eqref{con:paZ} we get
\begin{equation}
\psi_n(K_Y\times C)\subset \pi_{m}^{-1} (\hat{R}^{i_Y}_{m}(Z_{m}))\subset \Int Y.\label{con:KY}
\end{equation}

By {\black Theorem~\ref{odo}}, let $\mathcal{W}^\eps_i$ be a finite $\eps$-cover of $\Ps_i$ by closed {\black tamely embedded} cubes,  %of the form $Q_W\times Q'\subset \mathbb{D}^2\times \mathbb{D}^2\subset S_i$, where both $Q_W$ and $Q'$ are closed discs,  
that is $\mathcal P_i\subset \bigcup_{W\in \mathcal{W}^\eps_i} \Int W$,
and $\mesh \mathcal{W}^\eps_i<\eps$. 
%We use the same disc $Q'$ each time, since $\mathcal P_i$ was defined in $\mathbb{D}^2$ and then embedded in $\mathbb{D}^2\times \mathbb{D}^2$; see Remark \ref{rem:forbidden}.
{\black By Theorem~\ref{odo} we know} that there is a unique element 
$W_{p_i}\in \mathcal{W}^\eps_i$ such that $p_i\in W_{p_i}$.
Set 
$$\mathcal{W}^\eps(i):=\{\pi_i^{-1}(W):W\in \mathcal{W}^\eps_i\}.$$
Clearly, for each $\eps$ and $i$, $\mathcal{W}^\eps(i)$ is a $(\lambda_i(\eps) \eps+1/2^{i})-$cover of $\Ps_\infty$, where $\lambda_i(\eps)$ is a constant provided by uniform continuity of $\z_\infty^i$ and thus clearly $\lim_{\eps\to 0}\lambda_i(\eps)=0$.
In particular, if $\eps$ is sufficiently small, then every element $W\in \mathcal{W}^\eps(m)$ satisfies $\diam W<2^{-m+1}$. {\black Since $m$ is already fixed, we can also assume that $\eps$ is sufficiently small, so that $d_H(Z_m,\mathbb{F}_m)>\eps$.}

Denote by $Z$ the unique connected component of $\pi_{k_n}^{-1}(Z_{k_n})$ containing $p$. Clearly $Z$ is a pseudo-arc and $\pi_m(Z)=Z_m$ {\black is also a pseudo-arc. Let $\mathcal{Q}_m^\eps$ be a chain subcover of $\mathcal{W}^{\eps}_m$ which covers $Z_m$ and is provided by Theorem~\ref{odo}. Preliminarily, denote $U=\bigcup\{\pi_m^{-1}(W):W\in \mathcal{Q}^\eps_m\}$ and set $U_m=\pi_m(U)=\bigcup \mathcal{Q}^\eps_m$; later we will make $U$ and as a consequence $U_m$ smaller sets. The set $U_m$ is the union of elements of chain cover of $Z$ provided by Theorem~\ref{odo}. Hence, if we follow the proof of claim~\ref{clm:extension}, then we see that $U_m$ is a neighborhood of a set $E$ such that $Z_m\subset E=\cup_{i\in \mathcal{I}} D_i^{\leftarrow}\subset U_m$ where $(D_k^{\leftarrow})_{i\in \mathcal{I}}$ is a chain with a finite index set $\mathcal{I}\subset \{0,\ldots,j\}$  ($U$ is disjoint with $\mathbb{F}_m$), therefore $E\supset Z_m$ is a cell-like set contained in $\mathbb{R}^3\times\{0,0,0\}\subset \mathbb{R}^6$ and we can find a tamely embedded cube $\tilde U_m$ in $S_m$ 
such that $Z_m\subset \textrm{int} \tilde U_m\subset U_m$. Then we replace $U$ by $\pi_m^{-1}(\tilde U_m)\cap U\supset Z$ and $U_m$ by $\tilde U_m$.}
% If $\eps$ is sufficiently small then the elements of $\mathcal{W}_m^\eps$ defining $\mathcal{Q}_m^\eps$ are a chain cover of $Z_m$ consisting of cubes $Q_W\times Q'$, so by Zoretti's Theorem \cite[p. 109]{Whyburn} we may require that $U_m$ is a tamely embedded cube.

If we take sufficiently small $\eps>0$ then

\begin{equation}\label{disjoint}\hat R_{m}^i(U_m)\cap \hat R_{m}^j(U_m)=\emptyset \textrm{ for any } 0\leq i <j <2^m,\end{equation} 

By the assumption on the sets $W_{p_i}$, there is a unique element $V\in \mathcal{Q}^\eps_m$ that contains the point $p$. There exists a pseudo-arc  $A\subset Z\cap \Int V\subset \Ps_\infty$ such that $p\in A$.
%{\color{blue} \bf  By \cite{LeCY} (see also \cite[Proposition 3.1]{Bonino}) the connected components of $V_m\cap U_m\cap (\mathbb{D}^2\times\{(0,0)\})$ are topological disks, and so the component $V^p_m$ of $V_m\cap U_m$ that contains $p$ is a tamely embedded cube.} For simplicity of notation we shall 
{\black
Let $V=\pi_m^{-1}(V_m^p)$ where $V_m$ is a tamely embedded cube. Now perform construction similar to previous paragraph obtaining a tamely embedded cube $V'_m\supset \pi_m(A)$ which is small neighborhood of $A$ and define $V'=\pi_m^{-1}(V_m')$. Note that $p\in A\subset \textrm{int} V'$. The set $V_m'$ and as a consequence $V'$, has to be sufficiently small neighborhood of $\pi_m(A)$, which will be specified later in the proof.}

%Let $Q\times Q'\subset \Int \pi_m(V)$ be a product of discs such that $\pi_m(A)\subset Q\times Q'$. Set
%\begin{equation}\label{def:V'}
%V'=\pi_m^{-1}(Q\times Q')
%\end{equation}
%and observe that $p\in A\subset V'$. 
%Note that by definition $p\in \Int V$ and so we may also assume that $p\in \Int V'$.
%We will specify later some of its additional properties.
 
 Let $V_j:=\pi_j(V)$ for each $j$ and observe that $V_j$ is a cube, for $j\leq m$. 
Since $\pi_{k_n}(Z)=Z_{k_n}\subset \Int \pi_{k_n}(X)$, we may assume that $V_{k_n}\subset \pi_{k_n}(X)$, and that each of the disjoint sets $\hat R_{k_n}^i(V_{k_n})$
for $i=-n-1,\ldots, s_n+n+1$ either is contained in $\Int(\pi_{k_n}(J))$, for some $J\in \E_{(n)}^{n+1}$, or $\hat R_{k_n}^i(V_{k_n})\cap\pi_{k_n}(\rs (\E_{(n)}^{n+1}))=\emptyset$.

Let $$\tilde{\E}_{(n+1)}^{0}=\{\hat R_{\infty}^i(V):\hat R_{\infty}^i(V)\cap K\neq\emptyset, i=0,...,s_{n+1}\},$$ and 
\begin{equation}
    \E_{(n+1)}^{0}=\{\hat R_{\infty}^i(V'):\hat R_{\infty}^i(V)\cap K\neq\emptyset, i=0,...,s_{n+1}\}.\label{def:En+1}
\end{equation}
By the definition of $s_{n+1}$ and \eqref{def:En+1}, we obtain that $K\subset \rs (\E_{(n+1)}^{0})$.
 
 Note that $\mathbf{A_{3}}$ will be satisfied in the limit if $\eps$ is chosen small enough, and $\mathbf{A_{2}}$ is satisfied by the choice of the set $K$
 (see Lemma~\ref{lem:choiceofKinA}\eqref{lem:choiceofKinA:4}). Condition $\mathbf{A_{1.b}}$ follows from the following observation. 
 Fix $j<n+1$ and let $D\in \E_{(j)}^0$ be a neighborhood of $p$. It follows that $V\subset\Int(D)$, provided that $\eps$ was chosen small enough. For each $i\in\mathbb{Z}$ we have $\mainmap^i(V)\subset \Int \mainmap^i(D)$. 
 By induction hypothesis, each element of $\E_{(j)}^{j+1}$ is an iterate of $\E_{(j)}^{0}$ defined by \eqref{def:En+1}, hence it is of the form $\hat R_{\infty}^i(D)$ for some $i$.

 Note that if $Y\in \E_{(j)}^{0}$ then $Y\cap K\neq \emptyset$, and there is $Y'\in \E_{(n+1)}^{0}$ such that $Y'\cap Y\cap K\neq \emptyset$, since $\E_{(n+1)}^{0}$ covers $K$.
By definition we have $Y=R_{\infty}^i(D)$ and $Y'=R_{\infty}^{i'}(V)$ for some $i\leq s_j$, $i'\leq s_{n+1}$, and $\hat R_{k_n}^{i'}(V_{k_n})\subset\Int(\pi_{k_n}(J))$, for some $J\in \E_{(n)}^{0}$ such that $J\subset Y$.
Consequently, 
 if $\hat R_{\infty}^i(D)\in \E_{(j)}^{j+1}$ then $\hat R_{\infty}^i(V')\in \E_{(n+1)}^{j+1}$, and so each element from $\E_{(j)}^{j+1}$ contains at least one element from $\E_{(n+1)}^{n+2}$. 
 
 Since $p_m$ is $2^m$ periodic under $\hat R_m$, and for each $X^*\in \E_{(j)}^{j+1}$ we have $|\pi_j(X^*)\cap %O_{\hat R_j}
 {\black \{
 \hat R_j^s(p_j): s\geq 0\}}|=1$, taking $\eps$ small enough we easily obtain that for every  $X^*\in \E_{(j)}^{j+1}$ and $Y^*\in \E_{(n+1)}^{n+2}$ either $X^*\cap Y^*=\emptyset$, or $X^*\subset \Int(Y^*)$. Indeed $\mathbf{A_{1.b}}$ holds.

Condition $\mathbf{A_{1.a}}$ is clear because elements of $\E_{(n+1)}^{0}$ are defined by $\hat R_{\infty}^i(V')$ and these iterates are disjoint
for $i=0,\ldots, 2^m-1$, and furthermore $\hat R_{\infty}^i(V')\notin \E_{(n+1)}^{0}$ for $i=s_{n+1}+1,\ldots, s_{n+1}+n+2$ and $i=-n-2,\ldots, -1$. So $\E_{(n+1)}^{0}$ is $(n+2)$-iterable, and $\mathcal G(\E_{(n+1)}^{n+1})$ has no cycle.

By conditions $\mathbf{A_{1.a}},\mathbf{A_{1.b}}$, and since it is clear from the definition that $\rs (\E_{(n+1)}^{0})\subset \rs (\E_{(n)}^{0})$, to check $\mathbf{A_{1.c}}$ we only need to check that
$$  
\rs(\E_{(n+1)}^{n+2}) \cap R_\infty^k( \rs(\E_{(n)}^0) ) = R_\infty^k( \rs(\E_{(n+1)}^0) ),
$$
for all $|k|\leq n+1$. 

Condition  $\rs(\E_{(n+1)}^{n+2}) \cap R_\infty^k( \rs(\E_{(n)}^0) )\neq\emptyset $ means that there is an $0\leq i\leq s_{n+1}$, $|j|\leq n+2$ and $|k|\leq n+1$ such that $\mainmap^i(V')\in \E_{(n+1)}^0$ and $\mainmap^{i+j}(V')\cap \mainmap^k(\rs(\E_{(n)}^0)\neq\emptyset$. By the fact that $\mainmap^k(\rs(\E_{(n)}^0))\subset \rs(\E_{(n)}^{n+1})$ and $\E_{(n+1)}^{n+2}$ refines $\E_{(n)}^{n+1}$, there exists $X^*\in \E_{(n)}^{0}$ such that $\mainmap^{i+j-k}(V')\subset X^*$. There are two possibilities. If $\mainmap^{i+j-k}(V')\in \E_{(n+1)}^0$ then $\mainmap^{i+j}(V')\subset \mainmap^k(\rs(\E_{(n+1)}^0))$ and we are done. 
If $\mainmap^{i+j-k}(V')\notin \E_{(n+1)}^0$ then by the fact that $|j-k|\leq 2n+3$ we obtain $\mainmap^{i+j-k}(V')\cap K=\emptyset$. But  $\pi_{k_n}(\mainmap^{i+j-k}(V')\cap\Lambda)\subset \pi_{k_n}(X^*\cap\Lambda)\subset\pi_{k_n}(K)$ by the construction in Lemma~\ref{lem:choiceofKinA}, so $\mainmap^{i+j-k}(V')\cap\Lambda\subset\bigcup_{i\geq n+1}\Lambda_i$. Observe that if 
$x\in \Lambda_{n+i}\setminus \Lambda_n$ then by \eqref{def:AnK} from Lemma~\ref{lem:choiceofKinA}
$$2^{k_{n+i}-1}-4(n+1+i)2^{k_{n+i-1}} \leq x_{k_{n+i}}\leq 2^{k_{n+i}-1}+4(n+1+i)2^{k_{n+i-1}}$$
while 
$$x_{k_{n+i}} (\text{mod }2^{k_n})> 4(n+1)2^{k_{n-1}}$$
and
$$
x_{k_{n+i}} (\text{mod }2^{k_n})<2^{k_n-1}-4(n+1)2^{k_{n-1}}.
$$
Therefore
$$
2^{k_{n+i}-1}-4(n+1+i)2^{k_{n+i-1}}+4(n+1)2^{k_{n-1}} \leq x_{k_{n+i}}\leq 2^{k_{n+i}-1}+4(n+1+i)2^{k_{n+i-1}}-4(n+1)2^{k_{n-1}}
$$
while still $x\in \Lambda_{n+i}$.
This shows that
if $x\in \bigcup_{i\geq n+1}\Lambda_i$ then $\mainmap^r(x)\notin K$ for $|r|\leq 4n+1$.

This implies $\mainmap^i(V')\cap K=\emptyset$, which contradicts the choice of $i$. We have just verified that conditions $\mathbf{A_{1}}-\mathbf{A_{3}}$ hold for $\E_{(n+1)}^0$. 

Now, we are going to define a homeomorphism $\hat h\colon S_m\to S_m$, from which we will obtain $h\colon \mathbb{D}_{\infty}\to \mathbb{D}_{\infty}$.

First, both $U_m$ and $V_m$ are{\black,} by the construction{\black,} tamely embedded cubes, in particular are homeomorphic.
By Theorem~\ref{thm:extension-single-cube} we can compose this homeomorphism with a self-homeomorphism of $V_m$ to obtain
a homeomorphism $\hat h\colon U_m\to V_m$ such that $\hat h(p_m)=p_m$, and $\hat h(Z_m)=A_m$. Note that $\{p_m\}=U_m\cap \pi_m(\Lambda)$, provided that $\eps$ is small enough.

We extend $\hat h^{-1}$ onto pairwise disjoint sets $\hat{R}_m^{i}(V_m)$ for each $i=1,\ldots,2^m-1$, 
such that $\hat R_\infty^i(V)\in \tilde{\E}_{(n+1)}^{0}$,
by the formula 
\begin{equation}
\hat h^{-1}|_{\hat{R}_m^{i}(V_m)}=\hat{R}_m^{i} \circ \hat h^{-1}|_{V_m} \circ \zeta_\chi \circ \hat{R}_m^{-i}\label{CHI}
\end{equation} 
where $\zeta_\chi\colon V_m\to V_m$ is a homeomorphism fixed for the connected component $\chi$ of $\tilde{\E}_{(n+1)}^n$ containing $\hat R_\infty^i(V)$. We will specify this homeomorphism later. At this point we only say that it is obtained by Theorem~\ref{thm:extension-single-cube} as extension of some
homeomorphism between tamely embedded Cantor sets to the whole $V_m$.

The proper definition of $\hat{h}$ is supported by condition \eqref{disjoint}. Note that, for $\eps$ small enough, from \eqref{con:KY} we get
\begin{equation}\label{conIMP}
\textrm{if $\hat{R}_m^{i}(V_m)\subset \pi_m(X^*)$, for some $X^*\in \E_{(n)}^0$, and $\mainmap^i(V)\in \tilde \E^0_{(n+1)}$ then $\hat{R}_m^{i}(U_m)\subset \pi_m(X^*)$. }
\end{equation}

We perform additional extension of the above, when necessary, as follows. Assume that we have an edge $X\to \mainmap (X)$ in graph $\mathcal{G}(\E^n_{(n)})$
and $\hat h^{-1}$ was defined on $\hat{R}_m^{i}(V_m)$ but not on $\hat{R}_m^{i+1}(V_m)$ or vice-versa, where $\pi_m^{-1}(\hat{R}_m^{i}(V_m))\subset X$.
Then we use \eqref{CHI} to define $\hat h^{-1}$ also on the other set. We repeat this procedure several times, until there is no edge in graph $\mathcal{G}(\E^n_{(n)})$ with the above problem. We have to justify that this procedure does not depend on the order of consecutive extensions. 
If $i\in \{0,\ldots, 2^m -1\}$ is such that $\mainmap^i(V)\cap K\neq \emptyset$ and $q>0$ is the smallest number such that $\mainmap^{i+q}(V)\cap K\neq \emptyset$, then by \ref{lem:choiceofKinA:3}. from Lemma~\ref{lem:choiceofKinA} and the 
fact that $q\leq 2n+2$ we see that $i+q\in \{0,\ldots, 2^m -1\}$.
Therefore, if $X,X'\in \tilde \E_{(n+1)}^0$ and there is a path from $X$ to $X'$ in $\mathcal{G}(\E^{n}_{(n+1)})$ then extending $\hat h$ forward from $X$ to $X'$ or backward from $X'$ to $X$, the result will 
be the same, because $X=\mainmap^{i}(V)$ and $X'=\mainmap^{i+q}(V)$. 

By a computation similar to the one verifying $\mathbf{A_{1.c}}$, we know that if $\mainmap^i(V)\in \tilde{\E}_{(n+1)}^{n}$ then there is $X\in \E_{(n)}^{n}$ such that $\mainmap^i(V)\subset \Int X$.
In particular $\hat R_m^i(V_m)\subset \Int \pi_m(X)$. Note that connected components of both $\pi_m(X)$ and $\hat R_m^i(V_m)$ are cubes. In particular, $\hat R_m^i(V_m)$ must be a subset of one of the connected components of $\pi_m(X)$, 
 and by \eqref{conIMP} the cube $\hat R_m^i(U_m)$ is a subset of the same component. We also see that $\hat h^{-1}$ defined so far is injective. 
Now we are going to extend $\hat h^{-1}$ to cubes $\E_{(n)}^n$.
For every maximal path in the graph $\mathcal G (\E_{(n)}^n)$ we use directly Theorem~\ref{thm:Kozlowski} only on the first element $X\in \E_{(n)}^n$ of the path, and then copy it on further elements of the path 
\begin{equation}\label{eq:pathFromX}
X\to \mainmap (X) \to \mainmap^2(X)\to \ldots
\end{equation}
by the procedure analogous to \eqref{CHI}, i.e.
$\hat h^{-1}|_{\hat{R}_m^{i}(X)}=\hat{R}_m^{i} \circ \hat h^{-1}|_{X} \circ \hat{R}_m^{-i}$. 

\begin{figure}
\begin{tikzpicture}[scale=2]
\draw[red](-0.7,-0.7)--(-0.2,-0.7)--(-0.2,-0.2)--(-0.7,-0.2)--(-0.7,-0.7);
\draw[rotate=90](-0.7,-0.7)--(-0.2,-0.7)--(-0.2,-0.2)--(-0.7,-0.2)--(-0.7,-0.7);
\draw[rotate=180](-0.7,-0.7)--(-0.2,-0.7)--(-0.2,-0.2)--(-0.7,-0.2)--(-0.7,-0.7);
\draw[red,rotate=270](-0.7,-0.7)--(-0.2,-0.7)--(-0.2,-0.2)--(-0.7,-0.2)--(-0.7,-0.7);
\draw[-latex] (-0.1,0.45)--(0.1,0.45);
\draw[-latex] (-0.1,-0.45)--(0.1,-0.45);
\node at (0,0.55) {\tiny$\hat R_{\infty}$};
\node at (0,-0.55) {\tiny$\hat R_{\infty}$};
\filldraw[black] (-0.5,-0.5) rectangle (-0.55,-0.55);
\filldraw[black] (-0.35,-0.35) rectangle (-0.4,-0.4);
\filldraw[black] (-0.35,0.5) rectangle (-0.4,0.55);
\filldraw[black] (-0.5,0.35) rectangle (-0.55,0.4);

\filldraw[black] (0.5,-0.35) rectangle (0.55,-0.4);
\filldraw[black] (0.35,-0.5) rectangle (0.4,-0.55);
\filldraw[black] (0.35,0.35) rectangle (0.4,0.4);
\filldraw[black] (0.5,0.5) rectangle (0.55,0.55);

\node at (-1.05,0) {\black\small$\hat h$};

\draw[-latex] (0.525,-0.3)--(0.525,0.45);
\draw[-latex] (-0.375,-0.3)--(-0.375,0.45);
\draw[-latex] (0.375,-0.45)--(0.375,0.3);
\draw[-latex] (-0.525,-0.45)--(-0.525,0.3);
\draw[latex-,red] (-0.8,0.5) arc (90:270:0.5);
\end{tikzpicture}
\caption{{\black Commuting of $\hat h$ and $\mainmap$ along edges of  $\mathcal G(\E_{(n)}^n)$}.}
\end{figure}

This way $\hat h$ and $\mainmap$ will commute along edges of  $\mathcal G(\E_{(n)}^n)$. When applying Theorem~\ref{thm:Kozlowski} to define $\hat{h}\colon X\to X$ we require that $\hat{h}$ is the identity on the boundary.
Our procedure is consistent with previous construction in \eqref{CHI}, namely $\hat h$ defined up to now commutes with $\hat{R}^i_m$ already along edges of $\mathcal G(\E_{(n+1)}^{n})$. So far $\hat h$ is defined on disjoint cubes tamely embedded in $S_m$ and is identity on their boundary, so it extends to $S_m$ as the identity on the complement
of these cubes. By definition of $\hat{h}$, and since $0_m\not\in \pi_m(\rs (\E_{(n)}^n))$, we have that $\hat{h}(p_m)=p_m$ and $\hat{h}(0_m)=0_m$. 
Lifting inductively $\hat h$
to $S_{m+k}$ through $\varphi_\infty$, for each $k$, we obtain a homeomorphism $h_{n+1}\colon \spaceX \to \spaceX$ such that $h_{n+1}(p)=p$.

{\black Now is the time to specify the additional properties of $V'$. % from (\ref{def:V'}). 
Note that directly from the construction the conditions $\mathbf{B_1},\mathbf{B_2}$ are satisfied, because $V'\subset V$. Since $h_{n+1}$ is already defined and $\mathcal{C}_n$ is a cover of $\Ps_\infty$ by the definition, we can find a chain cover $\mathcal{D}_{n+1}$ of $\Ps_\infty$ refining $\mathcal{C}_{n}$ such that 
}

\begin{equation}
 \mesh \big( (h_{n+1}\circ \psi_{n})^{-1} (\mathcal{D}_{n+1}) \big)  < 2^{-n-1}. 
\end{equation}

Observe that $\supp h_{n+1}\subset \rs({\Mm_n})$ (where $\supp h_{n+1}$ denotes the support of $h_{n+1}$) which holds by $\mathbf{B_7}(c)$ from step $n$ and the way we constructed $h_{n+1}$; see \eqref{eq:pathFromX}.

Furthermore 
$$
(h_{n+1}\circ \psi_{n})^{-1} (\overline{\rs(\mathcal{D}_{n+1})})\subset \psi_{n}^{-1}(\overline{\rs(\mathcal{D}_{n+1})}\cup \supp h_{n+1})\subset \psi_{n}^{-1}(\rs({\Mm_n}))
$$
and therefore by $\mathbf{B_7}(b)$, for every $k<n+1$ we have  
\begin{equation}
(h_{n+1}\circ \psi_{n})^{-1}( \overline { \rs ( \mathcal{D}_{n+1}) })
\subset  \psi_k^{-1}( \rs ( \Mm_{k}))
\end{equation}
and
$(h_{n+1}\circ \psi_{n})^{-1} (\mathcal{D}_{n+1})$ is crooked inside $\psi_{n}^{-1} (\Mm_{n}).$ Since $A\subset Z\subset \Ps_\infty$ and $\Ps_\infty$ contains all inverse sequences passing through $\pi_m(A)=A_m$, we
can find a set $V'$ as defined before, with the additional property that $A\subset V'\subset V\cap \rs(\mathcal{D}_{n+1})$ and $\mainmap^{i}(V')\subset \rs(\mathcal{D}_{n+1})$ for all $|i|<2^{k_{n+1}}+3(n+1)$.
In particular, it implies that $\rs (\E^{n+1}_{(n+1)})\subset \rs(\mathcal{D}_{n+1})$.
If additionally $H_{n+1}$ is constructed in such a way that $\mathbf{C_{1}}$ holds for $n+1$, then $\mathbf{B_{7}}(a)-\mathbf{B_{7}}(c)$ are satisfied for the chain cover of $\Ps_\infty$ given by $\mathcal{C}_{n+1}=H_{n+1}(\mathcal{D}_{n+1})$.

Next, let $X\in \E_{(n+1)}^{n+2}\setminus \E_{(n+1)}^{n}$.
By the properties of $\mathbf{A_{1.c}}$ we obtain that $X\cap \rs (\E_{(n)}^n)=\emptyset$. Combining this with the definition of $h_{n+1}$ around (\ref{eq:pathFromX}) it follows that 
$$
\psi_n^{-1}\circ h_{n+1}^{-1}(X)= \psi_n^{-1}(X).
$$
This ensures that $\mathbf{B_3}$ holds, provided that $\eps$ was small enough.

If $\mathbf{C_1}$ holds, then $\mathbf{B_5}$ holds as well by \eqref{con:KY}, definition of $V'$ and provided that $\eps$ was chosen to be small enough.

To check $\mathbf{B_6}$, take any $X,X'\in \E_{(n+1)}^0$ such that there is a path in $\mathcal{G}(\E_{(n+1)}^{n+1})$ from $X$ to $X'=\mainmap^q(X)$.
Recall that $\E_{(n+1)}^0$ is compatible with $\E_{(n)}^0$ for $n+1$ iterates and $\E_{(n+1)}^0$ refines $\E_{(n)}^0$, so path in $\mathcal{G}(\E_{(n+1)}^{n+1})$
induces a path on $\mathcal{G}(\E_{(n)}^{n+1})$ between some $Y,Y'\in \E_{(n)}^0$ such that $X\subset Y$ and $X'\subset Y'$.
By definition of $g_{n+1}$ in \eqref{eq:10} we have $g_{n+1}^q=\psi_n^{-1}\circ h_{n+1}^{-1}\circ \mainmap^q \circ h_{n+1}\circ \psi_n$.
For $(x,c)\in K_X\times C$ by \eqref{con:KY} we have that $\psi_n(x,c)\in X$ and $h_{n+1}(X)\subset X$, therefore $h_{n+1}\circ \psi_n(x,c)\in \rs (\E_{(n+1)}^0)$ (cf. definition of $\hat h$ in \eqref{eq:pathFromX}).
If the whole path is contained in $\mathcal G(\E_{(n+1)}^n)$, and so induced path is contained in $\mathcal G(\E_{(n)}^n)$, then we can apply commutativity provided by $\mathbf{B_2}$ and then
$$
\psi_n^{-1}\circ h_{n+1}^{-1}\circ \mainmap^q \circ h_{n+1}\circ \psi_n=\psi_n^{-1}\circ h_{n+1}^{-1}\circ h_{n+1} \circ \mainmap^q \circ \psi_n
$$
which gives by $\mathbf{C_6}$ on $K_Y\times C \supset K_X\times C$ in the step $n$ that
$$
g_{n+1}^q(x,c)=G_n^q(x,c)=(\mainmap^q(x),\sigma^q(c)).
$$
This proves that $\mathbf{B_6}(a)$ holds. Now we will check $\mathbf{B_6}(b)$. It is the first place where definition of $\zeta$ in \eqref{CHI} becomes important. 
It is the place where we will finally specify the precise conditions it must satisfy.
Clearly, there is an isomorphism between $\mathcal G(\tilde{\E}^{n+1}_{(n+1)})$ and $\mathcal G(\E^{n+1}_{(n+1)})$ identifying the vertex $\mainmap^{i}(V)$ with $\mainmap^{i}(V')$, for each
$i\in \{0,\ldots, 2^m -1\}$, such that    $\mainmap^{i}(V)\in \tilde{\E}^{n+1}_{(n+1)}$. Note that if $X,X'$ are consecutive elements of $X,X'\in \E_{(n+1)}^0$ on the path in $\mathcal{G}(\E_{(n+1)}^{n+1})$, and it is not proper path in $\mathcal{G}(\E_{(n+1)}^n)$ then we have two different components $\chi,\chi'$ of $\mathcal G(\E_{(n+1)}^n)$ such that $h_{n+1}^{-1}$ on $X$ is defined by
$$
h_{n+1}^{-1}|_X=\hat{R}_\infty^{i} \circ h_{n+1}^{-1}|_{V} \circ \hat{\zeta}_\chi \circ \hat{R}_\infty^{-i}
$$
while on $X'$ it is defined by
$$
h_{n+1}^{-1}|_{X'}=\hat{R}_\infty^{i+q} \circ h_{n+1}^{-1}|_{V} \circ \hat{\zeta}_{\chi'} \circ \hat{R}_\infty^{-i-q}
$$
where $\hat{\zeta}_{\chi},\hat{\zeta}_{\chi'}$ are extensions to $\mathbb{D}_{\infty}$ of homeomorphisms $\zeta_{\chi}$, $\zeta_{\chi'}$ defined coordinate-wise by commutative diagrams involving $\varphi_{\infty}$.
Let $Y\in \E_{(n)}^0$ be such that $X\subset Y$. Then $X^*=h_{n+1}^{-1}(X)\subset Y$ and by $\mathbf{B_5}$ and provided that additionally $H_{n+1}$ is constructed in such a way that $\mathbf{C_{1}}$ holds, for any
$(x,c)\in K_X\times C$ we have $\psi_{n+1}(x,c)\in X$ and $H_{n+1}(X)=X$, thus
$\psi_n(x,c)\in X^{*}$.

Then by definition 
\begin{eqnarray*}
g_{n+1}^q(x,c)&=&\psi_{n}^{-1}\circ h_{n+1}^{-1}\circ \mainmap^q \circ h_{n+1}\circ\psi_{n}(x,c)\\
&=&\psi_{n}^{-1}\circ h_{n+1}^{-1}|_{X'}\circ \mainmap^q \circ h_{n+1}|_{X^*}\circ\psi_{n}(x,c)
\end{eqnarray*}
But for the middle term we have
\begin{eqnarray*}
h_{n+1}^{-1}|_{X'}\circ \mainmap^q \circ h_{n+1}|_{X^*}&=&\hat{R}_\infty^{i+q} \circ h_{n+1}^{-1}|_{V} \circ \hat{\zeta}_{\chi'} \circ \hat{R}_\infty^{-i-q}\circ \mainmap^q\circ \hat{R}_\infty^{i} \circ \hat{\zeta}_\chi^{-1} \circ (h_{n+1}^{-1}|_{V})^{-1} \circ \hat{R}_\infty^{-i}\\
&=&\hat{R}_\infty^{i+q} \circ h_{n+1}^{-1}|_{V} \circ \hat{\zeta}_{\chi'} \circ \hat{\zeta}_\chi^{-1} \circ (h_{n+1}^{-1}|_{V})^{-1} \circ \hat{R}_\infty^{-i}
\end{eqnarray*}
Homeomorphisms $\psi_{n}\circ \hat{R}_\infty^{i}$ are known (and fixed) in step $n$, $h_{n+1}^{-1}|_{V}$ is defined as the first element of the inductive construction in step $n+1$, however homeomorphisms $\hat{\zeta}_{\chi}$ can be defined recursively, depending on (linear) order of connected components of $\E_{(n+1)}^n$
within components of $\E_{(n+1)}^{n+1}$, and such order is well defined by the fact that $\mathcal G(\E_{(n+1)}^{n+1})$ has no cycle. In particular, the form of $\zeta_{\chi'}$
can be decided after $\zeta_{\chi}$ was chosen. To see this, observe that

\begin{eqnarray*}
 (h_{n+1}^{-1}|_{V})^{-1}\circ \hat{R}_\infty^{-i-q} \circ \psi_{n}\circ G_n^q(x,c)&=& \hat{\zeta}_{\chi'} \circ \hat{\zeta}_\chi^{-1} \circ (h_{n+1}^{-1}|_{V})^{-1} \circ \hat{R}_\infty^{-i}
\circ\psi_{n}(x,c)
\end{eqnarray*}
so we have to specify the image of $\hat{\zeta}_{\chi'}$ between Cantor sets $\hat{\zeta}_\chi^{-1} \circ (h_{n+1}^{-1}|_{V})^{-1} \circ \hat{R}_\infty^{-i}
\circ\psi_{n}(K_X\times C)$ and $h_{n+1}|_{V}\circ \hat{R}_\infty^{-i-q} \circ \psi_{n}\circ G_n^q(K_X\times C)$ provided by the above formula, and then extend it to the set $V$, while the construction is done at the level of the cube $V_m$ for $\zeta_{\chi}$ and $\zeta_{\chi'}$. Note that the acyclic structure of the graph ensures that for the components $\chi$ and $\chi'$ the iteration $q$ is uniquely determined.
This way we ensure that also $\mathbf{B_6}(b)$ holds.

The conditions $\mathbf{C_1}-\mathbf{C_6}$ are shown similarly as in the proof of Proposition 8.4. from \cite{Cro}, the only difference is that we are working on levels $S_{k_n}$ but arguments remain valid also in this setting.

Note that $\E_{(n+1)}^{n+1}$ is a subset of $\rs(\mathcal{D}_{n+1})$ so if $H_{n+1}$ is identity outside $\rs(\E_{(n+1)}^{n+1})$ then its composition with $\psi_n$
will not violate condition $\mathbf{B_7}$ and other conditions proven under the assumption that $\mathbf{C_1}$ holds.  Additionally observe that conditions $\mathbf{C}_7$, $\mathbf{C}_8$ are on sets $P^{(n)}_i$ which are cylinders over cubes in $S_n$.
At the moment, we have not defined sets in $\E_{(n+2)}^0$ which are necessary ingredient in conditions $\mathbf{C}_7$, $\mathbf{C}_8$.
However, we already know how they will be constructed. Namely, they will come from neighborhoods of the set
$Z_{k_{n+1}}$ in condition \eqref{con:paZ} together with the deformation as in \eqref{CHI}, and this does not involve the definition of $\Psi_{n+1}$. We also have sets $O_{n+1}\subset \mathbb{D}_{\infty}$ and so the sets $O'_{n+1}=h_{n+1}\circ \Psi_n(O_{n+1})$
are known. Therefore, to repeat the argument in \cite{Cro} we do the following. We fix small neighborhood $U_1$
of set $K_X\times C$, where $X\in \E_{(n+1)}^{0}$ is a properly chosen element (see the discussion before Lemma~9.6 in  \cite{Cro}).
But $\pi_{k_{n+1}}(h_{n+1}\circ \Psi_n(K_X\times C))\subset X\cap A_{k_{n+1}}$  so we can choose as $W_1$
the pullback by $h_{n+1}\circ \Psi_n$ of a sufficiently small neighborhood of $A_{k_{n+1}}$, which enables us to apply Lemma~\ref{lem:diff}.\\
Therefore, with respect to these cubes obtained by the pullback of $h_{n+1}\circ \Psi_n$, we can repeat construction from Section~9 of \cite{Cro}, and then extend obtained homeomorphism of $S_{k_{n+1}}\to S_{k_{n+1}}$ to 
a homeomorphism of $\spaceX$. We leave the details to the reader as only slight changes comparing to \cite{Cro} are needed.
On the other hand the whole proof in \cite{Cro} is long and involved, therefore we decided not to reiterate it here.
 The main difference is that instead of one sequence $P_i$ as in \cite{Cro}, we have different sequence $P_i^{(n)}$ for each level $S_{k_{n}}$ of our construction. So formally, in each step we work with different family of bins. Still they are cubes on respective levels, so that arguments from \cite[pp. 295-296]{Cro} can be repeated.
\end{proof}

\subsection{Crucial properties.}\label{Q_is_pseudo}

Let $\Mm_n$ be a sequence of chains such that $\rs(\Mm_{n+1})\subset \rs(\Mm_n)$, $\Ps_\infty=\cap_n \rs(\Mm_n)$ and $\supp h_n\subset \rs(\Mm_n)$. 
Let $\psi$ denote the semi-conjugacy map $\psi\colon \spaceX\to\spaceX$ obtained by the construction  in Lemma~\ref{lem:main}
and denote $\tilde{\mathbb{P}}:=\psi^{-1}(\Ps_\infty)$.

\begin{lem}\label{lem:arc-like} 
The continuum $\tilde{\mathbb{P}}$ is a pseudo-arc.
\end{lem}

\begin{proof}
Recall that by $\mathbf{B}_7$ there is a nested sequence of chains $\{\Mm_{n}\}_{n=0}^\infty$, such that $\bigcap_{n=0}^\infty \rs(\Mm_n)=\Ps_\infty$,
\begin{align}\label{B0}
\mesh \big( \psi_{n}^{-1} (\Mm_{n}) \big) & < 2^{-n}
\end{align}
and 
\begin{equation}\label{B9}
\psi_{n}^{-1} (\Mm_{n})\textrm{ is crooked inside }\psi_{n-1}^{-1} (\Mm_{n-1}).
\end{equation}
Set
\begin{equation}
\mathbb{P}'=\bigcap_{n=0}^{\infty} \psi_n^{-1}(\rs ( \Mm_{n})).
\end{equation}
We are going to show that
$\tilde{\mathbb{P}}=\mathbb{P}'$.
Recall that for every $k<n$ we have
\begin{equation}\label{B8}
\psi_n^{-1}( \overline { \rs ( \Mm_{n}) })
\subset   \psi_k^{-1}( \rs ( \Mm_{k}))
\end{equation}
so $\mathbb{P}'$ is a continuum and by \eqref{B0} and \eqref{B9} $\mathbb{P}'$ is a pseudo-arc.

Choose any $x \in \mathbb{P}'$. 
By the choice, 
$\psi_n (x) \in \rs (\Mm_{n})$ for all  $n\geq 0$. 
Since $\{\rs (\Mm_{n})\}_{n=0}^\infty$ is a nested sequence,
it follows that 
$\psi(x) \in \rs ( \Mm_n)$ for each $n\geq 0$. Thus, 
\begin{equation}
\psi(x) \in \bigcap_{n=0}^{\infty} \rs \big( \Mm_n \big)=\Ps_\infty. 
\end{equation}
Then $x \in \psi^{-1}\circ \psi (x) \subset \psi^{-1}(\Ps_\infty)=\tilde{\mathbb{P}}$ 
and so $\mathbb{P}' \subseteq \tilde{\mathbb{P}}$.

Next, by contradiction, suppose that $\mathbb{P}'$ is a proper subcontinuum of $\tilde{\mathbb{P}}$. Choose $y\in \tilde{\mathbb{P}}\setminus \mathbb{P}'$. Then there exists an $n$ such that 
\begin{equation}\label{belong}
y\notin\psi_n^{-1}( \rs ( \Mm_{n}))
\end{equation} 
Then $\psi_n(y)\notin  \rs ( \Mm_{n})$ and so $\psi_{n+1}(y)\notin  \rs ( \Mm_{n+1})$. But $h_{n+1}$ is identity outside of $\rs ( \Mm_{n})$ by $\mathbf{B}_7(c)$ and  $\mathbf{B}_1$, so
$\psi_{n+1}(y)=h_{n+1}(\psi_n(y))=\psi_n(y)$. Therefore $\psi_n(y)=\psi_m(y)$ for each $m>n$. It follows that $\psi(y)=\psi_n(y)$ and by (\ref{belong}) we get that $\psi(y)=\psi_n(y)\notin \Ps_\infty$. This contradiction shows that $\mathbb{P}'=\tilde{\mathbb{P}}$. 
\end{proof}

Recall that we modified conditions $\mathbf{C_{7,8}}$ compared to respective conditions in \cite{Cro}. We have to prove that this change does not influence statements about invariant measures in the extension.

Let $G$ be the map $\psi\colon \spaceX\to\spaceX$ obtained as limit of maps $G_n$ in virtue of Lemma~\ref{lem:main}.
Recall that the Cantor set $K$ was provided by Lemma~\ref{lem:choiceofKinA} and definition of $C$ together with definitions and identifications of $C$ and $K\times C$ are provided just after the statement of the lemma.
Let $\chi\colon K\times C \to K\times C$ be the first return map induced by $G$ and some ergodic invariant measure $\nu$ of $G$
such that $\nu(K\times C)>0$. By {\black the} Poincar\'e recurrence theorem {\black \cite[Theorem 1.4]{Walters}} $\chi$
is defined on a set $\tilde{K}\subset K\times C$ such that
$\nu(\tilde{K})=\nu(K\times C)$ and $\chi(\tilde{K})\subset \tilde{K}$. Therefore in what follows, we assume that $\chi$ is defined on a set $\tilde{K}$
containing all points that return infinitely many times to $K\times C$.
This way $\chi$ is defined $\nu$-almost everywhere on $K\times C$.

\begin{lem}\label{lem:omegalimitinKC}
	Assume hypotheses $\mathbf{C_{1,2,5,6,7,8}}$ are satisfied. If $x\in K\times C$ and $y=\chi(x)$ then $x\in \bigcap_{n}P^{(n)}_{i_n}$ and there are indices $j_n\geq i_n$, with $j_n>i_n$ for some $n$, such that
	$y\in \bigcap_{n}P^{(n)}_{j_n}$.
\end{lem}
\begin{proof}
Let $q$ be such that $G^q(x)=y$. Now all works analogously as in Lemma 9.3 in \cite{Cro} with the only difference that 
first, using $\mathbf{C_{8}}$ we show for each $n$ that if $x\in P^{(n)}_{i_n}$ then there is $j_n>i_n$ such that $y\in P^{(n)}_{j_n}$.
Then the result follows by the fact that for each $n,i$ there is $j$ such that $P^{(n)}_{i}\subset P^{(n-1)}_{j}$
and $j$ is uniquely defined.
\end{proof}

\begin{lem}\label{lem:omegalimitinKCcomplete}
Assume hypotheses $\mathbf{C_{1,2,5,6,7,8}}$ are satisfied and $x\in K\times C$ is such that $G^n(x)\in \psi^{-1}(K)$ for infinitely many $n$. Then $\omega_{\chi}(x)\subset K\times C$.
\end{lem}
\begin{proof}
The proof is the same as Proposition 9.2 in \cite{Cro} with the following modification. Working on coordinate spaces $S_n$, by $\mathbf{C_{8}}$ we obtain that $\omega_{\chi}(x)\subset P^{(n)}=\limsup_{i\to \infty} P_i^{(n)}$.
Then (\ref{eq:(d)})  gives $\omega_{\chi}(x)\subset K\times C$.
\end{proof}

With the above two lemmas proved we now prove a version of Corollary 9.4 in \cite{Cro} which is our final goal.

\begin{corollary}\label{cor:9.4}
Assume that hypotheses $\mathbf{C_{1,2,5,6,7,8}}$ are satisfied and let $\nu$ be a $G$-invariant probability measure on $\tilde{\mathbb{P}}$
which is
an extension of the Haar measure $\mu$ on $C$. Then 
$$\nu \big( \bigcup_{j\in \Z} G^j(\psi^{-1}(K)\setminus (K\times C)\big)=0.$$
In particular, $G$ is universally isomorphic to the disjoint union
$$\big( \bigcup_{j\in \Z} \mainmap^j(K)\times C, \mainmap \times\sigma \big)\bigsqcup \big(\tilde{\mathbb{P}}\setminus  \bigcup_{j\in \Z} \mainmap^j(K), \mainmap\big).$$
\end{corollary}

\begin{proof}
First note that the restriction of $\nu$ to $\psi^{-1}(K)$ is $\chi$-invariant. Therefore, using Lemma~\ref{lem:omegalimitinKCcomplete} and Poincar\'e recurrence theorem we obtain that $\nu (\psi^{-1}(K)\setminus (K\times C))=0$, which shows the first part of the corollary.

To prove the second part, one defines a bi-measurable map 
$$\Theta: \big(\tilde{\mathbb{P}}\setminus  \bigcup_{j\in \Z} G^j(\psi^{-1}(K))\big) \bigsqcup  \big(\bigcup_{j\in \Z} G^j(K\times C)\big)\to \big(\tilde{\mathbb{P}}\setminus  \bigcup_{j\in \Z} \mainmap^j(\psi^{-1}(K))\big) \bigsqcup  \big(\bigcup_{j\in \Z} \mainmap^j(K)\times C\big)$$
given by $\psi$ on the set $\tilde{\mathbb{P}}\setminus  \cup_{j\in \Z} G^j(\psi^{-1}(K)))$ and by $\mathbf{C_6}$
on the set $\cup_{j\in \Z} G^j(K\times C))$ (cf. Proposition 8.2 in \cite{Cro}). By the first part of the present proof, for the set
$$\tilde{\mathbb{P}}_0:= \big(\psi^{-1}(\Lambda)\setminus  \bigcup_{j\in \Z} G^j(\psi^{-1}(K))\big)$$ it holds that $\nu(\tilde{\mathbb{P}}_0)=0$ for every invariant probability measure $\nu$. By Lemma~\ref{lem:hnconverge} and Lemma~\ref{lem:fibres} we have that $\psi$ is one-to-one on the set $(\tilde{\mathbb{P}}\setminus \Lambda)\cup (\Lambda\setminus  \cup_{j\in \Z} G^j(\psi^{-1}(K)))$, so $\Theta$ is one-to-one as well. Using Lemma~\ref{lem:hnconverge} and $\mathbf{C_6}$ 
on the set $\cup_{j\in \Z} G^j(K\times C))$
we see that the map $\Theta$ is a conjugacy which shows the second part of the corollary.
\end{proof}

\begin{proof}[Proof of Theorem~\ref{thm:main}]
Fix any $r\in [0,\infty]$. The case of $r=0,\infty$ is already known (see \cite{Mouron}), so we may assume that $r\in (0,\infty)$.
Next step can be done the same way as in \cite[p.255]{Cro} since $\mainmap|_\Lambda$
is {\em strictly ergodic} (i.e. minimal and uniquely ergodic) with zero topological entropy. Simply we take as $\sigma\colon C\to C$ a strictly ergodic model of Bernoulli system with entropy $r$ (e.g. by the application of celebrated Jewett-Krieger theorem {\black(see e.g. \cite{DGS})}).
Then we know that $\mainmap \times \sigma$ is uniquely ergodic, since zero entropy systems and K-systems are measure-theoretically disjoint (see \cite[p.255]{Cro} for more details; cf. \cite[Proposition~4.6]{Thouvenot}).
This ensures that if we apply the construction, then by Corollary~\ref{cor:9.4} any measure supported on $\big( \bigcup_{j\in \Z} \mainmap^j(K)\times C, \mainmap \times\sigma \big)$ is in fact the unique measure of $(\Lambda\times C, \mainmap \times \sigma)$.
Since this is the only invariant measure of $(\tilde{\mathbb{P}},G)$ not isomorphic to an invariant measure of $(\Ps_\infty,\mainmap)$, we obtain by the Variational Principle that
$\htop(G)=\htop(\sigma)=r$ completing the proof.
\end{proof}

\appendix

\section[abc]{ (by George Kozlowski${}^{5}$)}\label{sec:appendix}
\stepcounter{footnote}
\footnotetext{Department of Mathematics and Statistics, Auburn University, AL 36849, USA}

This section contains a proof of the following theorem, which undoubtedly already exists in folklore.  Details are included here for the convenience of those who may not be familiar with the history of the crucial results stated below on which the proof is based.
\begin{theorem}\label{thm:A1} 
Let $X$,$Y$ be two $d$-dimensional disks and $\alpha$ an orientation preserving homeomorphism from $X$ onto $Y$.  
For  $i=1,\dots, n$ let 
$B_i$ be pairwise disjoint locally flat $d$-dimensional  disks in $\Int(X)$, let $C_i$ be pairwise disjoint locally flat $d$-dimensional  disks in $\Int(Y)$, 
and let $\beta_i$ be orientation preserving homeomorphisms from $B_i$ onto $C_i$. 
Then there is a homeomorphism $\gamma$ from $X$ onto $Y$ which coincides with each $\beta_i$ on $B_i$ and with $\alpha$ on the boundary of $X$.
\end{theorem}

References for the basic material used here are papers \cite{Brown-GST} and \cite{Brown-LFE} of Morton Brown.
It follows from results of those papers that a locally flat embedding of a disk is the  same as a tame embedding of that disk in the sense that the latter phrase (which e.g.~is found  in \cite{Fisher}) is used elsewhere in this paper.
His collaboration
   \cite{BrownGluck-SS} (summarized  in  \cite{BrownGluck-BAMS}) with Herman Gluck forged the link 
  which stimulated the interest in 
the deep results on which the theorem above depends.

The statements involve two definitions found in \cite{BrownGluck-SS} and also \cite{BrownGluck-BAMS}. Consider a homeomorphism $h$ of Euclidean space $R^d$ (or the unit sphere $S^d$ in $R^{d+1}$) onto itself.  
If there is a nonempty open set $U$ which $h$ maps identically, then $h$ is said to be \emph{somewhere the identity}. If $h$ is the composition of a finite number of homeomorphisms each of which is somewhere the identity, then $h$  is said to be \emph{stable}.
\begin{theorem}[\textbf{Stable Homeomorphism}]  Every orientation preserving homeomorphism of $R^d$ (or $S^d$) onto itself is stable.
\end{theorem}
In dimension 2 (i.e.~$d=2$) the result is said to be classical, and in dimension 3 it is attributed to Moise. 
In dimension 4 it is due to Quinn \cite{Quinn}. In dimensions 5 and greater it is due to Kirby \cite{Kirby}, although for dimension 5 Kirby uses a result of C.~T.~C.~Wall which appeared in \cite{Wall}.
A useful succinct discussion of these matters is found in Section 8 of \cite{Edwards}.

With the Stable Homeomorphism Theorem as a basis the crucial results then follow from \cite{BrownGluck-SS}. 
\begin{theorem}[\textbf{Annulus}] Let $f,g \colon S^{d-1} \to R^d$ be disjoint, locally flat imbeddings with $f(S^{d-1})$ inside the bounded component of $R^d - g(S^{d-1})$. Then
the closed region $A$ bounded by $f(S^{d-1})$ and $g(S^{d-1})$ is homeomorphic to $S^{d-1} \times [0,1]$.
\end{theorem}

\begin{theorem}[\textbf{Isotopy}] Every orientation preserving homeomorphism of $R^d$ (or $S^{d}$) onto itself is isotopic to the identity.
\end{theorem}

Throughout the remainder of this section  certain assumptions are in force: 
(1) all homeomorphisms  preserve orientation, (2) all embeddings and embedded objects are locally flat, and (3) everything occurs in some Euclidean space of fixed dimension with norm $| x |$ denoting the distance from the point $x$ to the origin. It will be convenient to work with disks $D[r]$, spheres $S[r]$, and annuli $A[r,s]$, where
\[ D[r]= \big  \{ x  \colon | x | \le r  \big \}, \ S[r]=  \big \{ x  \colon | x | = r  \big \}, A[r,s]= \big \{ x \colon r \le | x | \le s \big  \}.\]
A homeomorphic image of one of these objects will be referred to as a disk, sphere, or annulus as appropriate.  
The homeomorphisms $S[1] \approx S[r]$ and $S[1] \times [r,s] \approx A[r,s]$ defined  by $x \mapsto rx$ and $(x,t) \mapsto tx$ ($x \in S[1], r \le t \le s$) occur implicitly below. 
The boundary of the disk  $D$ is denoted $\partial D$.

\begin{lem} Let $B$, $C$, and $D$ be disks with $B \cup C \subset \Int (D)$ and suppose there is given a homeomorphism $\beta\colon B \approx C$.  Then there is a homeomorphism $\gamma \colon D \approx D$ such that 
\[ \gamma (x) = \begin{cases} x \text{ for all } x \in \partial D \text{ and}\\
\beta (x) \text{ for all } x \in B.
\end{cases} \]
\end{lem}

\begin{proof}  By the Annulus Theorem there are homeomorphisms $\varphi \colon   D - \Int (B) \approx A[1,2] $ and  $\psi \colon   D - \Int (C) \approx A[1,2] $, and these can be chosen so 
 that $\varphi (\partial D) = \psi (\partial D)=S[2]$.
From  the Isotopy Theorem the homeomorphism of $S[2]$ onto itself defined by $\psi  \circ \varphi^{-1}$, considered as a map of $S[1]$  onto itself,  and the
 homeomorphism of $S[1]$ onto itself defined by $\psi \circ \beta \circ \varphi^{-1}$ are isotopic to the identity map and  hence are isotopic.  Thus there is a 
 homeomorphism $\theta \colon A[1,2] \approx A[1,2]$ which is  $\psi \circ \beta \circ \varphi^{-1}$ on $S[1]$ and  $\psi  \circ \varphi^{-1}$ on $S[2]$. The desired homeomorphism is defined by
\[ \gamma (x) = \begin{cases} \big ( \psi^{-1} \circ \theta \circ \varphi  \big ) (x) \text{ for all } x \in D - \Int (B) \text{ and}\\
\beta (x) \text{ for all } x \in B,
\end{cases} \]
which completes the proof of the lemma.
\end{proof}

\begin{proof} [Proof of Theorem~\ref{thm:A1}]  
The indices $i$ and $j$ will satisfy $1 \le i, j \le n$.  Choose real numbers \[0<r<s<t<u<1.\]
Let $\varphi \colon X \approx D[2]$ and $\psi \colon Y \approx D[2]$ be homeomorphisms satisfying $\varphi (B_i) \subset A[r,s]$ and  $\psi (C_j) \subset A[ t,u]$  for all $i, j$, which can be achieved by radial homeomorphisms of $D[2]$.  One then finds disjoint disks $E_1, \dots, E_n$ in $\Int (D[1])$ such that $\varphi (B_i) \cup \psi (C_i) \subset \Int (E_i)$.  By the lemma there are homeomorphisms $\varepsilon_i \colon E_i \approx E_i$ such that  
\[ \varepsilon_i (x) = \begin{cases} x \text{ for all } x \in \partial E_i \text{ and}\\
 \big ( \psi \circ \beta_i  \circ \varphi^{-1}  \big ) (x) \text{ for all } x \in \varphi B_i.
\end{cases} \]
Define $\varepsilon \colon D[1] \approx D[1]$ by $\varepsilon (x) = \varepsilon_i (x)$ for all $x$ in $E_i$ and 
$\varepsilon (x) = x$  for all $x \in D[1] - \cup_{i} \,  E_i$.

The homeomorphism $\psi \circ \alpha \circ \varphi^{-1}$ defines a homeomorphism of $S[2]$ onto itself and hence is isotopic to the identity.  Thus there is a homeomorphism $\delta \colon A[1,2] \approx A[1,2]$ which agrees with this map on $S[2]$ and which is the identity on $S[1]$.
Define homeomorphisms (1) $\theta \colon D[2] \approx D[2]$  by extending $\delta$ over $D[1]$ by $\varepsilon$ and  (2) 
$\gamma \colon X \approx Y$ by $\gamma =   \psi^{-1} \circ \theta \circ \varphi$, which completes the proof of the theorem.
\end{proof}

\section*{Acknowledgements}
We are grateful to W. Lewis for his comments on \cite{LewisHomeoGroup} and
 \cite{LewisPeriodic}, and to X.-C. Liu who made us familiar with \cite{Cro}. We are also very grateful to H. Bruin, S. Crovisier, and L. Hoehn for helpful comments on the final version of the paper. 

J. Boro\'nski was supported in part by the National Science Centre, Poland (NCN), grant no. 2019/34/E/ST1/00237.
J. \v{C}in\v{c} was supported by the Austrian Science Fund (FWF) Schr\"odinger Fellowship stand-alone project J 4276-N35 and Slovenian Research Agency ARIS grant J1-4632.
 P. Oprocha was supported by
National Science Centre, Poland (NCN), grant no. 2019/35/B/ST1/02239.

\nocite{*}
\bibliographystyle{alpha}

\end{document}